\documentclass[a4paper]{article}

\usepackage{amsmath}
\usepackage{amssymb}
\usepackage{amsthm}
\usepackage{amscd}
\usepackage{amsfonts}
\usepackage{pgf}
\usepackage{tikz}
\usepackage{graphicx}
\usepackage{tikz-3dplot}
\usepackage{hyperref}

\usetikzlibrary{decorations.markings}
\tikzset{middlearrow/.style={
        decoration={markings,
            mark= at position 0.5 with {\arrow{#1}} ,
        },
        postaction={decorate}
    }
}
\usetikzlibrary{calc,cd}

\newtheorem{theorem}{Theorem}
\newtheorem{definition}[theorem]{Definition}
\newtheorem{proposition}[theorem]{Proposition}

\newtheorem{lemma}[theorem]{Lemma}
\newtheorem{remark}[theorem]{Remark}
\newtheorem{corollary}[theorem]{Corollary}
\newtheorem{example}[theorem]{Example}
\newtheorem{conjecture}[theorem]{Conjecture}

\newcommand\torus{\mathbb T_2}
\newcommand\rec{\mathcal R}
\newcommand{\E}{{\mathbb E}}
\newcommand{\supp}{\operatorname{supp}}
\newcommand{\F}{\mathcal F}
\newcommand\core{\mathcal A}
\newcommand{\N}{{\mathbf N}}
\newcommand{\Z}{{\mathbf Z}}
\renewcommand{\P}{{\mathbf P}}

\newcommand{\R}{{\mathbf R}}
\newcommand{\ind}{{\mathbf 1}}
\newcommand{\D}{{\mathbb D}}

\newcommand{\Card}{{\operatorname{Card}}}

\newcommand{\Span}{\text{span}}
\renewcommand{\dim}{{\operatorname{dim}}}
\renewcommand{\deg}{{\operatorname{deg}}}
\newcommand{\im}{\operatorname{im}}

\newcommand{\s}{\mathcal S}

\renewcommand\path{\mathfrak{P}}
\renewcommand\L{\mathfrak L}
\newcommand{\chain}{\sigma}
\newcommand{\cochain}{f}

\newcommand\cyl{\mathcal D}
\newcommand\car{\mathbf 1}
\newcommand\ch{\mathcal C} 
\newcommand\p{\partial}
\renewcommand\d{\p^*}
\newcommand\KK{\mathbf R}
\newcommand{\generator}{\mathcal{A}} 
\newcommand{\graphLaplacian}{L} 
\newcommand\up{\uparrow}
\newcommand\down{\downarrow}
\newcommand\Id{\operatorname{Id}}
\newcommand\dif{\text{ d}}
\newcommand{\dom}{\operatorname{dom}}

\def\Cf1{{\mathfrak C^1}}
\def\Cf{{\mathfrak C}}
\def\Co{{\mathcal C}}
\def\dd{{\mathrm d}}

\title{Random walks on simplicial complexes}
\author{Thomas Bonis\thanks{LAMA, Univ Gustave Eiffel, Univ Paris Est Creteil, CNRS, F-77454 Marne-la-Vall\'ee, France. 
\texttt{thomas.bonis@univ-eiffel.fr}}, \quad
Laurent Decreusefond\thanks{LTCI, Telecom Paris, I.P. Paris,  France. \texttt{laurent.decreusefond@telecom-paris.fr}},\quad Viet Chi Tran\thanks{LAMA, Univ Gustave Eiffel, Univ Paris Est Creteil, CNRS, F-77454 Marne-la-Vall\'ee, France; IRL 3457, CRM-CNRS, Université de Montréal, Canada. 
\texttt{chi.tran@univ-eiffel.fr}},\quad Zhihan Iris Zhang
\thanks{LTCI, Telecom Paris, I.P. Paris,  France. \texttt{zhihan.zhang@telecom-paris.fr}}}
\date{\today}

\begin{document}

\maketitle{}

\begin{abstract}
  The notion of Laplacian of a graph can be generalized to simplicial complexes
  and hypergraphs, and contains information on the topology of these
  structures. Even for a graph, the consideration of associated simplicial
  complexes is interesting to understand its shape. Whereas the Laplacian of a
  graph has a simple probabilistic interpretation as the generator of a
  continuous time Markov chain on the graph, things are not so direct when
  considering simplicial complexes. We define here new Markov chains on
  simplicial complexes. For a given order~$k$, the state space is
  the set of $k$-cycles that are chains of $k$-simplexes with null boundary. This new framework is a natural
  generalization of the canonical Markov chains on graphs. We show that the
  generator of our Markov chain is the upper Laplacian defined in the
  context of algebraic topology for discrete structure. We establish several key
  properties of this new process: in particular, when the number of vertices is finite, the Markov chain is positive recurrent. 
  We study the diffusive limits when the simplicial complexes
  under scrutiny are a sequence of ever refining triangulations of the flat
  torus. Using the analogy between singular and Hodge homologies, we express this limit as valued in the set of currents. The proof of tightness and the identification of the limiting martingale problem make use of the flat norm and carefully controls of the error terms in the convergence of the generator. Uniqueness of the solution to the martingale problem is left open. An application to hole detection is carried.
\end{abstract}

\noindent \textbf{keywords}: stochastic geometry; random walk; simplicial complex; algebraic topology; Laplacian of a simplicial complex; homology; limit theorem; Hodge diffusion.\\
\noindent \textbf{Subjclass}: 60D05.\\
\noindent \textbf{Acknowledgements}: L.D. would like to thanks A. Vergne, M. Glisse, O. Devillers and P. Chassaing for fruitful discussions. V.C.T. thanks P. Melotti for inspiring simulations. This work benefited from the GdR GeoSto 3477. T.B., L.D. and Z.Z. acknowledge support from ANR ASPAG (grant ANR-17-CE40-0017 of the French National Research Agency). V.C.T. is supported by Labex B\'ezout (ANR-10-LABX-58), ANR Econet (ANR-18-CE02-0010) and by the European Union (ERC-AdG SINGER-101054787). 

\section{Introduction}
\label{sec:intro}

Exploring and understanding complex structures such as graphs is a difficult and rich problem that has motivated an abundant literature in the last years. Random walks are one of the many tools used to study the connectivity of a graph. For instance, the PageRank algorithm \cite{Pagebrinmotwaniwinograd} uses invariant measures of random walks to highlight central nodes of graphs. Random walks on graphs are also used to provide distances between nodes by considering the expected time required for the walk to travel between nodes. This metric is called the commute distance and has been used in a large number of applications such as graph embedding \cite{commutedistance_graphreconstruction}, semi-supervised learning \cite{commutedistance_semisupervised}, clustering \cite{commutedistance_clustering} and many more, see for instance the introduction of \cite{commutedistance_luxburg} for a more complete list. Finally, diffusions of random walks on graphs is at the center of the diffusion maps graph embedding approach \cite{diffusionmaps}. \\

The reason random walks on graphs are so tightly linked to the connectivity structure of the graph can be found in their generators. Consider a finite non-oriented graph $G=(V,E)$ consisting of the (finite) sets of vertices $V$ and edges $E$ determining the pairs of vertices that are connected. The adjacency matrix $A$ of $G$ is defined as the matrix whose entry at line $u$ and column $v$ is 1 if and only if $(u,v)$ is an edge of $G$, which we will denote by $u\sim v$. For any $u\in V$ and for any function $f$ from $V$ to $\R$, the generator of the random walk is
\begin{equation}\label{eq:defLaplacianmarche}
\generator f(u)=\sum_{v\sim u} \big(f(v)-f(u)\big)=-\big(D-A\big) f (u),
\end{equation}where $D$ is the diagonal matrix containing the degrees of the vertices. Thus,
the generator of the random walk is the opposite of the Laplacian of the graph, $\graphLaplacian = D-A$. It is well-known that this Laplacian contains information regarding the connectivity structures of the graph. For instance, the dimension of its kernel is equal to the number of connected components of the graph while, more generally, small eigenvalues indicate almost disconnected components \cite{luxburgbelkinbousquet}.  \\
But connectivity is only a fraction of the topological information contained in complex structures. For instance, let us consider a circular-like graph presented in Figure~\ref{fig:examplecircle}(a). Such a topological structure is not well described through connectivity. More generally, graphs are not suited to deal with such structures since adding a single node to the graph creates another artifact circle which is not an actual feature of the data, see Figure~\ref{fig:examplecircle}(b). The correct structures one should use to deal with higher-order topology are simplicial complexes. The definition of simplicial complexes is recalled later, but a natural simplicial complex associated with a graph is its Rips-Vietoris simplicial complex \cite{Jean-Daniel-Boissonnat:2018aa} obtained by adding to the pair $(V,E)$ the set $\s_2$ of all triangles whose edges belong to $E$, the set of all tetrahedrons $\s_3$ whose triangles belong to $\s_2$ etc. For instance, adding a triangle in the previous example, we recover a circular structure (see Figure \ref{fig:examplecircle}(c)). This notion of circular structure is formalized by the concepts of homology classes and Betti number which we present in Section~\ref{sec:simplicial}. \\

We are interested in generalizing the notion of random walks on graphs to random walk on simplicial complexes, hoping they can be used to derive new algorithms for topological data analysis by generalizing algorithms such as PageRank. Furthermore, providing a probabilistic interpretation of topological properties of simplicial complexes is important since there is a growing literature on the subject in recent years. While the relations between graphs and simplicial complexes are considered in \cite{devriendtvanmieghem}, papers more focused on the simplicial complexes themselves include for example studies of the Betti numbers and volume-like computation for random clique complexes built over the Erd\"os-R\'enyi graphs \cite{kahle,kahlemeckes,owadasamorodnitskythoppe} or \v{C}ech and Vietoris-Rips complexes built over stationary point processes \cite{akinwandereitzner,decreusefondferrazrandriamvergne,yogeshwaranadler}, or computation of convex hulls of simplicial complexes \cite{grotekabluchkothale}. 
\par It is already known that the graph Laplacian $D-A$ is a specific instance of the more general combinatorial Laplacian, introduced by Eckmann \cite{eckmann}. In a similar way that the graph Laplacian contains information regarding the connectivity of the graph, these combinatorial Laplacians describe the structure of the homology groups of the simplicial complex and are related to higher order Betti numbers. Since the generator of random walks on nodes of graphs is equal to the opposite graph Laplacian, it was proposed in \cite{kaufmanoppenheim,parzanchevskirosenthal,rosenthal,schaubbensonhornlippnerjadbabaie} to define random walks on simplicial complexes as random walks with generators equal to the opposite of the combinatorial Laplacians. As the combinatorial Laplacian is defined as a sum of two operators, called up-Laplacian and down-Laplacian, such an approach leads to defining two different random walks from which it is not clear how to generalize graph analysis algorithms. For instance, if a combinatorial Laplacian is associated to two different random walks and thus to two different invariant measures, which one should be preferred to obtain an equivalent of PageRank for simplicial complexes?\\

In this paper, we propose to define a random walk on a simplicial complex in a totally different way. More precisely, we consider a random walk on the space of cycles of the simplicial complex whose transitions are given by the very definition of homology groups which, incidentally, has the opposite of the combinatorial up-Laplacian as generator. 
In particular, similarly to how a random walk on a graph cannot leave a given connected component, which is a homology class of dimension $0$, our random walk is bound to stay in the homology class of its initial state. \\

\par Let us detail our random walk in the case of a simplicial complex of order 2 (a general and precise description is given in the paper). Recall that a simplicial complex $\mathbf{C}$ of dimension 2 is a collection of vertices $V$, of edges $E$ and of triangles $T$ such that the edges of the triangles belong to $E$ and the vertices of the edges belong to $V$. Define the oriented edge from $u$ to $v$ as $[u,v]$.  
The cycles are defined as chains, i.e. elements of the vector space spanned by $E^+$ the set of edges with positive orientation:
\[\mathcal{C}_1=\Big\{\sum_{[u,v]\in E^+} \lambda_{[u,v]} [u,v],\ \mbox{ with }\forall [u,v]\in E^+,\ \lambda_{[u,v]} \in \mathbb{R} \Big\},\]
with the constraint that they have no border. Defining, the 
boundary map $\partial$ for edges by $\partial [u,v]=v-u$, and for triangles by $\partial [u,v,w]=[v,w]-[u,w]+[u,v]$, a chain $\sigma\in \mathcal{C}_1$ is a cycle if $\partial \sigma=0$. 
Our random walk $(X_t)_{t\in \mathbb{R}_+}$ is a continuous time Markov chain whose state space consists of oriented cycles. Given its current state $\sigma$, we consider all the triangles that are adjacent to $\sigma$ (\textit{i.e.} that share at least an edge with the cycle). The jump rate is the number of these triangles, weighted by the number of their edges common to $\sigma$. When there is a jump, say at time $t$, we chose randomly one of these triangles, say $\tau$, with a probability proportional to the number of their edges common with $\sigma$ and the Markov chain jumps from $X_{t_-}=\sigma$ to $X_t=\sigma-\partial \tau$. Heuristically this deletes the common edges and replaces them with the other edges of the triangle. For example in Fig. \ref{fig:examplecircle}(c), starting from the circle, we delete the edge adjacent to the triangle and replace it with the two other edges. The state thus remains a cycle.
\par More precisely the generator of this random walk is:
\begin{equation}\label{generator:dim1}
\generator f(\sigma)=\sum_{\tau \in \s_2 }\big(f(\sigma -\partial \tau)-f(\sigma)\big) \langle \partial \tau,\sigma\rangle^+,
\end{equation}where $x^+$ denotes the positive part of $x$ and $\langle \partial \tau,\sigma\rangle $ corresponds to the number of edges that $\tau$ and $\sigma$ have in common (with a sign corresponding to the orientation). When $f$ is a linear function, $\generator f(\sigma)=-L^\uparrow_1 f(\sigma)$ where $L^\uparrow_1$ is the up-Laplacian of order 1. The natural questions are whether the random walk is recurrent and whether we can derive diffusive scaling limits.\\

\par After recalling notions of homological and algebraic topology in Section \ref{sec:prel}, the general random walk (of order $k\geq 0$) is introduced and studied in Section \ref{sec:rw}. The case $k=0$ is the `usual' random walk, and $k=1$ corresponds to the generator described in \eqref{generator:dim1}. Concerning the recurrence or transience of this Markov chain, a difficulty is that the cycle can have loops: for $k=1$ and example \eqref{generator:dim1}, even if the number of vertices is finite, the state space $\mathcal{C}_1$ is infinite. 
It is shown in Theorem \ref{th:recurrence} that under simple conditions, the random walk is recurrent and admits a unique invariant measure.  For the usual random walk that jumps from vertices to neighboring vertices, the invariant measure puts weights on vertices $u\in V$ that are proportional to their degrees. Such result is not true any more for higher dimensions.
\par When the simplicial complex is embedded into a geometrical space (for instance, considering the vertices in $\R^d$), we discuss in Section \ref{sec:conv} the convergence of the rescaled random walk to a continuous process on a set of continuous paths. Because computations become complicated, we focus on the case of a triangulation of the torus for the sake of simplicity. We study the convergence of the random walk generator, and the limit involves the Hodge operator. The tightness is then obtained by an Aldous-Rebolledo criterion. We identify all the limiting values as solutions of the same martingale problem, but uniqueness of the solution is still left open. Finally, in Section \ref{sec:applications}, we use the invariant measure of our random walk to provide a generalization of the PageRank algorithm to the simplicial complexes which highlights the boundaries of topological structures.

\begin{figure}[!ht]
  \begin{center}
  \begin{tabular}{ccc}
\includegraphics[width=3cm,height=4cm,trim=0cm 1cm 0cm 1cm]{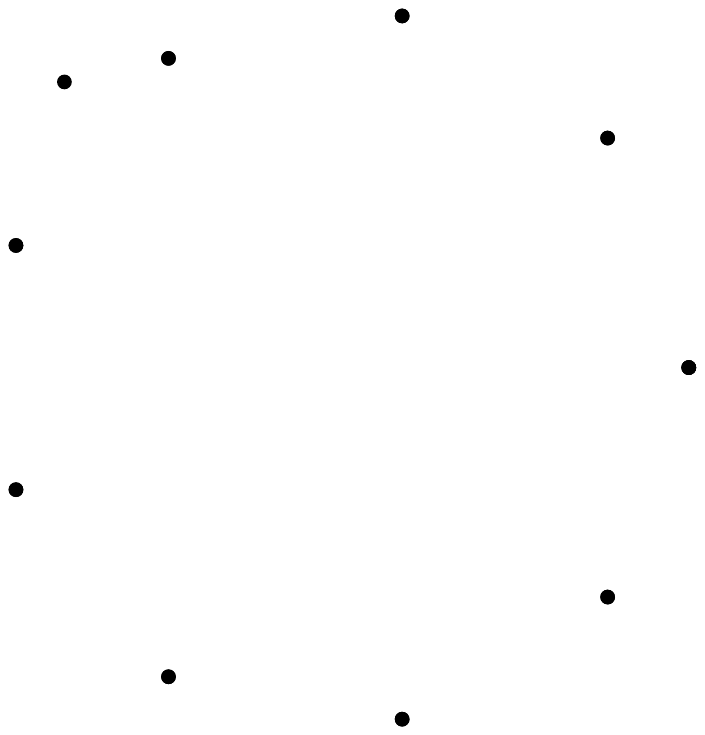} &
\includegraphics[width=3cm,height=4cm,trim=0cm 1cm 0cm 1cm]{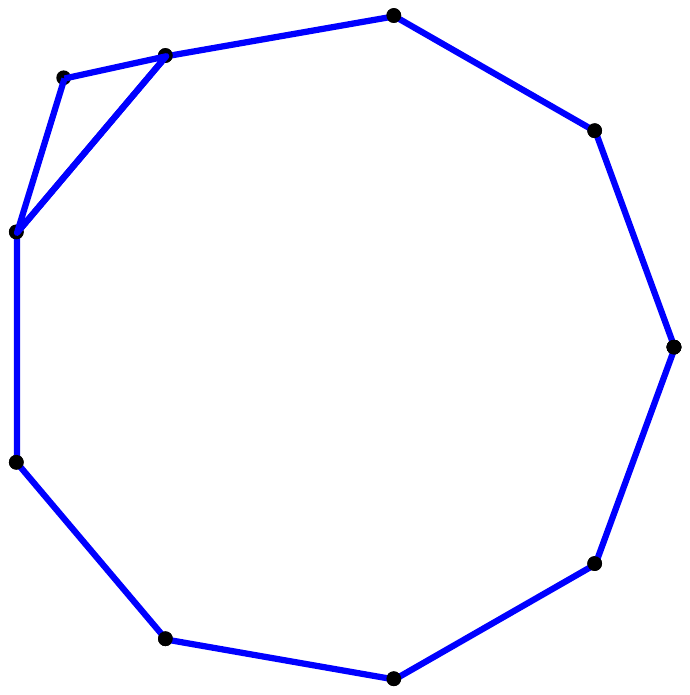}&
\includegraphics[width=3cm,height=4cm,trim=0cm 1cm 0cm 1cm]{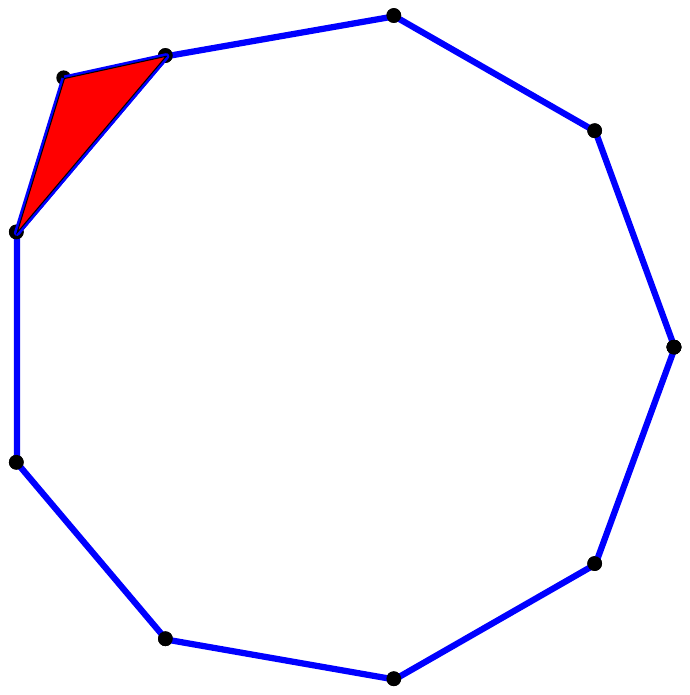}\\
  (a) & (b) & (c)
\end{tabular}
\caption{{\small \textit{Example where simplicial complexes are the correct structure to capture the topology of the data, and in particular detect a circular structure. (a): Data drawn from a circular structure. (b): Neighborhood graph structure, that reveals two different circular structures: a triangle and a circle. (c): A simplicial complex recovers the topology of the data.}}}
  \label{fig:examplecircle}
\end{center}
\end{figure}


\section{Preliminaries}
\label{sec:prel}

\subsection{Simplicial complexes}
\label{sec:simplicial}

As explained in the introduction, a natural generalization of graphs requires to consider simplicial complexes and adopt considerations from the field of homological and algebraic topology.
For further reading on algebraic topology, see~\cite{armstrong,hatcher,Munkres1984}. While graphs model binary relations, simplicial complexes
represent higher order relations.\\
Given a finite or denumerable set of vertices $V$, a $k$-simplex is an unordered subset $\{v_0,\,v_1,\,\dots,\, v_k\}$
where $v_i\in V$ and $v_i\not=v_j$ for all $i\not=j$. The faces of the $k$-simplex $\{v_0,\, v_1,\, \dots,\, v_k\}$ are defined as all the
$(k-1)$-simplexes of the form $\{v_0,\,\dots,\,v_{j-1},\,v_{j+1},\,\dots,\, v_k\}$ with $0\leq j\leq k$. The cofaces of a $k$-simplex $\tau$ are all the $(k+1)$-simplexes of which $\tau$ is a face.  A simplicial complex $\mathbf{C}$ is a collection of simplexes which is closed with respect to the inclusion of faces, i.e. if $\{v_0,\, v_1,\, \dots,\, v_k\}$ is a $k$-simplex of $\mathbf{C}$, then all its faces are in the set of $(k-1)$-simplexes of $\mathbf{C}$. We denote by $\s_{k}(\mathbf{C})$ the set of  $k$-simplexes of $\mathbf{C}$. In the sequel, when there is no ambiguity, we will drop the dependency on $\mathbf{C}$ and simply write $\s_k$. $\s_{k}$ can be viewed as a subset of $V^{k+1}$ which itself can be embedded in $ \N^{k+1}$. By convention, $\s_{0}=V$ consists of all the vertices. $\s_1$ of all the edges $\{v_0,v_1\}$ of $\mathbf{C}$ linking two vertices $v_0$ and $v_1\in V$, $v_0\not=v_1$. $\s_2$, $\s_3$ are the set of all triangles and tetrahedra of $\mathbf{C}$ etc. Then,
\[\mathbf{C}=\bigcup_{k\geq 0} \s_k.\]
One can define an orientation on simplexes by defining an order on vertices. The oriented $k$-simplexes are denoted with square brackets, with the convention that:
\begin{eqnarray*}
  [v_0,\, \dots,\, v_i,\, \dots,\, v_j,\, \dots,\, v_k]=-[\, v_0,\,
  \dots,\, v_j,\, \dots,\, v_i,\,\dots,\, v_k],
\end{eqnarray*}
for $0 \leq i,j \leq k$. Each simplex may thus appear in two species: positively
or negatively oriented. We denote by $\s_{k}^{+}$ (respectively $\s_{k}^{-}$)
the set of positively (respectively negatively) oriented $k$-simplexes containing all simplices $[v_0, \dots, v_k]$ such that $v_0 < v_1 < \dots < v_k$. 
For an oriented edge $[u,v]$ (going from $u$ to $v$), we call $u$ (resp. $v$) the ego (resp. alter) of the edge.
Also, when $[u,v]\in \s_1$, we will write $v\sim u$.

\begin{example}[Cech complex]
  For $V=\{v_{i},i=1,\cdots,n\}$ $n$ points in $\R^{d}$ (or in a metric space), and $R>0$,
  the \v{C}ech complex $\text{\v{C}ech}(V,R)$ of radius $R$ is defined as follows:
  $\s_{0}=\{v_{i},i=1,\cdots,n\}$ and, for any $k > 0$, $[v_{i_{0}},v_{i_{1}},\cdots,v_{i_{k}}]$
  belongs to $\s_{k}$ whenever
  \begin{equation*}
    \bigcap_{m=0}^{k} B(v_{i_{m}}, R)\neq \emptyset.
  \end{equation*}
  This complex has the property that its topological features, as defined below, reflect that of the geometric
  set
  $\cup_{i} B(v_{i},R)$.
\end{example}
\begin{example}[Rips-Vietoris complex]
  Unfortunately, the construction of the \v{C}ech complex is exponentially hard
  so it is very common to work with the Rips-Vietoris complex. The simplicial complex $\text{Rips}(V,R)$ has the same
  vertices $V$ and edges as the \v{C}ech complex, but for $k\ge 2$, $\{v_{i_{0}},\cdots,v_{i_{k}}\}$
  belongs to $\s_{k}$ whenever all the possible pairs made by choosing two
  points among $\{v_{i_{0}},\cdots,v_{i_{k}}\}$ belong to the set $\s_1$
  of edges of $\text{Rips}(V,R)$. Otherwise stated, the Rips-Vietoris is fully determined by the
  vertices and edges of the \v{C}ech complex. This graph is called the
  skeleton of the Rips-Vietoris. Though a priori coarser than the \v{C}ech
  complex, the Rips-Vietoris is not that far since (see
  \cite{Jean-Daniel-Boissonnat:2018aa})
  \begin{equation*}
    \text{Rips}(V,R)\subset \text{\v{C}ech}(V,R)\subset  \text{Rips}(V,2R).
  \end{equation*}
\end{example}

\subsection{Chains and co-chains}
For each integer~$k$, $\ch_k$ is the $\KK$-vector space spanned by the
set $\s_{k}^{+}$ of  $k$-simplexes of ${ V}$: an element $\sigma\in \ch_{k}$ is called a \textit{chain} or \textit{$k$-chain} and can be
uniquely written as
\begin{equation}\label{eq_preliminaries:2}
  \sigma=\sum_{\tau \in \s_{k}^{+}} \lambda_{\tau}(\sigma)\,\tau,
\end{equation}
where all but a finite number of $\{\lambda_{\tau}(\sigma)\in \R,\, \tau \in \s_{k}^{+}\}$ are
non null. We define the support of $\sigma$ to be:
\begin{equation}\label{eq_preliminaries:5}
  \supp \sigma=\{\tau\in \s_{k}^{+}, \, \lambda_{\tau}(\sigma) \neq 0\}.
\end{equation}
Equations \eqref{eq_preliminaries:2} and \eqref{eq_preliminaries:5} amount to define a scalar product that makes $\ch_k$ a Hilbert space and $\{\tau, \, \tau \in \s_{k}^{+}\}$
an orthonormal  basis of $\ch_{k}$.
We can thus consider the canonical norm on the set $\ch_k$ of $k$-chains by taking
\begin{equation*}
\|\sigma\|_{\ch_k}^2 = \|\sum_{\tau \in \s_{k}^{+}} \lambda_{\tau}(\sigma)\,\tau \|_{\ch_{k}}^{2}=\sum_{\tau \in \s_{k}^{+}} |\lambda_{\tau}(\sigma)|^{2}.
\end{equation*} \\

Let $\ch^{k}$ be the topological and
algebraic dual of $\ch_{k}$:
\[\ch^k=\Big\{f\ :\ \ch_k \rightarrow \R,\ \mbox{ linear and continuous }\Big\}\]
and any element of $\ch^k$ is called a \emph{cochain}.
Because $\ch_k$ is a Hilbert space, so is $\ch^k$.
Note that $\ch_{k}$ and $\ch^{k}$
are isomorphic and that any element
 $\tau\in \s_{k}$ can be viewed either as an element of $\ch_{k}$ or as an element of $\ch^{k}$ by identification by the canonical isometries between an Hilbert space and
 its dual (see Example \ref{Example:identification} below). When it will be convenient, we will indifferently manipulate chains or cochains in what follows as it is the most intuitive depending on the situation. In particular, given any $f \in \ch^{k}$ and any $\sigma \in \ch_k$, we will write
 \[
 f(\sigma) = \left<f, \sigma \right>_{\ch^k, \ch_k}. 
 \]\\
  Also, any function from $\s_k$ to $\R$ can be associated canonically with a function from $\ch^k$ to $\R$.

\begin{example}\label{Example:identification}To illustrate the two assertion above, let us consider the case $k=0$. We can identify the vertex $v\in \s_0=V\subset \ch_0$ to the function
\[\begin{array}{ccccc}
v^* & : & \ch_0 & \rightarrow  & \R\\
 &  & v & \mapsto  & 1\\
  &  & u\not=v & \mapsto  & 0.
\end{array}
\]Then, we can extend any function $\varphi\ :\ V\rightarrow \R$ to a function $\varphi\in \ch^0$ by
\[\varphi\big(\sum_{v\in V}\lambda_v v\big):=\sum_{v\in V} \lambda_v \varphi(v).\]
\end{example}

\subsection{Boundary and coboundary maps}
For any integer $k$, the boundary
map $\partial_k$ is the linear transformation $\partial_k\, :\,
\ch_k\rightarrow \ch_{k-1}$ which acts on basis elements
$[v_0,\dots,v_k] \in \s_k$ as
\begin{equation}
  \label{eq:erased}
  \p_k [v_0,\,\dots,\, v_k]= \sum_{i=0}^{k} (-1)^{i} [v_0,\,\dots,\, v_{i-1},\, v_{i+1},\, \dots,\,v_{k}],
\end{equation}
and $\partial_0$ is the null function.  Examples of such operations are
given in Table~\ref{fig. beta}.

\begin{table}[!ht]
\begin{center}
  \begin{tabular}{ccc}
    \begin{tikzpicture}[scale=0.45, font=\tiny]
      \draw[color=black] (0,0) -- (1,1) -- (2,0); \draw[-stealth]
      (0,0) -- (.5,.5); \draw[-stealth] (1,1) -- (1.5,.5); \filldraw
      [color=black] (0,0) node[below] {$v_0$} circle (.1); \filldraw
      [color=black] (1,1) node[above] {$v_1$} circle (.1); \filldraw
      [color=black] (2,0) node[below] {$v_2$} circle (.1); \filldraw
      [color=black] (3,0) node[below] {$v_0$} node[above] {$-$}
      circle (.1); \filldraw [color=black] (5,0) node[below] {$v_2$}
      node[above] {$+$ } circle (.1); \node at (2,-2)
      {$[v_0,v_1]+[v_1,v_2]\stackrel{\partial_2}{\longrightarrow}
      [v_2]-[v_0]$};
    \end{tikzpicture}
    &
    \begin{tikzpicture}[scale=0.45, font=\tiny]
      \filldraw[color=blue!50, draw=black] (0,0) -- (1,1) -- (2,0)--
      cycle; \draw[draw=black] (3,0) -- (4,1) -- (5,0)-- cycle;
      \draw[->] (1,.2) arc (270:-30:.2); \draw[-stealth] (0,0) --
      (.5,.5); \draw[-stealth] (1,1) -- (1.5,.5); \draw[-stealth]
      (2,0) -- (1,0); \draw[-stealth] (3,0) -- (3.5,.5);
      \draw[-stealth] (4,1) -- (4.5,.5); \draw[-stealth] (5,0) --
      (4,0); \filldraw [color=black] (0,0) node[below] {$v_0$}
      circle (.1); \filldraw [color=black] (1,1) node[above] {$v_1$}
      circle (.1); \filldraw [color=black] (2,0) node[below] {$v_2$}
      circle (.1); \filldraw [color=black] (3,0) node[below] {$v_0$}
      circle (.1); \filldraw [color=black] (4,1) node[above] {$v_1$}
      circle (.1); \filldraw [color=black] (5,0) node[below] {$v_2$}
      circle (.1); \node at (2,-2)
      {$[v_0,v_1,v_2]\stackrel{\partial_3}{\longrightarrow}
      [v_1,v_2]-[v_0,v_2]$}; \node at (3.3,-2.7) {$+[v_0,v_1]$};
    \end{tikzpicture}
    &
    \begin{tikzpicture}[scale=0.45, font=\tiny]
      \filldraw[color=blue!40, draw=black] (0,0) -- (0,1.5) --
      (-1.5,.75)-- cycle; \filldraw[color=blue!60, draw=black] (0,0)
      -- (0,1.5) -- (1,.75)-- cycle; \draw[->] (-.6,1) arc
      (110:-150:.15 and .3); \draw[->] (.4,.5) arc (235:-25:.1 and
      .3); \filldraw [color=black] (0,0) node[below] {$v_0$} circle
      (.1); \filldraw [color=black] (0,1.5) node[above] {$v_1$}
      circle (.1); \filldraw [color=black] (-1.5,.75) node[left]
      {$v_2$} circle (.1); \filldraw [color=black] (1,.75)
      node[below] {$v_3$} circle (.1);

      \draw[dashed] (2.5,.75) -- (5,.75); \fill[color=blue!40,
      opacity=.4] (5,.75) -- (4,1.5) -- (2.5,.75)-- cycle;
      \filldraw[color=blue!40, draw=black, opacity=.4] (4,0) --
      (4,1.5) -- (2.5,.75)-- cycle; \filldraw[color=blue!80,
      draw=black, opacity=.4] (4,0) -- (4,1.5) -- (5,.75)-- cycle;

      \node at (-.1,-.8) {Filled}; \node at (3.9,-.8) {Empty};

      \draw[->, densely dotted] (4.1,.85) arc (-40:220:.25 and .15);
      \draw[->, densely dotted] (4.2,.3) arc (-40:220:.3 and .15);
      \draw[->] (3.4,1) arc (110:-150:.15 and .3); \draw[->]
      (4.4,.5) arc (235:-25:.1 and .3);

      \filldraw [color=black] (4,0) node[below] {$v_0$} circle (.1);
      \filldraw [color=black] (4,1.5) node[above] {$v_1$} circle
      (.1); \filldraw [color=black] (2.5,.75) node[below] {$v_2$}
      circle (.1); \filldraw [color=black] (5,.75) node[right]
      {$v_3$} circle (.1);

      \node[anchor=east] at (2.5,-3.05)
      {$[v_0,v_1,v_2,v_3]\stackrel{\partial_4}{\longrightarrow}$};
      \node at (4,-2) {$ +[v_1,v_2,v_3]$}; \node at (4,-2.7)
      {$-[v_0,v_2,v_3]$}; \node at (4,-3.4) {$+[v_0,v_1,v_3]$};
      \node at (4,-4.1) {$-[v_0,v_1,v_2]$};
    \end{tikzpicture}\\
    a) & b) & c)
  \end{tabular}
  \caption{{\small \textit{Examples of boundary maps. From left to right. An
    application over 1-simplexes. Over a 2-simplex. Over a 3-simplex,
    turning a filled tetrahedron to an empty one.}}}
  \label{fig. beta}
  \end{center}
\end{table}
The maps $(\p_{k},\, k\ge 1)$ link the spaces $\ch_k$'s as follows:
\begin{equation}\label{chcomplex1}
  \dots \stackrel{\p_{k+2}}{\longrightarrow}\ch_{k+1}
  \stackrel{\p_{k+1}}{\longrightarrow}\ch_{k}
  \stackrel{\p_{k}}{\longrightarrow}\dots
  \stackrel{\p_2}{\longrightarrow}\ch_{1}
  \stackrel{\p_1}{\longrightarrow}\ch_{0}.
\end{equation}
It can then easily checked that for any integer $k$,
\begin{equation}\label{chcomplex2}
\p_k\circ\p_{k+1}=0.\end{equation}
In topology, a sequence of vector spaces
and linear transformations satisfying \eqref{chcomplex1} and \eqref{chcomplex2} is called a chain complex.
Equation \eqref{chcomplex2} implies that $B_k = \im \p_{k+1}\subset \ker \p_k = Z_k$, for $k\geq 0$, and we can define the $k$-th homology vector space $H_k$ as the quotient
vector space,
\begin{equation}
  \label{eq:erased2}
  H_k=\ker \p_k / \im \p_{k+1}.
\end{equation}
The $k$-th Betti number of ${\ch}$ is defined as its dimension:
\begin{equation}
\beta_k=\dim H_k=\dim \big(\ker \p_k\big) -\dim \big(\im \p_{k+1}\big).\label{def:betti}
\end{equation}

Notice that an element of $\ch_{1}$ is a sum of edges. It belongs to $\ker\p_{1}$ whenever
  these oriented edges form a cycle, in the sense of graph theory. So elements
  of $\ker \p_{k}$ are called $k$-cycles.\\

The Fig. \ref{fig: sets} illustrate the chain complex described above.\\

 \begin{figure}[!h]
   \centering
   \begin{tikzpicture}[font=\small]
     \draw (0,0) circle (1 and 1.5); \draw (-3,0) circle (1 and 1.5);
     \draw (3,0) circle (1 and 1.5); \filldraw[color=blue!60,
     draw=black] (0,0) circle (.8 and 1.2); \filldraw[color=white,
     draw=black] (0,0) circle (.5 and .75); \fill[color=black] (0,0)
     node[below] {0} circle (0.05); \fill[color=black] (3,0)
     node[below] {0} circle (0.05); \fill[color=black] (-3,0)
     node[below] {0} circle (0.05); \node[above] at (80:1 and 1.5)
     {$\ch_k$}; \node[below] at (90:.8 and 1.2) {$Z_k$};
     \node[below] at (90:.5 and .75) {$B_k$}; \node[above] at
     ($(-3,0)+(80:1 and 1.5)$) {$\ch_{k+1}$}; \node[above] at
     ($(3,0)+(80:1 and 1.5)$) {$\ch_{k-1}$}; \draw ($(-3,0)+(80:1 and
     1.5)$) -- ($(90:.5 and .75)$); \draw ($(-3,0)+(-80:1 and 1.5)$) --
     ($(-90:.5 and .75)$); \draw ($(80:.8 and 1.2)$) -- ($(3,0)$);
     \draw ($(-80:.8 and 1.2)$) -- ($(3,0)$); \draw ($(80:.5 and .75)$)
     -- ($(3,0)$); \draw ($(-80:.5 and .75)$) -- ($(3,0)$); \node at
     (-1.5,0) {$\stackrel{\p_{k+1}}{\longrightarrow}$}; \node at
     (1.5,0) {$\stackrel{\p_k}{\longrightarrow}$};
   \end{tikzpicture}
   \label{fig: sets}
   \caption{{\small \textit{A chain complex showing the sets $\ch_k$, $Z_k$ and
     $B_k$.}}}
 \end{figure}
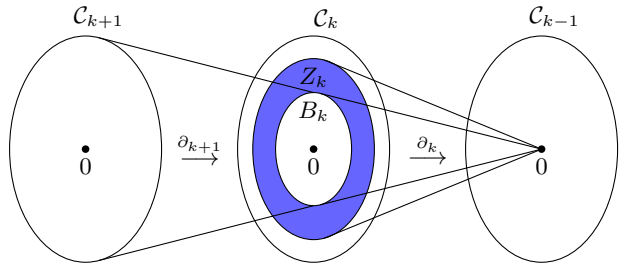

\begin{example}Let us consider the case of $k=0$ again to illustrate \eqref{def:betti}. Since $\p_0\equiv 0$, we have $\ker \p_0=V$. Also, we have $\im \p_1=\{u-v,\ [u,v]\in \s_1\}$. Hence, $H_0=\Span(V)/ \{u-v,\ [u,v]\in \s_1\}$ consists of equivalence classes of vertices which can be linked by a path in the graph. We thus recover that $\beta_0$ is the number of connected components of the graph. Recall that the latter number also corresponds to the number of zeros in the spectrum of the graph Laplacian \eqref{eq:defLaplacianmarche}. 
\end{example}

\begin{remark}
An element of $\ker \p_{1}$ which is not in $\im \p_2$ is a cycle which cannot be written as the boundary of a sum of triangles in $\s_{2}$.\\
  For instance, if
  \begin{equation*}
    \s_{1}=\{[ab],[bc],[cd],[ad]\} \text{ and }\s_{2}=\emptyset,
  \end{equation*}
  then the cycle
  \begin{equation*}
    [ab]+[bc]+[cd]-[ad]
  \end{equation*}
  which corresponds to the edges of a $4$-gone,  cannot be written as a sum of triangles since
  $\s_{2}$ is empty. Thus $\beta_{1}=1$ in this case.
\end{remark}

As $\ch_{k}$ and $\ch^k$ are Hilbert spaces, we can define the coboundary map $\d_{k}\, :\, \ch^{k-1}\longrightarrow \ch^{k}$ as the adjoint of $\p_{k}$:
namely for $f\in \ch^{k-1}$, $\d_{k}f\in \ch^{k}$ is
defined by its action over a $k$-chain by
\begin{multline}\label{eq_preliminaries:1}
  (\d_{k}f)[v_{0},\cdots,v_{k}]=\sum_{i=0}^{k}(-1)^{i}\langle f,\  [v_{0},\cdots,v_{i-1},v_{i+1},\cdots,v_{k}]\rangle_{\ch^{k-1},\ch_{k-1}}
  =f\bigl(  \p_{k}[v_{0},\cdots,v_{k}]\bigr).
\end{multline}
We will set by convention $\d_0\equiv 0$.

\begin{remark}[Interpretation of the coboundary map in the case $k=1$]
Recall that $\ch_{0}$ is generated as a $\KK$-vector space by the points
$ v\in V$. Let us denote by $\{v^{*},\, v\in V\}$ the corresponding dual
basis of $\ch^{0}$: $v^*(v)=1$ and $v^*(w)=0$ for $w\not=v$. 
Hence, following \eqref{eq_preliminaries:1}, we have for any function $f\in \ch^0$,
\[\p_1^* f[v_0,v_1]=f(v_1)-f(v_0).\]
In particular, if $f=w^*$ for $w\in V$,
\begin{equation*}
  (\d_{1}w^{*})([v_0,v_1])=-\langle w^{*},v_0\rangle_{\ch^{0},\ch_{0}}+\langle w^{*},v_1\rangle_{\ch^{0},\ch_{0}}=
  \begin{cases}
    0 & \text{ if } w\neq v_0,v_1\\
    -1 & \text{ if } w=v_0\\
    1 & \text{ if } w=v_1.
    \end{cases}
  \end{equation*}
The above computation shows that the coboundary of a vertex $w$ gives a weight 1 (resp. -1) to oriented edges having $w$ as alter (resp. ego), i.e. arriving at (resp. departing from) $w$. The coboundary map can then be interpreted in terms of fluxes.
\end{remark}

\begin{example}Let us consider an example with four vertices: $a$, $b$, $c$ and $d$.
   Consider that the edge $[a,b]$ belongs to the two triangles $[a,b,c]$ and
   $[a,b,d]$. Locally, the matrix representation of $\p_{2}$ looks
   like
   \begin{equation*}
     \bordermatrix{
       \ &  [a,b,c] & [a,b,d]\cr
       [a,b] & 1  & 1\cr
       [a,c] &-1 &0\cr
       [a,d] & 0 & -1\cr
       [b,c] & 1 & 0 \cr
       [b,d] & 0 & 1\cr
     }.
   \end{equation*}
The matrix representation of $\d_2$ is of course the transposed of the matrix representing $\p_2$ and recalling that we have identified $\mathcal{C}_1$ (resp. $\mathcal{C}_2$) to its dual $\mathcal{C}^1$ (resp $\mathcal{C}^2$),
   \begin{equation}\label{eq_preliminaries:6}
     \d_{2}[a,b]=[a,b,c]+[a,b,d]=\sum_{\tau\in \s_{2}^+} \langle[a,b]^*,\, \p_{2}\tau\rangle_{\ch^{1},\ch_{1}} \ \tau.
   \end{equation}The bracket on the right hand side counts the occurrence of the edge $[a,b]$ among the faces of $\tau$.
   Otherwise stated, the coboundary of an edge is the sum of the triangles
   which contain it, respecting its orientation.
\end{example}

As before, the $k$-th cohomology vector space, denoted by $H^{k}$, is defined as
\begin{equation}\label{eq:coboundaryHk}
H^{k}=\ker \d_{k}/ \im \d_{k-1}
\end{equation}and is the dual of $H_{k}$.

\subsection{Combinatorial Laplacian}
A crucial notion for the following is that of combinatorial Laplacian (see
\cite{forman,muhammad_control_2006} for details).
We know that

\begin{center}
 \begin{tikzcd}
   \arrow[r,"\p_{k+2}",dashed] & \ch_{k+1}  \arrow[d,leftrightarrow] \arrow[r,"\p_{k+1}"] & \ch_k \arrow[d,leftrightarrow]\arrow[r,"\p_{k}"] & \ch_{k-1} \arrow[d,leftrightarrow] \arrow[r,"\p_{k-1}",dashed] & { } \\
  { }  &\ch^{k+1}  \arrow[l,"\p_{k+2}^*",dashed]  & \ch^k \arrow[l,"\p_{k+1}^*"] & \ch^{k-1} \arrow[l,"\p_{k}^*"] & { } \arrow[l,"\p_{k-1}^*",dashed]
\end{tikzcd}
\end{center}
where the two tips arrows mean that we have an isometric isomorphism  between the
two concerned Hilbert spaces. Since we have identified the spaces $\mathcal{C}_k$ and $\mathcal{C}^k$ for any $k\in \N$, we may consider
\begin{equation}\label{eq_preliminaries:3}
 L_{k}^{\up}:= \p_{k+1}\p_{k+1}^{*}\, :\, \ch_{k} \xrightarrow{\p_{k+1}^{*}} \ch_{k+1} \xrightarrow{\p_{k+1}} \ch_{k}
\end{equation}
and
\begin{equation}\label{eq_preliminaries:4}
 L_{k}^{\down}:=\p_{k}^{*}\p_{k}\, :\, \ch_{k} \xrightarrow{\p_{k}} \ch_{k-1} \xrightarrow{\p_{k}^{*}} \ch_{k}.
\end{equation}These latter operators are called respectively the upper and lower Laplacians (of order $k$).
The combinatorial Laplacian of order~$k$ is defined as
\begin{equation}
  \label{eq:defLaplacian}
  \begin{split}
    L_{k}\, :\, \ch_{k} & \longrightarrow \ch_{k}\\
    \sigma & \longmapsto \left( \p_{k+1}\p_{k+1}^{*}+\p_{k}^{*}\p_{k} \right)(\sigma)= L_{k}^{\up}(\sigma)+ L_{k}^{\down}(\sigma).
  \end{split}
\end{equation}

By definition, all these three operators are self-adjoint, non-negative and compact. Their eigenvalues are thus real and non-negative. 
Furthermore, the combinatorial Hodge theorem \cite{friedmanBetti} says that
\begin{theorem}
  For any $k\in \N$, we have
  \begin{equation}
    \label{eq:HodgeDecomposition}
    \ch_{k}=\im \p_{k+1}\, \oplus \, \im \p_{k}^{*}\, \oplus\, \ker L_{k}.
  \end{equation}
  It follows that
  \begin{equation}
  \label{eq:Hodge}
  \ker L_{k}\simeq H_{k}.
\end{equation}
In particular, the $k$-th Betti number is the dimension of the null space of
$L_{k}$:
\begin{equation}\label{eq:Betti-Lk}
\beta_k= \dim \ker (L_k).
\end{equation}
\end{theorem}

\begin{remark}\label{remark:L0=-L}
The combinatorial Laplacian of order $0$ corresponds to the graph Laplacian.
Since we set by convention 
$\partial_{0}\equiv 0$, $L_{0}=L_0^{\up}=\p_{1}\p_{1}^{*}$. The map
$\p_{1}$ maps edges to vertices and its matrix representation is exactly the
so-called incidence matrix $B$ of the graph $(V,\s_{1})$: for $V=\{v_0,\cdots v_k\}$ and for $(e_j)$ an enumeration of the set of oriented edges $\s_1^+$,
\[B_{ij}=\begin{cases}
1 & \mbox{ if }v_i \mbox{ is the ego of }e_j\\
-1 &  \mbox{ if }v_i \mbox{ is the alter of }e_j\\
0 & \mbox{ otherwise}.
\end{cases}\]
Thus, $L_{0}=BB^{t}$ is such that 
\begin{equation*}
  (L_{0})_{ij}=
  \begin{cases}
   \deg(v_{i}) & \text{ if }i=j,\\
    -1 & \text{ if $v_{i}$ is adjacent to $v_{j}$,}\\
    0 & \text{ otherwise.}
  \end{cases}
\end{equation*}
The map $L_0$ from $\ch_0$ into itself is characterized for $v\in V$ by:
\[
  L_0 v^* = - \sum_{w\in V: [vw]\in \s_{1}} (w^*-v^*)=\sum_{w\in V: [vw]\in \s_{1}} \p_{1}[vw].
\]
Thus, for a function $f=\sum_{v\in V}\lambda_v v^* \in \ch^0$,
\begin{align}
-L_0f(u)=  &  -\sum_{v\in V}\lambda_v L_0 v^*(u) \nonumber\\
= & \sum_{v\in V}\lambda_v \sum_{w\in V : [vw]\in \s_{1}} (w^*-v^*)(u)\nonumber\\
= & \sum_{v\sim u}\lambda_v - \lambda_u \Card(w\sim u)\nonumber\\
= & \sum_{v\sim u} \big(f(v)-f(u)\big)= \generator f(u).\label{calcul:generatorL0}
\end{align}
Thus as mentioned above, $-L_{0}$ appears as the generator $\generator$ of the continuous time random walk on the
graph $(V,\s_{1})$ as defined in \eqref{eq:defLaplacianmarche}.
\end{remark}

To describe $L_{k}$, we need to introduce the notion of lower and upper
adjacency for $k$-simplexes. Two $k$-simplexes are said to be upper adjacent
whenever they are two faces of a common $k+1$ simplex of $\mathcal{C}$. Two   $k$-simplexes are said to be lower adjacent
whenever they are cofaces of a common $k-1$ simplex. For a simplex~$\tau$, its upper degree, $\deg_{\up}(\tau)$, is the number of simplexes
which are upper adjacent to it. Two upper adjacent simplexes are said to be similarly oriented if the
orientation they would inherit from their common higher order simplex coincides
with their orientation. They are said to be dissimilarly oriented otherwise. The lower degree $\deg_{\downarrow}(\tau)$ is the number of simplexes which are lower adjacent to $\tau$.
Two lower adjacent simplexes are similarly oriented whenever they induce the same
orientation to their intersection.

\begin{example}For instance, two edges are upper adjacent if they are part of a common triangle
and they are lower adjacent if they share a vertex. \end{example}

The lower and upper adjacency matrix are defined as it can be expected. For $\tau$ and $\tau' \in \s_k$,
\begin{equation*}
  A^{\up/\down}_{k}(\tau,\tau')=
  \begin{cases}
    1 &\text{ if } \tau \text{ and } \tau' \text{ are upper/lower adjacent and similarly oriented},\\
    -1& \text{ if } \tau \text{ and } \tau' \text{ are upper/lower adjacent and dissimilarly oriented,}\\
    0&\text{ otherwise.}
  \end{cases}
\end{equation*}
Consider also $D_{k}$ the diagonal matrix whose entries are the upper degrees of
each $k$-simplex. \\

We can compute the matrix of the combinatorial Laplacian of order $k$ \eqref{eq:defLaplacian}. Then (see \cite{muhammad_control_2006}),
\begin{equation*}
  L_{k}=L_{k}^{\up}+L_{k}^{\down}=\left( D_{k}-A^{\up}_{k} \right)+ \left(  (k+1)\Id+A^{\down}_{k}\right).
\end{equation*}

The map $L_{k}^{\up}$ has the features of a generator of a Markov process. Indeed, considering back the computation in Remark \ref{remark:L0=-L}, we see that the upper Laplacian can be explained in terms of upper-adjacency. For a vertex $u$ ($k=0$), we consider all the edges that are upper-adjacent to $u$, choose one with $u$ as ego and jump to the alter of this edge. This can be generalized for larger orders $k\geq 0$: we consider all the $k+1$ simplexes that are adjacent to a given $k$ simplex $\tau$, choose one at random that will determine the next movement (this shall be precised in the sequel). Associating to $L^\up_k$ a Markov chain provides a probabilistic interpretation to this operator which can help understand it better. In Parzanchevski and Rosenthal \cite{parzanchevskirosenthal}, a connection between this random walk $Y$ and homology is made by considering the `expectation process' defined for an oriented edge $e\in \s_1$ by $\mathcal{E}_t(e)=\P(Y_t = e)-\P(Y_t=-e)$.\\
As to the map $L_{k}^{\down}$, Mukherjee and Steenbergen \cite{mukherjeesteenbergen} proposed a similar random walk exploiting the lower-adjacency, but the generator $L_k^{\down}$ does not correspond to the generator of a Markov process. These authors introduce killings to deal with this problem. \\


However, the following result says that as long as
we are concerned with spectral properties, we can retrieve the information about
$L_{k}^{\down}$ by looking at $L_{k-1}^{\up}$ (see \cite{Zobel}):
\begin{theorem}\label{th:Zobel}
  Let $\lambda >0$ and $f$ be an eigenvector of $L_{k-1}^{\up}$. Then,
  $\p_{k}^{*}f$ is a $\lambda$-eigenvector of $L_{k}^{\down}$. Conversely, if
  $g$ is a $\lambda$-eigenvector of $L_{k}^{\down}$, then $\p_{k}g$ is a $\lambda$-eigenvector of $L_{k-1}^{\up}$.
\end{theorem}
\begin{proof}
  If $f$ satisfies $L_{k-1}^{\up}f=\lambda f$, then
  \begin{equation*}
    L_{k}^{\down}(\p_{k}^{*}f)=\p_{k}^{*}\p_{k}\p_{k}^{*}f=\p_{k}^{*} L_{k-1}^{\up}f=\lambda \,\p_{k}^{*}f.
  \end{equation*}
  The proof is similar for the converse.
\end{proof}
This means that it is reasonable to look only at the maps $L_{k}^{\up}$ as long as we are interested in properties related with the spectral decomposition of these operators.


\section{Random walk}
\label{sec:rw}
The idea behind the dynamics of our random walk is the following. The usual
random walk on a graph goes from vertex to vertex. The generator of the
continuous time random walk can be written as
\begin{equation}
  \generator \cochain(u)=\sum_{v\sim u}\cochain(v)-\cochain(u)=\sum_{v\in V: [uv]\in \s_{1}} \cochain(u+\p_{1}[uv])-\cochain(u).\label{def:Laplacian2}
\end{equation}
In the next dimension, $k=1$, points are replaced by edges and edges by
triangles. If we follow Parzanchevski and Rosenthal
\cite{parzanchevskirosenthal,rosenthal} or Kaufman and Oppenheim \cite{kaufmanoppenheim}, a natural edge-valued random walk
consists in jumping from the current edge $e$ to a uniformly chosen
upper-adjacent edge. These `higher-order' random walks can be extended to $k\geq 1$. Mukherjee and Steenbergen proposed a similar random walks
exploiting the lower-adjacency. But if we look for an analogue of
\eqref{def:Laplacian2}, another way is to add the boundary of a triangle to an
edge, which gives a combination of edges, i.e. a 2-chain. It is thus natural not
to restrict to edges, but to consider a random walk that takes its values in
$\ch_2$ or more generally $\ch_k$. 

Recall that in \eqref{eq:erased2}, the homology space $H_k=\ker \p_k / \im
\p_{k+1}$, a natural way to explore the homology classes of $H_k$ is to start
with an element of $\ker \p_k$ and then have transitions in $\im \p_{k+1}$.
Proceeding so, the random walk will remain in the homology class of its initial
element, in the same way as the usual random walk remained in the connected
component of its initial condition. Let us
now describe the transitions and generator of this random walk.

\subsection{Generator of the chain-valued random-walk}

In what follows, $k\ge 1$ is fixed.\\

\noindent \textbf{Space of test functions.} To define the generator of the random walk, we first introduce a space of functions on which it will operate.\\

 \begin{definition}
   Consider $\cyl$ the space of functions from $\ch_{k}$ to $\R$ of the form
   \begin{equation}
     \label{eq_rw:3}
     F(\chain)=f\left(\langle\eta_{1},\chain\rangle_{\ch^{k},\ch_{k}},\cdots,\langle\eta_{m},\chain\rangle_{\ch^{k},\ch_{k}}\right)
   \end{equation}
   for some $m\ge 1$, $(\eta_{1},\cdots,\eta_{m})$ some elements of $\ch^{k}$
   and $f$ measurable and bounded from $\R^{m}$ into~$\R$.

   We define the support of $F$ as:
   \begin{equation*}
     \supp F=\bigcup_{i=1}^{m} \supp \eta_{i},
   \end{equation*}where we recall that $\supp \eta_i$ is defined in \eqref{eq_preliminaries:5}.
 \end{definition}

 \begin{lemma}
   The space $\cyl$ is separating in $B(\ch_{k})$, the Banach space of bounded
   measurable functions from $\ch_{k}$ to~$\R$, equipped with the sup-norm.
 \end{lemma}
 \begin{proof}
   Since $\ch^{k}$ can be embedded as a closed subset of $l^{2}(\N^{k+1})$, it is
   a separable Hilbert space, we can consider a dense sequence $(\eta_{n},\,
   n\ge 1)$ and
   \begin{equation*}
     \F_{m}=\sigma\{\langle\eta_{i},.\rangle_{\ch^{k},\ch_{k}},i=1,\cdots,m\}.
   \end{equation*}
   Since $\ch_{k}$ is an Hilbert space, its Borel $\sigma$-field is equal to
   $\vee_{m}\F_{m}$, hence the result.
 \end{proof}

\noindent \textbf{Transition kernel. }
Let us explain the transition of our chain-valued random walk. For the sake of simplicity, imagine here that $k=1$. Assume that we are in a state $\chain\in \ch_1$ (a cycle for $k=1$). Because transitions are in $\im \p_{2}$, let us consider an element of this space, say $\p_{2}\tau$ for $\tau\in \s_{2}$ (a triangle). $\p_2\tau$ defines a possible transition if $\tau$ is upper-adjacent to $\chain$, i.e. if $\tau$ and $\chain$ share at least one edge. All the possible transitions from $\chain$ are obtained by letting $\tau$ vary in $\s_2$. The more $\tau$ and $\chain$ have edges in common and the more $\tau$ will be likely to define the next step of the random walk and we thus need to define a weight to account for this.\\

For a $k$-chain $\chain\neq 0$ and an oriented $(k+1)$-simplex $\tau\in \s_{k+1}$,
define the number of common faces between $\chain$ and $\p_{k+1}\tau$ by:
\begin{equation*}
  w\left(\chain,\  \p_{k+1}\tau\right)=\langle (\p_{k+1}\tau)^*,\ \chain \rangle_{\ch^{k},\ch_{k}}^{+}
\end{equation*}where $x^+=\max(x,0)$ for $x\in \R$.
For another chain $\chain'$, we say that $\chain$ and $\chain'$ are adjacent (in the sense that the random walk can reach $\chain'$ from the state $\chain$) and write
\begin{equation}
  \label{eq:Voisinage}
  \chain \sim \chain' \Longleftrightarrow \exists \tau \in \s_{k+1},  w\left(\chain,\  \p_{k+1}\tau\right)>0 \text{ and } \chain'=\chain - \p_{k+1}\tau.
\end{equation}
Finally, let us define the weight of the transition from $\chain$ to $\chain'$:
\begin{equation}\label{eq_rw:2}
  K(\chain,\, \chain')=
  \begin{cases}
    1 & \text{ if } \chain=\chain'=0\\
    w\left(\chain,\  \chain-\chain'\right)&\text{ if } \chain \sim \chain' \\
    0&\text{ otherwise.}
  \end{cases}
\end{equation}

 \begin{example}
   In Figure~\ref{fig:orientation}, we see how a difference of orientations is
   simply reflected in the value of the scalar product. Note also that
   $w(\chain,\chain-\chain')$ can be viewed as a scalar product that counts the number
   of edges of the triangle which are adjacent to the chain with the good
   orientations.
   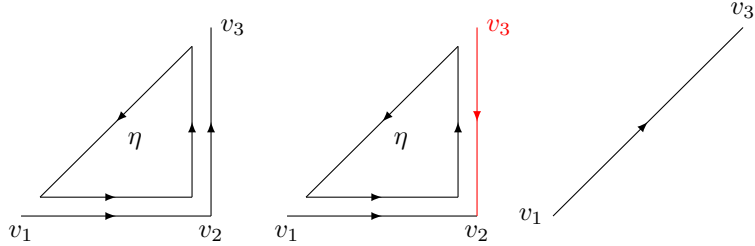
\begin{figure}[!ht]
     \centering
     \begin{tikzpicture}[scale=0.5]
       \draw[middlearrow={latex}] (0,0) node[below] {$v_{1}$}-- (5,0)
       node[below] {$v_{2}$}; \draw[middlearrow={latex}] (5,0) -- (5,5)
       node[right] {$v_{3}$}; \draw[middlearrow={latex}] (0.5,0.5) -- (4.5,0.5);
       \draw[middlearrow={latex}] (4.5,0.5) -- (4.5,4.5);
       \draw[middlearrow={latex}] (4.5,4.5) -- (0.5,0.5); \draw (3, 2)
       node{$\eta$}; \draw[middlearrow={latex}] (7,0) node[below] {$v_{1}$}--
       (12,0) node[below] {$v_{2}$}; \draw[middlearrow={latex},color=red] (12,5)
       node[right] {$v_{3}$} -- (12,0); \draw[middlearrow={latex}] (7.5,0.5) --
       (11.5,0.5); \draw[middlearrow={latex}] (11.5,0.5) -- (11.5,4.5);
       \draw[middlearrow={latex}] (11.5,4.5) -- (7.5,0.5); \draw (10, 2)
       node{$\eta$}; \draw[middlearrow={latex}] (14,0) node[left]{$v_{1}$} --
       (19,5) node[above] {$v_{3}$};
     \end{tikzpicture}
     \caption{{\small \textit{Different cases of orientations:
           $\tau=[v_{1}v_{2}v_{3}]$. Here $\chain$ is a 1-chain, not necessarily a cycle. (a) In this case, $\chain=[v_{1}v_{2}]+[v_{2}v_{3}]$ and
           $w(\chain,\p_2\tau)=2$, which is the number of edges in common between $\chain$ and $\tau$. (b) Here, $\chain=[v_{1}v_{2}]-[v_{2}v_{3}]$ and
           $w(\chain,\p_2\tau)=0$. So $\tau$ is never chosen for defining the
           transition to the next step here. (c) In the case (a), the next step
           is $\chain'=[v_{1}v_{2}]+[v_{2}v_{3}]-\p_{2}\tau=[v_{1}v_{3}]$. }}}
     \label{fig:orientation}
   \end{figure}
 \end{example}

 \paragraph{Generator of the random walk} Let us define by $(\generator,D(\generator))$ the
 generator of the continuous-time random walk.
 \begin{definition}[Cycle-valued random walk]\label{def:A-rw}Let $D(\generator)$ be the set of functions $F$ such that
   $|\sum_{\chain'\sim \chain} \Bigl( F(\chain')-F(\chain) \Bigr)\, K(\chain,\,
   \chain')|<+\infty$. For $F\in D(\generator)$, we can define
   \begin{align*}
     \generator F\, :\, \ch_{k} &\longrightarrow \R\\
     \sigma &\longmapsto \sum_{\chain'\sim \chain} \Bigl( F(\chain')-F(\chain) \Bigr)\, K(\chain,\, \chain').
   \end{align*}
 \end{definition}
Let us consider a Lipschitz continuous function $F$ from $\ch_{k}$ to $\R$, i.e. such that there exists $c_{F}>0$ such that for any
 $\chain,\chain'\in \ch_{k}$, such that $\chain \sim \chain'$, we have 
 \begin{equation*}
   |F(\chain)-F(\chain')|\le c_{F} \|\chain - \chain'\|_{\ch_k} = c_F \|\p_{k+1}\tau\|_{\ch_k} = c_F \sqrt{k+2},
 \end{equation*}
 because the boundary of $\tau \in \ch_{k+1}$ contains $k+2$-$k$ simplexes. 
Moreover, we remark that if $\chain$ and $\chain'\in \ch_k$,
 \begin{equation*}
   \chain\sim \chain' \Longrightarrow K(\chain,\chain')\le \sqrt{k+1}\|\sigma\|_{\mathcal{C}_k}.
 \end{equation*}Indeed, $\sigma'-\sigma$ has $k+1$ faces of weight 1, so  $K(\chain,\chain')$ is bounded by the sum of the weights of the faces of $\sigma$. Hence, 
 \[
   \left| \sum_{\chain'\sim \chain} \Bigl( F(\chain')-F(\chain) \Bigr)\, K(\chain,\,
     \chain') \right|\le c_{F}(k+2)\ \|\sigma\|_{\mathcal{C}_k} ,
 \]from which we deduce that $F \in D(\generator)$.\\
 It is immediate that (see Ethier and Kurtz
 \cite[Chapter 4, Section 2 and Chapter 8, Section 3]{ethierkurtz}):
 \begin{theorem}
   The map $\generator$ of domain $D(\generator)$ generates a strong Feller continuous Markov
   process $X=(X_t)_{t\geq 0}$ on $C_b(\ch_{k},\R)$, the set of continuous
   bounded functions on $\ch_{k}$. The set $\core=\cyl\cap C_b(\ch_{k},\R)$ is a
   core for $X$.
 \end{theorem}
%

 Remark that the process $X$ is a continuous-time pure jump process and admits a
 representation with a discrete-time Markov chain and exponentially distributed
 clocks attached with each possible transitions (see e.g. \cite[Chapter 4,
 Section 2]{ethierkurtz}).

\begin{theorem}
  For any $t>0$, $X_t$ remains in the same homology class as $X_0$. Moreover,
  if $X_0\in \ker \p_{k}$, then for any $t\ge 0$, $X_t$ belongs to
  $\ker\p_{k}$.
\end{theorem}
\begin{proof}
  At each change of state, we add to $X$ an element of $\im \p_{k+1}$. Since $\p_{k}\circ\p_{k+1}=0$, we add only
  elements of $\ker\p_{k}$ to $X_0$, hence $X_t$ always belongs to
  $\ker\p_{k}$ and the homology class does not change along the dynamics of $X$.
\end{proof}

We can precise the link between $\generator$ and $L^{\up}_k$. Note that they cannot be
equal since $\generator$ operate on functions of chains whereas $L^{\up}_k$ operates on
cochains $\ch^k$, i.e. linear functions of chains. However we will see that
$\generator$, $-L_k$ and $-L^{\up}_k$ coincide when restricted to $\ker \p_k$.

\begin{theorem}
  \label{thm:AEgalLup}
  We have for every $\zeta\in \ker \p_{k}$, and for any $\sigma\in \ch_k$,
  \begin{equation}\label{eq_rw:1}
    \generator \zeta (\chain)=-L^{\up}_{k} \zeta(\chain)=-L_k \zeta(\chain).
  \end{equation}
\end{theorem}

\begin{proof}
From \eqref{def:A-rw}, we have:  \begin{equation}    \label{eq:AEgalLup}
    \generator\zeta(\chain)=\sum_{\chain'\sim \chain} \big(\zeta(\chain')-\zeta(\chain)\big) K(\chain,\,\chain').
  \end{equation}
  By the definitions of the $\sim$ relation \eqref{eq:Voisinage} and of
  $K$ \eqref{eq_rw:2}:
  \begin{align}
    \sum_{\chain'\sim \chain} \big(\zeta(\chain')-\zeta(\chain)\big) K(\chain,\,\chain')= & - \big\langle\zeta, \sum_{\tau\in\s_{k+1}}  \langle  (\p_{k+1} \tau)^* ,\chain\rangle_{\ch^{k},\ch_{k}}^{+}\ \p_{k+1}\tau\big\rangle_{\ch^k,\ch_k} \nonumber\\
    =&-\big\langle\p_{k+1}^{*}\zeta,\  \sum_{\tau\in\s_{k+1}} 
       \langle  (\p_{k+1} \tau)^* ,\chain\rangle_{\ch^{k},\ch_{k}}^{+}\ \tau\big\rangle_{\ch^{k+1},\ch_{k+1}} \nonumber\\
    =&-\big\langle\p_{k+1}^{*}\zeta,\  \sum_{\tau\in\s_{k+1}} 
       \langle  \tau^* ,\p_{k+1}^* \chain\rangle_{\ch^{k},\ch_{k}}^{+}\ \tau\,\big\rangle_{\ch^{k+1},\ch_{k+1}} \nonumber\\
    =&-\langle\p_{k+1}^{*}\zeta,\
       \p_{k+1}^*\chain\rangle_{\ch^{k+1},\ch_{k+1}}\nonumber\\
    =&-\langle\p_{k+1}\circ \p_{k+1}^{*}\zeta,\
       \chain\rangle_{\ch^{k+1},\ch_{k+1}} \nonumber\\
    =&- L^{\up}_k\zeta(\chain).\label{calcul:A-L}
  \end{align}
  When $\chain \in \ker \p_k$, $L_k^{\downarrow}(\tau)=0$. This concludes the
  proof.\end{proof}

 \subsection{Non-explosion and recurrence of the chain on finite simplicial complexes}

 In contrast with the $0$-dimensional case, there is
 no reason why the state space of the stochastic process $X$ should be finite in the general case, even when $\s_k$ is, since
 the weights $\lambda_\tau$ in \eqref{eq_preliminaries:2} can be unbounded. For a cycle random walk, this means for instance, that the cycle can do loops and go several times through the same edge. 
Even in a simple structure such as the triangulation of the torus, this can happen. 

 \begin{figure}[!ht]
  \centering
 \begin{tikzpicture}[scale=0.5,font=\fontsize{6}{6}\selectfont]
            \draw (-3.5,-2) -- (3.5,-2);
            \draw (-3.5,2) -- (3.5,2);
            \draw (-3.5,0) -- (3.5,0);
            \draw (-3.5,-3) -- (-0.5,3);
            \draw (-1.5,-3) -- (1.5,3);
            \draw (0.5,-3) -- (3.5,3);
            \draw (-3.5,3) -- (-0.5,-3);
            \draw (-1.5,3) -- (1.5,-3);
            \draw (0.5,3) -- (3.5,-3);
            \draw [->,line width=0.75mm, blue] (-3.5,-2) -- (-1,-2);
            \draw [->,line width=0.75mm, green] (-1,-2) -- (1.1,-2);
            \draw [->,line width=0.75mm, blue] (1.1,-2) -- (1.5,-3);
            \draw [line width=0.75mm, orange] (-3.5,3) -- (-1,-2);
            \draw [->,line width=1mm, red] (-1,-2) -- (0,0);
            \draw [line width=1mm, orange] (0,0)-- (1,-2);
            \draw [->,line width=1mm, orange ] (1,-2) -- (3.5,3);
          \end{tikzpicture}
  \caption{\textit{Example illustrating the occurrence of edges with weights greater than 1 in the triangulation of the torus. Imagine a cycle that is composed of the blue-green curve, and then the red-orange curve. If the red edge is selected to determine the next step of the chain, and if the lower triangle adjacent to it is chosen, then the green edge will have a weight equal to two.}}
  \label{fig:torus-contreexemple}
\end{figure}
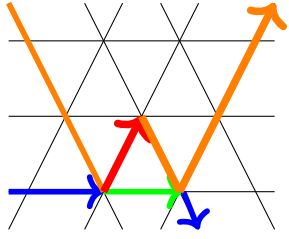
 
 Fortunately, the dynamics of the random walk tend to limit its length. 
 
 \begin{definition}
  For a $k$-chain $\chain$, we denote $\rec(\chain)$, the recurrence class of
  $\chain$, consisting of all the $k$-chains which can be attained by $X$ starting
  from $\chain.$
\end{definition}

 \begin{theorem}\label{th:recurrence}
 Suppose $\s_{k+1}$ is finite and $X_0$ is a $k$-cycle. Then, the process $(X_t)_{t \geq 0}$ on $\rec(X_0)$ is non-explosive, recurrent and admits a unique invariant measure $\pi$. Furthermore,
 \[
 \int_{\ch_{k}} \|\chain\|_{\ch_k}^2 d\pi(\chain) < \infty. 
 \]
 \end{theorem}
 
 \begin{proof}
 To prove the result, we are going to show that $\|.\|^2_{\ch_k}$ is a Lyapunov function for $\generator$. Let $P_{\ker}, P_{\im} : \ch_k \rightarrow \ch_k$ be the projection operators of $k$-chains to $\ker L_{k}^{\up}$ and $\im L_{k}^{\up}$ respectively. 
For $\chain \in \rec(X_0)$ we have
 \begin{align}
	(\generator \|.\|_{\ch_k}^2) (\chain) & = \sum_{\tau \in \s_{k+1}} (\|\chain - \partial_{k+1} \tau\|_{\ch_k}^2 - \|\chain\|_{\ch_k}^2) K(\chain, \chain-\tau) \nonumber\\
	& = \sum_{\tau \in \s_{k+1}} (-2 \langle (\partial_{k+1} \tau)^*,\chain \rangle_{\ch^k,\ch_k} + \|\partial_{k+1} \tau\|_{\ch_k}^2) K(\chain, \chain-\tau)\nonumber\\
	& = 2 \langle \generator \sigma^*, \sigma\rangle_{\ch^{k},\ch_{k}} + (k+2) \sum_{\tau \in  \s_{k+1}} \left<(\partial_{k+1} \tau)^*, \chain \right>_{\ch^{k},\ch_{k}}^+ \nonumber\\
	& = - 2 \left<L_{k}^{\up} \chain^*, \chain \right>_{\ch^{k},\ch_{k}} + (k+2) \sum_{\tau \in  \s_{k+1}} \left<\tau^*, \partial_{k+1}^* \chain \right>_{\ch^k,\ch_k}^+ .\label{etape:majA}
\end{align}The factor $k+2$ in the third line comes from the fact that $\|\partial_{k+1}\tau \|^2_{\ch_k}=k+2$. For the fourth equality, we have used Theorem~\ref{thm:AEgalLup} from which:
 \[
  \langle \generator \sigma^*, \sigma\rangle_{\ch^{k},\ch_{k}} = - \left<L_{k}^{\up} \chain^*, \chain \right>_{\ch^{k},\ch_{k}}.
 \]
 Since $\mathcal{S}_k$ is finite, $L_{k}^{\up}$ is a finite dimensional operator with positive discrete eigenvalues. Thus, denoting by $\lambda_m$ the smallest positive eigenvalue of $L_{k}^{\up}$, we have
 \[
	\left<L_{k}^{\up} P_{\im} \chain^*, P_{\im} \chain \right>_{\ch^{k},\ch_{k}} \geq \lambda_m \|P_{\im} \chain\|_{\ch^{k}}^2.
 \]
 On the other hand, by Cauchy-Schwarz inequality:
 \begin{multline}\label{etape:cauchy-schwarz}
    \sum_{\tau \in  \s_{k+1}} \left\langle (\partial_{k+1} \tau)^*, \chain \right\rangle_{\ch^{k},\ch_{k}}^+ = \sum_{\tau \in  \s_{k+1}} \left\langle \tau^*, \partial_{k+1}^* \chain \right\rangle_{\ch^{k+1},\ch_{k+1}}^+ \\
      \leq | \s_{k+1}|^{1/2} \left(\sum_{\tau \in  \s_{k+1}} \left\langle \tau^*, \partial_{k+1}^* \chain \right\rangle_{\ch^{k+1},\ch_{k+1}}^2 \right)^{1/2} = | \s_{k+1}|^{1/2} \| \partial_{k+1}^* \chain \|_{\ch^{k+1}}.
 \end{multline}
 Then:
 \begin{equation}
	\| \partial_{k+1}^* \chain \|_{\ch^{k+1}}^2 = \langle \partial_{k+1}^* \chain^*, \partial_{k+1}^* \chain^*\rangle_{\ch^{k+1},\ch^{k+1}} = \langle L_{k}^{\up} \chain^*, \chain^*\rangle_{\ch^{k},\ch^{k}} \leq \lambda_M \|\chain\|^2_{\ch_k}, \label{triangles_adjacents2}
 \end{equation}
 where $\lambda_M$ is the largest eigenvalue of $L_{k}^{\up}$. \eqref{etape:cauchy-schwarz} and \eqref{triangles_adjacents2} imply that:
 \begin{equation}
     \sum_{\tau \in  \s_{k+1}} \left\langle(\partial_{k+1} \tau)^*, \chain \right\rangle_{\ch^{k},\ch_{k}}^+ \leq \sqrt{\lambda_M |\s_{k+1}|}\  \|\sigma\|_{\ch_k}.\label{etape:rec}
 \end{equation}
 Now, since the transitions of $X$ are in $\im L_{k}^{\up}$, we have that 
 \[
 \forall \chain \in \rec(X_0),\quad  \chain = P_{\ker} X_0 + P_{\im} \chain
 \]
 thus
 \[
 \|P_{\im} \chain\|_{\ch_{k}}^2 =  \|\chain\|_{\ch_{k}}^2 - \|P_{\ker} X_0\|_{\ch_{k}}^2.
 \]
 Combining everything together, we have 
 \begin{align}
	(\generator \|.\|_{\ch^{k}}^2) (\chain) & 
	 \leq - 2 \left<L_{k}^{\up} P_{\im} \chain^*, P_{\im} \chain \right>_{\ch^{k},\ch_{k}} +(k+2) | \s_{k+1}|^{1/2} \lambda_M^{1/2} \|\chain\|_{\ch_{k}}  \nonumber\\
	& \leq  -2 \lambda_m \|P_{\im} \chain\|_{\ch_{k}}^2 + (k+2) | \s_{k+1}|^{1/2} \lambda_M^{1/2} \|\chain\|_{\ch_{k}} \nonumber\\
	& \leq -2 \lambda_m \|\chain\|_{\ch_{k}}^2 + (k+2) | \s_{k+1}|^{1/2} \lambda_M^{1/2} \|\chain\|_{\ch_{k}} + 2 \lambda_m \|P_{\ker} X_0\|_{\ch_{k}}^2 \nonumber\\
	& \leq - \lambda_m \|\chain\|_{\ch_{k}}^2 + ((k+2) | \s_{k+1}|^{1/2}  \lambda_M^{1/2} \|\chain\|_{\ch_{k}} - \lambda_m \|\chain\|_{\ch_{k}}^2)  + 2 \lambda_m \|P_{\ker} X_0\|_{\ch_{k}}^2 \nonumber\\
	& \leq  - \lambda_m \|\chain\|_{\ch_{k}}^2 + \frac{(k+2)^2 | \s_{k+1}| \lambda_M}{4\lambda_m} + 2 \lambda_m \|P_{\ker} X_0\|_{\ch_{k}}^2,\label{eq:lyapunov}
\end{align}
where the last inequality comes from the fact that the function \[x\mapsto (k+2) | \s_{k+1}|^{1/2} \lambda_M^{1/2} x-\lambda_m x^2\]reaches its maximum at $x=(k+2)| \s_{k+1}|^{1/2}  \lambda_M^{1/2}/(2\lambda_m)$ at which point the value of the function is $(k+2)^2 | \s_{k+1}| \lambda_M/(4\lambda_m)$. Thus, $\|.\|_{\ch_k}^2$ is a Lyapunov function for $\generator$ and the results then follow by applying Theorems 2.1, 4.2 and 4.3 from Meyn and Tweedie \cite{meyntweedie}.
 \end{proof}

 The computation \eqref{eq:lyapunov} in the proof of Theorem \ref{th:recurrence} can be further exploited to show that the expectation of the square norm remains bounded along the chain:
 \begin{corollary}\label{cor:moment_chk}
 Suppose that $\s_{k+1}$ and $\E\big(\|X_0\|^2_{\ch_k}\big)$ are finite. Then, for $T>0$, and for any $t\in [0,T]$,
 \begin{equation}
     \E\big(\|X_t\|^2_{\ch_k}\big)\leq \Big[(1+2\lambda_m T)\E\big(\|X_0\|^2_{\ch_k}\big)+ T \frac{(k+2)^2 |\s_k| \lambda_M}{4\lambda_m} \Big]e^{-\lambda_m t}.
 \end{equation}
 \end{corollary}
 
 \begin{proof}For $M>0$, let us consider the stopping time $S_M=\inf\{\|X_t\|^2_{\ch_k} > M\}$.
   Using It\^o's formula \cite[Th.5.1 P.66]{ikedawatanabe}:
   \begin{align*}
       \E\big(\|X_{t\wedge S_M}\|^2_{\ch_k}\big)= & \E\big(\|X_0\|^2_{\ch_k}\big)+\E\Big(\int_0^{t\wedge S_M} \big(\mathcal{A}\|.\|^2_{\ch_k}\big)(X_s)ds \Big)\\
       \leq & \E\big(\|X_0\|^2_{\ch_k}\big)- \lambda_m \int_0^t \E\big( \|X_{s\wedge S_M}\|^2_{\ch_k}\big) ds \\
       & \hspace{1cm} + T \Big[\frac{(k+2)^2| \s_{k+1}|\lambda_M}{4\lambda_m} + 2 \lambda_m \|P_{\ker} X_0\|_{\ch_{k}}^2\Big]\\
       \leq & \Big[\E\big(\|X_0\|^2_{\ch_k}\big)+ T \frac{(k+2)^2  |\s_{k+1}| \lambda_M}{4\lambda_m} + 2 \lambda_m T \E\big(\|P_{\ker} X_0\|_{\ch_{k}}^2\big)\Big]e^{-\lambda_m t},
   \end{align*}by Gronwall's lemma. Because the right hand side does not depend on $M$, we deduce that $\lim_{M\rightarrow +\infty} S_M=+\infty$ almost surely and the result follows.
 \end{proof}
 


We can go a little further. For any $k+1$-simplex $\tau$, let $\deg_{\downarrow}(\tau)$ be
the lower adjacency degree of $\tau$ given by 
\begin{equation}\label{def:deglow}
  \deg_{\downarrow}(\tau) = \sum_{\tau' \in \s_{k+1}, \tau' \neq \tau} |\langle
  \big(\partial_{k+1}\tau\big)^*,
  \partial_{k+1}\tau'\rangle_{\ch^{k},\ch_{k}}|.
\end{equation}The quantity in the sum corresponds to the number of faces that $\tau$ and $\tau'$ have
in common, so $\deg_{\downarrow}(\tau)$ counts
 the number of $k$-simplexes (with multiplicity) that share a face with $\tau$. 

\begin{theorem}\label{th:rw_remains_simple}
Let $k \in \N$. Let $\chain_1, \dots, \chain_{\beta_k} \in \ch_k$ be a basis
of $H_k$ and let $\tau_1,\dots,\tau_n\in \s_{k+1}^+$ be the $k+1$-simplexes of
our simplicial complex. Suppose that there exists $\chain \in \ch_k$ and $(\lambda_{\tau}, \, \tau \in \s_{k+1}^+)\in \{-1,0,1\}^{\s_{k+1}^+}$ such that 
such that
\[
    X_0 = \chain + \sum_{\tau \in \s_{k+1}^+}
    \lambda_{\tau} \partial_{k+1} \tau.
\]
If we have for all $\tau \in \s_{k+1}^+$ that:
\begin{equation}\label{hyp:th10}
    \deg_{\downarrow}(\tau) \leq k + 2 - |\langle \big(\partial_{k+1} \tau\big)^*, \chain \rangle_{\ch^k,\ch_k}|,
\end{equation}
then $X$ has a finite state space and for any $t\geq 0$, there exists
$(\lambda_{\tau}(t), \tau \in \s_{k+1}^+)\in \{-1,0,1\}^{\s_{k+1}^+}$ such that
\begin{equation}\label{ecriture:decomposition}
    X_t = \chain + \sum_{\tau \in \s_{k+1}^+}
\lambda_\tau(t) \partial_{k+1} \tau.
\end{equation}
\end{theorem}

This theorem ensures that under condition \eqref{hyp:th10}, the state space is finite and the chain does not `loop' on itself too much as in the example of Fig. \ref{fig:torus-contreexemple}. This has direct consequences to determine its recurrence or to bound its length. In particular, if our simplicial complex is a triangulation of $\mathbb{R}^d$ with some holes, then \eqref{hyp:th10} is verified for any chain $\sigma$ surrounding one (or several) holes. 

\begin{proof}[Proof of Theorem \ref{th:rw_remains_simple}]
  By the very definition of $X$, for $T_{1}$ the first jump time of $X$, there
  exist $(\lambda_{\tau}(T_{1}), \tau \in \s_{k+1}^+)$ such that
  \[
    X_{T_{1}} = \chain + \sum_{\tau \in \s_{k+1}^+}
   \lambda_\tau(T_{1}) \partial_{k+1} \tau.
  \]
  and $| \lambda_\tau(T_{1})-\lambda_{\tau }| \leq 1$ for any $\tau\in
  \s_{k+1}^+$.
  First, consider $\eta\in \s_{k+1}^+$ such that $\lambda_{\eta} =1$. We
  have
  \begin{multline*}
    \langle (\partial_{k+1} \eta)^*, X_0 \rangle_{\ch^k,\ch_k} =    \langle (\partial_{k+1} \eta)^*, \chain\rangle_{\ch^k,\ch_k}
+ \sum_{\tau\in \s_{k+1}^+}\lambda_\tau  \langle (\partial_{k+1}\eta)^*, \partial_{k+1}\tau \rangle_{\ch^k,\ch_k}\\
    \geq  -| \langle (\partial_{k+1} \eta)^*, \chain\rangle_{\ch^k,\ch_k}|
           +(k+2) - \sum_{\tau\neq\eta}
           |\langle(\partial_{k+1}\eta)^*, \partial_{k+1}\tau\rangle_{\ch^k,\ch_k} |\geq 0,
  \end{multline*}
  by our assumption \eqref{hyp:th10} and by the definition of $\deg_{\downarrow}(\tau)$ in \eqref{def:deglow}. Thus,
  \begin{equation}
    K\big(X_0, X_0 + \partial_{k+1} \eta\big) 
    = w(X_0,-\partial_{k+1}\eta)=\max\big(0, -\langle (\partial_{k+1}
    \eta)^*,X_0\rangle_{\ch^k,\ch_k}\big) = 0.
  \end{equation}
Hence, $\lambda_\eta(T_{1})$ that was equal to 1 before the jump event can not increase and must be equal to $0$ or to $1$ after the jump.\\

  Proceeding similarly for the simplexes $\eta\in \s_{k+1}^+$ for which $\lambda_\eta = -1$,
  we obtain that $\lambda_\tau(T_{1}) \in \{-1, 0, 1\}$ for any $\tau\in
  \s_{k+1}^+$. The proof is then concluded by induction with respect to the steps of
  the embedded discrete time Markov chain.
\end{proof}

\begin{remark}
If the chain $X$ on a finite simplicial complex starts from a state $\tau \in \im \partial_{k+1}$ then it is absorbed by the null chain in finite time.
\end{remark}

    It is well known that for random walks on undirected graphs, the stationary distribution gives to each vertex a weight proportional to its degree, i.e. to the number of edges it is adjacent to. In a way, the stationary measure highlights central nodes with respect to the $0$-dimensional topology structure (i.e. connectivity). 
    While we are not able to directly compute the stationnary measure of our random walk, we would expect it to also highlight the topology of the simplicial complex. In Section~\ref{sec:applications}, we will see that, empirically, the stationnary measure of our random walk on edges tends highlight homology structures (in our case holes) by giving much weight to edges bordering the holes present in the homology class of our starting point. In practice our random walk would thus favor small chains circling the homology structure present in its homology class.   
\begin{example}
  To get some insights on how the situation can be complex when the state space
  is infinite, let us have a look at the random walk on the triangulation of the
  plane from which we have removed one triangle, so that the chain never dies. After a trillion of iterations,
  we get a graph similar to that of Figure~\ref{fig:Plan} (a).

 \begin{figure}[!ht]
   \centering 
   \begin{tabular}{cc}
   \includegraphics[width=6cm]{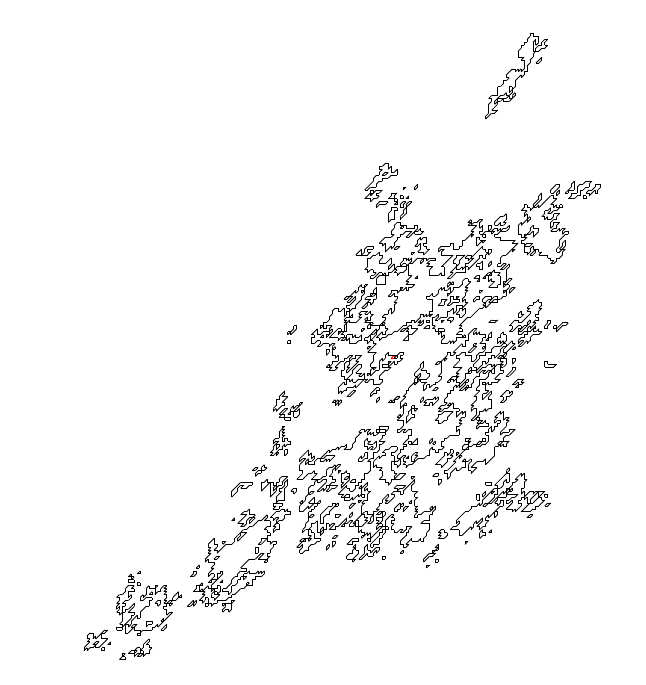} & \includegraphics[width=7cm]{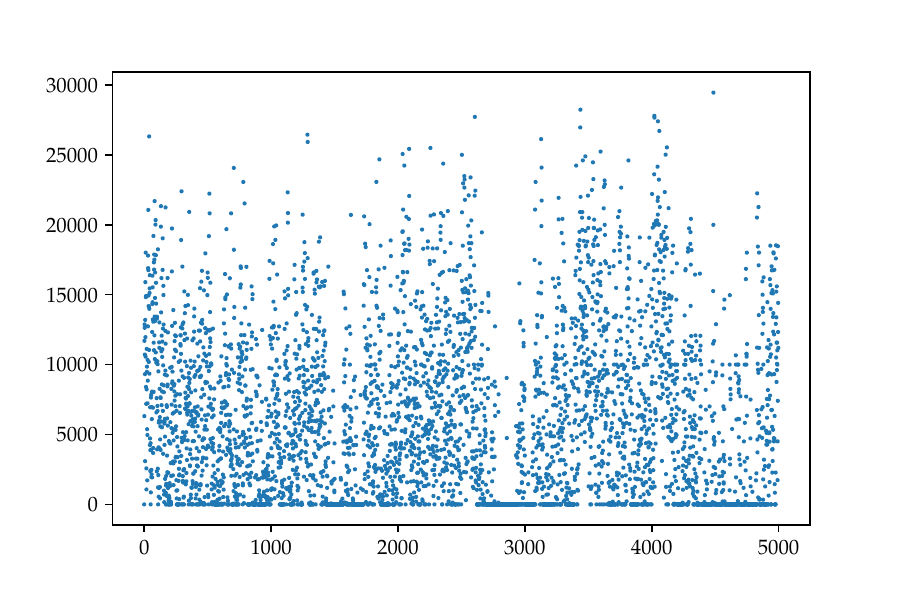}\\
   (a) & (b)
   \end{tabular}
   \caption{{\small \textit{(a) A realization of $X$ after 2 trillions steps. The removed triangle is in red (center of the image). (b) The number of times $X$ touches the removed triangle, by packets of ten thousands steps. Simulation by M. Glisse.}}}
   \label{fig:Plan}
 \end{figure}

 The support of the process $X$ is composed of several disconnected components,
 each of them may contain some holes. The isolated components are going to
 either die or merge with the component which contains the removed triangle.
 Provided that this is meaningful, if we look at the number of triangles which
 are \textsl{inside} the chain, it can increase or decrease by $1$ with equal
 probability at each step. This means that it follows the law of a symmetric
 random walk, which is thus null recurrent. However, when the chain touches the
 boundary of the removed triangle, there is a drift only in the positive sense
 which ruins this reasoning. The simulation represented on
 Figure~\ref{fig:Plan} (b) shows that $X$ touches the removed triangle very often.

\end{example}




\section{Convergence}
\label{sec:conv}

When dealing with random walks, it is natural to investigate their continuous diffusive limits. Here, we choose to embed our graph into another space and consider geometric random graphs (as is done e.g. in \cite{yogeshwaranadler}). Limits of random walks on graphs drawn on a manifold have been considered for instance in \cite{garciatrillosgerlachheinslepcev,Gine,guerinnguyentran,heinaudibertvonluxburg,singer} or \cite{tinghuangjordan}, but the literature deals only with the convergence of generators as the latter is the key for applications in machine learning. In this section, we also study the tightness of the distributions of the random walks and show that the limiting values of the correctly rescaled cycle random walk are the diffusion solutions to a same martingale problem associated with the limit of the combinatorial Laplacian. As we will see in the computation done in this section, things can become quickly intricate, and this is why we focus here on a particular case: we consider the scaling limits of the cycle random walk on the triangulation of the flat torus.\\

We denote by $\torus$ the flat torus, which we embed into
$\R^{2}$ as the rectangle $[0,2]\times [0,\sqrt{3}]$ where the opposite edges are
identified. Let $\epsilon_{n}=1/n$ and consider
\begin{multline*}
  V_{n}=\left\{ (2k\epsilon_{n}, 2l\sqrt{3}\ \epsilon_{n} ), 0\le k <n, 0\leq 2l< n  \right\} \\
  \bigcup \left\{ ((2k+1)\epsilon_{n}, (2l+1)\sqrt{3}\ \epsilon_{n} ), 0< 2k+1 \leq 2n , 0< 2l+1\leq n \right\},
\end{multline*}
the set of vertices of the regular triangulation of mesh $2\epsilon_{n}$ and denote by $\mathcal{T}_n$ the triangulation based on $V_n$, see
Figure~\ref{fig:torus}. We consider the simplicial complex $\mathbf{C}^n$ composed of all the triangles of $\mathcal{T}_n$, their edges and their vertices $V_n$.

\begin{figure}[!ht]
  \centering
 \begin{tikzpicture}[scale=0.5,font=\fontsize{6}{6}\selectfont]
           \fill[color=black] (0,0) circle(3pt) node[below] {$0$};
           \fill[color=black] (2,0) circle(3pt) node[below] {$1$};
           \fill[color=black] (-2,0) circle(3pt);
           \fill[color=black] (1,2) circle(3pt) node[below] {$2$};
           \fill[color=black] (-1,2) circle(3pt) node[below] {$4$};
            \fill[color=black] (-1,-2) circle(3pt);
            \fill[color=black] (1,-2) circle(3pt) node[below] {$3$};
            \draw (-2,0) -- (2,0);
            \draw (-1,-2) -- (1,2);
            \draw (-1,2) -- (1,-2);
            \draw[<->,color=blue] (-3,2) -- node[above] {$2\epsilon_{n}$} (-1,2);
            \draw[dashed] (-3.5,-2) -- (3.5,-2);
            \draw[dashed] (-3.5,2) -- (3.5,2);
            \draw[dashed] (-3.5,-3) -- (-0.5,3);
            \draw[dashed] (0.5,-3) -- (3.5,3);
            \draw[dashed] (-1.5,-3) -- (1.5,3);
            \draw[dashed] (-3.5,3) -- (-0.5,-3);
            \draw[dashed] (0.5,3) -- (3.5,-3);
            \draw[dashed] (-3.5,0) -- (3.5,0);
            \draw[dashed] (-1.5,3) -- (1.5,-3);
          \end{tikzpicture}
  \caption{Regular triangulation of the flat torus. $  0:=(0,0),\
    1:=(2\epsilon_{n},0),\ 2:=(\epsilon_{n},\sqrt{3}\epsilon_{n}),\,
    3:=(\epsilon_{n},-\sqrt{3}\epsilon_{n}), 4:=(-\epsilon_{n},\sqrt{3}\,\epsilon_{n})$}
  \label{fig:torus}
\end{figure}
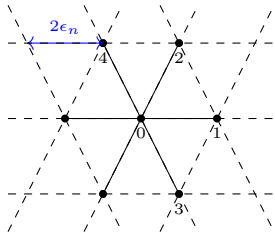

We will study here the cycle random walk $(X^n_t)_{t\geq 0}$ (correctly renormalized, as will be seen in the sequel) on this simplicial complex $\mathbf{C}^n$ and the sequence of these random walks for $n\in \N$.
More precisely, let us denote by $\s_2^n$ the set of triangles with vertices in $V_n$ and edge lengths $2\epsilon_n$, and $\ch_1^n$ the set of chains associated with this triangulation.

The generator of our random walk in \eqref{def:A-rw}, that we reformulate here for the cycle random walk on the torus, is given, for $\sigma\in \ch_1^n$ and for $F$ a continuous and bounded test function from $\ch_1^n$ to $\R$, by:\begin{align}
\mathcal{A}_n F(\sigma)=  \epsilon_n^{-2}\sum_{\tau \in \s_2^n}\Big(F(\sigma-\p_2 \tau)-F(\sigma)\Big)w(\sigma,\p_2 \tau). \label{def:An}
\end{align}A convenient class of test functions will be precised later in \eqref{test_function}.
Let us explain this generator. The rescaling of space is automatic since the number of points on the torus increases and the distance between two connected point is $2\epsilon_n$ which tends to zero. Starting from a cycle $\sigma$, we look at every (oriented) triangle that is adjacent to the cycle. The weight $w(\sigma,\p_2 \tau)$ defined in \eqref{eq_rw:2} indicates how many edges (with multiplicity) are shared by $\eta$ and $\tau$. At rate $w(\sigma,\p_2\tau)$, the triangle $\tau$ (of area in $\epsilon_n^2$) is chosen and its boundary is subtracted to the cycle, entailing the deletion of the shared edges and the inclusion of the new other edges so that the path remains a cycle. \\
Time is accelerated in $\epsilon_n^{-2}$, which can be understood from the estimate of Proposition \ref{prop:trousp} below. This explains the $\epsilon_n^{-2}$ at the beginning of the expression of $\mathcal{A}_n$. \\

In this example, $H_1$ is of dimension 2, and with the notation of Theorem \ref{th:rw_remains_simple}, we can choose as a basis of $H_1$:
\begin{align}
\sigma_1= & \sum_{k=0}^{n-2} [(2k\epsilon_n,0),(2(k+1)\epsilon_n,0)], \qquad \sigma_2=  \sum_{l=0}^{n-2} [(l\epsilon_n,l\sqrt{3}\epsilon_n),((l+1)\epsilon_n,(l+1)\sqrt{3}\epsilon_n)].\nonumber\end{align}
Notice that if the initial condition $X^n_0$ is of the form \eqref{ecriture:decomposition} with $\mu_1>0$ or $\mu_2>0$, then the condition \eqref{hyp:th10} is not satisfied and the state space might be infinite (even if the set of vertices is finite), since a cycle can ``loop'' on itself and the edge weights can be more than 1 as seen on Fig. \ref{fig:torus-contreexemple}.

\subsection{Spectral properties of $L_1^\up$}\label{sec:cv-eigens}

Recall that $\mathbf{C}^n$ is simplicial complex that is explored and $\s_1(\mathbf{C}^n)$ is its set of edges. 
Before studying the cycle random walk on the triangulation of the $2$-dimensional torus, let us give some insight into the spectral properties of $L_{1}^\up$. To do this, we use the simple edge-valued random walk jumping between upper adjacent edges similar to the one defined in Parzanchevski and Rosenthal \cite{parzanchevskirosenthal, rosenthal}. The transition kernel of this random walk is 
\[
\forall e,e' \in \s_1(\mathbf{C}^n), K(e,e') = 
  \begin{cases}
    \frac{1}{5} &\text{ if $e = -e'$} \\
    \frac{1}{5} &\text{ if $e \neq e'$ and $e$ is upper adjacent to $-e'$} \\
    0&\text{ otherwise.}
  \end{cases}
\]
Or, equivalently,
\[
\forall e,e' \in \s_1(\mathbf{C}^n), K(e,e') = \frac{1}{5} \langle (I - L_1^\up) e^*, e'\rangle_{\ch^k,\ch_k}
\]
Using standard analysis tool on this random walk gives the following result.
\begin{proposition}\label{prop:trousp}
Let $\lambda_m^n$ be the smallest non-zero eigenvalue of the $1$-dimensional up-Laplacian on the simplicial complex $\mathbf{C}^n$. We have 
\[
    \lambda_m^n = \mathcal{O}(\epsilon_n^2).
\]
\end{proposition}

\begin{proof}
By Corollary 1 in Diaconis and Stroock \cite{DiaconisStroock}, we have that the smallest non-zero eigenvalue of $L_{1}^\up$, denoted $\lambda_m^n$, is bounded from below by 
\[
\lambda_m^n \geq \frac{2|\s_1^+(\mathbf{C}^n)|}{25 d^n b^n},
\]
where $d^n$ is the diameter of the state space and $b^n$ is the maximum number of geodesic paths going through a given edge. Since we are triangulating the $2$-dimensional torus, we have that 
\[
|\s_1^+(\mathbf{C}^n)| = \mathcal{O}(\epsilon_n^{-2}) \qquad \mbox{ and }\qquad 
d^n = \mathcal{O}(\epsilon_n^{-1}).
\]
In order to bound $b^n$, let us note that the total number of geodesic path is $|\s_1^+(\mathbf{C}^n)|^2$ while their maximum length is $d^n$. Thus, the sum of the lengths of all geodesic paths is of order $d^n|\s_1^+(\mathbf{C}^n)|^2$. Finally, since our triangulation is symmetric the geodesic paths must be distributed among all edges, therefore 
\[
b^n = \mathcal{O}\left(\frac{d^n |\s_1^+(\mathbf{C}^n)|^2}{|\s_1^+(\mathbf{C}^n)|}\right) = \mathcal{O}(\epsilon_n^{-3}),
\]
which concludes the proof.
\end{proof}

The result of Proposition \ref{prop:trousp} gives us indications on the rescaling needed for our random walk to converge to some continuous limit. Next, we investigate the state space in which the convergence should be established.

\subsection{From chains to currents}

In the same way that $V_n$ was embedded into $\torus$, we need to embed the spaces of 1-chains into a bigger continuous space. The natural proposition is to consider chains as \textit{currents}, which are duals of linear forms on the torus. We first recall some notation and refer to \cite[Chapters 5 and 6]{bergergostiaux}, especially on differential forms and their integrals.
Recall also that we focus in this section on 1-chains and thus will not need to introduce $k$-forms in their whole generality, but only forms for $k\in \{0,1,2\}$. Denote by $T^*_x \torus$ the tangent vector space on the torus at $x$. The spaces $\Lambda^k T^*_x\torus$ are the spaces of alternating $k$-linear forms on $T^*_x \torus$. Note that $\Lambda^0 T^*_x\torus=\R$ (so that $0$-forms are continuous real functions on the torus), $\Lambda^1 T^*_x \torus=T^*_x \torus$. For $k=2$, $\Lambda^2 T^*_x\torus$ is the set of bilinear real functions $\phi$ on $T^*_x\torus$, such that $\forall u_1,u_2\in T^*_x\torus, \phi(u_1,u_2)=-\phi(u_2,u_1)$. It is then classical to define by $T^*\torus=\cup_{x\in \torus}T^*_x\torus$ the tangent bundle of the torus, and to let $\Lambda^k T^* \torus = \cup_{x\in \torus} \Lambda^k T^*_x \torus$ be the disjoint union of the spaces $\Lambda^k T^*_x \torus$.

 We denote by $\Cf^{k,p}(\torus)=\Co^p(\torus,\Lambda^k T^*\torus)$ the set of differential $k$-form of class $\Co^p$ on $\torus$. The space of continuous 1-differential forms is denote by $\Cf^1(\torus)$ and $\Cf^{1,p}(\torus)$ is dense in $\Cf^1(\torus)$. Notice that $\Cf^1(\torus)$ is separable (see \cite[Th.5.1.5 P.147]{bergergostiaux}).


\begin{definition}\label{def:canonicalbasisdiffform}
Let us denote by $dx_1$ and $dx_2$ the canonical basis of the space of linear forms on $T^*\torus$. For a continuous 1-differential form $\phi\in \Cf^{1}(\torus)$, we can write
   \begin{equation}\label{def:1-diff-form}
    \phi = \phi^{1}\dif x_{1} + \phi^{2}\dif x_{2},
  \end{equation}
  where $\phi^{1}$ and $\phi^{2}$ are continuous functions on the
  torus, i.e. they can be viewed as the restriction over $[0,1]^{2}$ of continuous $(1,1)$-periodic
  functions:
  \begin{equation*}
    \phi(x_{1}+l_{1},x_{2}+l_{2})= \phi(x_{1},x_{2})
  \end{equation*}
  for any pair of integers $(l_{1},l_{2})\in \mathbb{Z}^2$.
  We set
  \begin{equation*}
    \|\phi\|_{\Cf^{1}}=\|\phi^{1}\|_{\infty }+\|\phi^{2}\|_{\infty}.
  \end{equation*}
It topological dual is the set of currents, denoted by $\Cf_1$. It inherits the
Banach norm:
 \begin{equation}\label{def:normC}
   \|p\|_{\Cf_1}=\sup_{\phi\in \Cf^1}\frac{|\langle p,\phi \rangle_{\Cf_1,\Cf^1}|}{\|\phi\|_{\Cf^1}}  ~ .
 \end{equation}\end{definition}
By the previous definition and the preceding remarks, $(\Cf_1,\|.\|_{\Cf_1})$ is a Polish space. \\
Recall that for two functions $\phi$ and $\varphi\in T^*_x\torus$, $\phi\wedge \varphi\in \Lambda^2 T^*_x\torus$ is defined for all $u_1,u_2\in T^*_x \torus$ by:
 \begin{equation}
 \phi\wedge \varphi(u_1,u_2)=\phi(u_1)\varphi(u_2)-\varphi(u_1)\phi(u_2).
 \end{equation}
In particular, the 2-form $\dif x_1\wedge \dif x_2$ is a \textit{volume form} on $\torus$ that is canonically associated with the Lebesgue measure on $\torus$ and allows to integrate differential forms (see Chapter 6 of \cite{bergergostiaux}, for example).

As a particular case of application, let us mention the following useful definition:

\begin{definition}\label{def:currents}
 We denote by $\path$, the set  of  \textsl{paths}, i.e. the piecewise
differentiable  maps   from $[0,1]$ into
$\torus$. This set can be viewed as a subspace of $\Cf^1$. 
For $\phi\in \Cf_1$, the curvilinear integral of $\phi$ along an element $p\in
\path$ is a linear map and we denote by $\langle p,\phi\rangle$ the integral $\int_p \phi$, with 
\begin{equation*}
 \left| \langle p,\phi\rangle \right|= \left| \int_{p}\phi \right|\le \|\phi\|_{\Cf_1}\|p\|_{\Cf^1}.
\end{equation*}
\end{definition}

The definition \ref{def:currents} allows us to see 1-chains as paths by embedding the abstract graph into a geometric graph on the torus. \\

In particular, for a cycle $\sigma$ constructed on $\mathcal{T}_n$, 
\begin{equation}\label{estimate:intpath}
 \left| \langle \sigma,\phi\rangle_{\Cf_{1},\Cf^{1}} \right| = \left| \int_{\sigma}\phi \right| \le \|\phi\|_{\Cf^1} \ 2 \epsilon_n \sum_{e\in \mathcal{T}_n} |\lambda_e(\sigma)|,
\end{equation}with the notation of \eqref{eq_preliminaries:2}. Since the $\lambda_e(\sigma)$ are integers, the right hand side is upper-bounded by $\|\sigma\|_{\Co_1}^2$ (in discrete norm). As a consequence:
\begin{equation}
    \|\sigma\|_{\Cf_1}\leq 2 \epsilon_n  \|\sigma\|^2_{\Co_1}.
\label{estimate:intpath2}
\end{equation}Notice that when $\lambda_e\in \{-1,0,1\}$, the right hand side of \eqref{estimate:intpath} is equal to $2\epsilon_n \|\sigma\|_{\Co_1}^2$, which is the metric length of the chain $\sigma$.


Keeping in mind Donsker's theorem, we expect that the length of the diffusive limit cycle is infinite in the same way as the standard Brownian motion has infinite variation. Therefore, we will also require for what follows additional regularity assumptions on the differential forms.

 Recall that $\Cf^{1,p}$ denotes the set of $1$-differential forms of class $\mathcal{C}^p$, i.e. for which the functions $\phi^1$ and $\phi^2$ in \eqref{def:1-diff-form} are $\mathcal{C}^p$. In the sequel, we will consider $p\in \{1,2,3\}$. Since we are on the compact torus, these derivatives are bounded.
\begin{definition}Let us embed $\Cf^{1,3}$ with the norm:
\begin{multline*}
\|\phi\|_{\Cf^{1,3}}=\|\phi\|_{\Cf^{1}}
+ \sum_{i=1}^2  \Big(\big\|\frac{\partial \phi^i}{\partial x_1}\big\|_\infty+\big\|\frac{\partial \phi^i}{\partial x_2}\big\|_\infty \\ 
+\big\|\frac{\partial^2 \phi^i}{\partial x_1^2}\big\|_\infty+\big\|\frac{\partial^2 \phi^i}{\partial x_2^2}\big\|_\infty+\big\|\frac{\partial^2 \phi^i}{\partial x_1 \partial x_2}\big\|_\infty
+\big\|\frac{\partial^3 \phi^i}{\partial x_1^3}\big\|_\infty+ \big\|\frac{\partial^3 \phi^i}{\partial x_1^2 \partial x_2}\big\|_\infty+\big\|\frac{\partial^3 \phi^i}{\partial x_1 \partial x_2^2}\big\|_\infty+ \big\|\frac{\partial^3 \phi^i}{\partial x_2^3}\big\|_\infty\Big).\end{multline*}
Its topological dual is denoted by $\Cf_{1,3}$ and embedded with the Banach norm:
\begin{equation}\label{def:normCf12}
    \|\sigma\|_{\Cf_{1,3}}=\sup_{\phi\in \Cf^{1,3}}\frac{|\langle \sigma,\phi\rangle_{\Cf_{1,3},\Cf^{1,3}}|}{\|\phi\|_{\Cf^{1,3}}}.
\end{equation}
\end{definition}

The space $(\Cf_{1,3},\|.\|_{\Cf_{1,3}})$ is Polish and contains $\Cf_1$ and $\path$. Remark that if $\sigma\in \Cf_1$, then the bracket in the numerator of the right hand side of \eqref{def:normCf12} is equal to $\langle \sigma,\phi\rangle_{\Cf_1,\Cf^1}$. In the sequel, we are interested in limit theorems for the sequence of cycle random walks viewed as a $\Cf_{1,3}$-valued random process.\\

A function $F\,:\, \Cf_{1,3}\to \R$ is said to be cylindrical if it is of the
 form:
 \begin{equation}\label{test_function}
   F(\sigma)=f\Bigl( \langle \sigma,\phi_1 \rangle_{\Cf_{1,3},\Cf^{1,3}},\cdots,  \langle \sigma,\phi_{k} \rangle_{\Cf_{1,3},\Cf^{1,3}}\Bigr)
 \end{equation}
 for some $\phi_{1},\cdots, \phi_{k} \in \Cf^{1,3}$ and $f$ a continuous real function of class $\mathcal{C}^1$  with compact
 support from $\R^{k}$ into $\R$.
 Let $\mathcal{B}$ be the Banach space of bounded continuous functions from $\Cf_{1,3}$ to $\R$ equipped
 with the sup-norm. The set of cylindrical functions is a set of continuous functions on $\Cf_{1,3}$ that is closed under addition and separates points $\Cf_{1,3}$. For the latter point, we can notice with arguments similar to \cite[Lemma A.1]{meleardtran_suphist} that the topology generated on $\Cf_{1,3}$ by these functions is the same as the one of finite-dimensional convergence. These properties will be useful in the next section, when establishing the convergence of the cycle-random walk on the triangulation of the torus supported on $V_n$.

\subsection{Hodge operator}

For $0\leq k\leq 1$ (since we work on the $2d$-torus) and $p\geq 1$, there exists a unique linear operator $\dd$ from $\Cf^{k,p}(\torus)$ into $\Cf^{k+1,p-1}(\torus)$, called exterior differentiation (see \cite[Prop. 5.2.9.1]{bergergostiaux}), such that:\\ (i) $\dd\circ \dd=0$ \\
(ii) for $k$ and $k'$-forms $\phi$ and $\varphi$, $\dd(\phi\wedge \varphi)=\dd \phi \wedge \varphi+(-1)^{k} \phi\wedge \dd\varphi$,\\
(iii) for $\phi\in \Co^1(\torus,\R)$, $\dd \phi\ :\ X\mapsto T^* X$ is the derivative of $\phi$:
\begin{equation}\label{def:derivative}\dd\phi (x)=\frac{\partial \phi}{\partial x_1}(x) dx_1 + \frac{\partial \phi}{\partial x_2}(x) dx_2.\end{equation}

For $\phi\in \Cf^{1,1}(\torus)$, there exists two $\Co^1(\torus,\R)$ functions $\phi_1$ and $\phi_2$ such that $\phi(x)=\phi^1(x) \ dx_1+\phi^2(x) \ dx_2$. The exterior derivative of $\phi$, $\dd\phi \in \Cf^{2,0}(\torus)$, is then:
  \begin{align}
    d\phi= & \dd \phi^1 \wedge dx_1 + \dd\phi^2 \wedge dx_2 \nonumber\\
    = & \frac{\partial \phi^1}{\partial x_2} dx_2 \wedge dx_1 + \frac{\partial \phi^2}{\partial x_1} dx_1\wedge dx_2 \nonumber\\
    = & \Big(\frac{\p \phi^{2}}{\p x_{1}}(x)-\frac{\p \phi^{1} }{\p x_{2}}(x)\Big) dx_1\wedge dx_2.\label{etape2}
  \end{align}

\begin{definition}\label{def:Hodge}
  The \textit{Hodge transform} of forms is the linear transformation defined by its
  action on a basis of differential forms:
  \begin{equation*}
    *1=\dif{x_1}\wedge\dif {x_2}, \ *\dif{x_1}=\dif{x_2},\ *\dif {x_2}=-\dif{x_1},\ *(\dif{x_1}\wedge\dif{x_2})=1.
  \end{equation*}
  \end{definition}

  \begin{definition}
 The Rham-Hodge operator is then defined by
  \begin{align*}
    \L =\L^{\up}+\L^{\down} \text{ where }
    \L^{\up}=*\dd*\dd \text{ and }
    \L^{\down}=\dd*\dd*.
  \end{align*}
\end{definition}

\begin{proposition}
Let $\phi\in \Cf^{1,2}(\torus)$, with $\phi=\phi^1 \ dx_1+\phi^2 \ dx_2$.  We have 
\begin{align}
  & \L^{\up}\Bigl( \phi \Bigr)=\Bigl( \phi^{1}_{22}-\phi^{2}_{12} \Bigr) \, d{x_1} +\Bigl( \phi^{2}_{11}-\phi^{1}_{12} \Bigr)\, d{x_2}\label{eq_convergence:4} \\
  & \L^{\down}\Bigl( \phi \Bigr)=\Bigl( \phi^{1}_{11}+\phi^{2}_{12} \Bigr) \, d{x_1} +\Bigl( \phi^{1}_{12}+\phi^{2}_{22} \Bigr)\, d{x_2}.
\end{align}
where $f_{i}$ is a shortcut for the partial derivative of $f$ with respect to
the variable~$x_{i}$.
\end{proposition}

\begin{proof}Using Def. \ref{def:Hodge} and \eqref{etape2},
\begin{align}
*\dd \phi= \frac{\p \phi^{2}}{\p x_{1}}-\frac{\p \phi^{1} }{\p x_{2}}. \label{def:difphi}
\end{align}Then, by \eqref{def:derivative},
\begin{align*}
\dd * \dd \phi= & \Big(\frac{\p^2 \phi^{2}}{\p x^2_{1}}-\frac{\p^2 \phi^{1} }{\p x_1 \p x_{2}}\Big) \ dx_1 + \Big(\frac{\p^2 \phi^{2}}{\p x_{1} \p x_2}-\frac{\p^2 \phi^{1} }{\p x^2_{2}}\Big)\ dx_2\\
= & \big(\phi^2_{11}-\phi^1_{12}\big) \ dx_1 + \big(\phi^2_{12}-\phi^1_{22}\big)\ dx_2
\end{align*}Using again Def. \ref{def:Hodge} gives \eqref{eq_convergence:4}. Proceeding similarly provides the expression of $\L^\down \phi$.
\end{proof}

Notice that the Rham-Hodge operator appears also as the Witten Laplacian in the literature for $0$-forms (e.g. \cite{malliavin1974,witten}), which relates to the generator of diffusions on manifolds \cite{bakry1987,elworthylejanli,li2008,vanneervenversendaal} -- in our case, the Brownian motion on the torus $\mathbb{T}_2$. For $1$-forms, \cite{elworthylejanli,li2008} provide probabilistic reprensentations (different from ours) of the semigroups associated with the Rham-Hodge operator by mean of Feynman-Kac formulas and using the diffusions obtained for $0$-forms.

\begin{proposition}\label{prop:Lup_limite}
  The operator $\L^{\up}$ is closable, dissipative and there exists $\lambda >0$ such that  $\lambda
\Id -\L^{\up}$ is one-to-one. As a consequence, $\L^\up$, whose domain contains $\mathfrak C^{1,2}$, generates a strongly continuous contraction semi-group. Moreover, the martingale problem associated to $\L^\up$ is well posed.
\end{proposition}
\begin{proof}
First, let us prove that $\L^{\up}$ is closable,
dissipative and that there exists $\lambda >0$ such that  $\lambda
\Id -\L^{\up}$ is one-to-one.\\

By \eqref{eq_convergence:4}, for $\phi\in \Cf^{1,2}$,
  \begin{equation*}
    \langle \L^{\up}\phi,\, \phi \rangle=\int_{\torus} \frac{\partial}{\p x_{2}}\Bigl( \phi^{1}_{2}-\phi^{2}_{1} \Bigr)\ \phi^{1}\dif x_{1}\dif x_{2}-\int_{\torus} \frac{\partial}{\p x_{1}}\Bigl( \phi^{1}_{2}-\phi^{2}_{1} \Bigr)\ \phi^{2}\dif x_{1}\dif x_{2} .
  \end{equation*}
  By integration by parts, taking into account the periodicity of $\phi^{1}$ and
  $\phi^{2}$, we get
  \begin{align*}
        \langle \L^{\up}\phi,\, \phi \rangle= & -\int_{\torus}  \Bigl(
        \phi^{1}_{2}-\phi^{2}_{1} \Bigr) \phi^{1}_{2}\dif x_{1}\dif
        x_{2}+\int_{\torus} \Bigl( \phi^{1}_{2}-\phi^{2}_{1} \Bigr)\
        \phi^{2}_{1}\dif x_{1}\dif x_{2}\\
        =& -\int_{\torus} \Bigl(
        \phi^{1}_{2}-\phi^{2}_{1} \Bigr)^{2}\dif x_{1}\dif x_{2}.
      \end{align*}
      This means that $\L^{\up }$ is symmetric and negative (hence dissipative).
      Integrating by parts a second time yields, for any $\psi\in \mathfrak C^{1,2}$,
      \begin{equation*}
        \langle \L^{\up}\phi,\, \psi \rangle=\langle \phi,\, \L^{\up}\psi \rangle.
      \end{equation*}
      Hence, if $\phi_{n}\to 0$ and $\L^{\up}\phi_{n}\to \eta$, we get $\langle
      \eta,\, \psi \rangle=0$ for any $\psi\in \mathfrak C^{1,2}$. By density, this
      entails $\eta=0$. Hence, $\L^{\up}$ is closable. We still denote  by
      $\L^{\up}$ its extension, whose domain $\dom(\L^{\up})$ contains at least $\mathfrak C^{1,2}$.

      Consider the basis of $L^{2}(\torus, \mathbf C)$ given by
      \begin{equation*}
        e_{n,m}(x_{1},x_{2})=e^{2i\pi n x_{1}}e^{2i\pi m x_{2}}, \ n,m\in \mathbf Z.
      \end{equation*}
      For $i=1,2$, we have in $L^2(\torus,\mathbf C)$,
      \begin{align*}
        \phi^{i}=\sum_{n,m\in \mathbf Z} c_{n,m}^{i} \, e_{n,m}.
      \end{align*}
      Furthermore,
      \begin{align*}
        \phi^{1}_{22}-\phi^{2}_{12}&=-4\pi^{2}\sum_{n,m\in \mathbf Z} (c_{n,m}^{1}m^{2}-c_{n,m}^{2}mn) \, e_{n,m},\\
        \phi^{2}_{11}-\phi^{1}_{12}&=-4\pi^{2}\sum_{n,m\in \mathbf Z} (c_{n,m}^{2}n^{2}-c_{n,m}^{1}mn) \, e_{n,m}.
      \end{align*}
      Thus solving $\L^{\up}\phi=-4\pi^{2}\lambda \phi$ amounts to find the
      $c^{i}_{n,m}$'s such that
      \begin{align*}
      c_{n,m}^{1}(m^{2}-\lambda)-c_{n,m}^{2}mn&=0,\\
  -c_{n,m}^{1}mn+c_{n,m}^{2}(n^{2}-\lambda)&=0.
      \end{align*}
      For $\lambda$ negative irrational, this system admits the null form as
      unique solution, hence for such a $\lambda$, $\L^{\up}-4\pi^{2}\lambda \Id$ is
      one-to-one  and the third condition is satisfied.\\
      
      According to \cite[Th.2.12 P.16]{ethierkurtz}, the operator $\L^\up$ hence generates a strongly continuous contraction semi-group on $\Cf^{1,3}$.
      Moreover, by \cite[Th.4.1 P.182]{ethierkurtz}, uniqueness holds for the martingale problem associated with $\L^\up$.
    \end{proof}


\subsection{Cycle random walk on the torus}\label{sec:cycletorus}

\subsubsection{Rescaled random walk on cycles} Consider the random walk $X^n$ with generator \eqref{def:An}. It is a continuous-time jump process with values in $\ch_1^n\subset \Cf_{1,4}$ (the latter space does not depend on $n$), which we recall is a Polish space. The set $\mathbb{D}(\R_+,\Cf_{1,4})$ is embedded with the Skorokhod distance and is itself Polish (e.g. \cite{jakubowski}). To prove the convergence, we will use the explicit expression of $X^n$ as the solution of a stochastic differential equation driven by Poisson point measures. Let $N_n(ds,d\theta,d\tau)$ be a Poisson point measure on $\R_+\times \R_+ \times \s^n_2$ with intensity measure $ds\otimes d\theta \otimes n(d\tau)$ where $n(d\tau)$ is the counting measure on $\s^n_2$.
\begin{align}
X^n_{t} = & X^n_0 - \int_0^{t} \int_{\R_+}\int_{\s^n_2} \p_2 \tau \  \ind_{\{\theta\leq  w(X^n_{s_-},\p_2\tau)\}} N_n(ds,d\theta,d\tau).\label{eq:cycle-rw-Xn}
\end{align}In the sequel, we will study the convergence of the sequence of accelerated processes:
\begin{equation}
X^{(n)}_t=X^n_{\epsilon_n^{-2}t}.
\end{equation}

\begin{figure}[!ht]
  \begin{center}
\includegraphics[width=5.5cm,height=4.5cm]{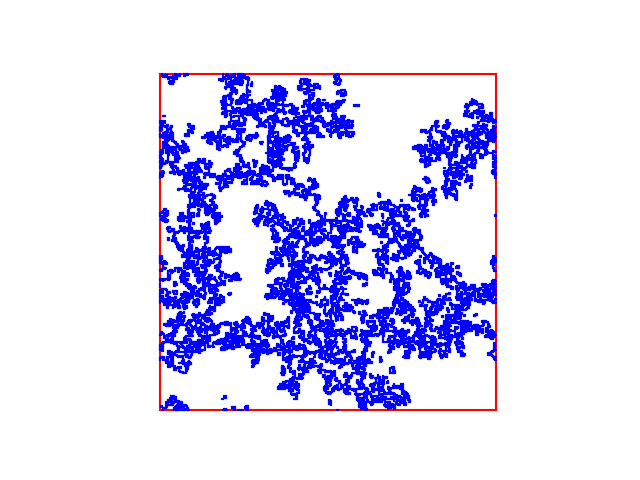} 
\caption{{\small Simulations of the cycle-valued random walk on the simplicial complex composed of all the triangles of $\mathcal{T}_n$. Simulation by Paul Melotti.}}
  \label{fig:conjecture}
\end{center}
\end{figure}

\subsubsection{Conjectures}

As the random cycle at any time $t$ can have loops and use several times the same edge, one of the main difficulty in studying limit theorems for $(X^n_t)_{t\geq 0}$ lies in controlling its size. Indeed, for the usual convergence of random walks to the Brownian motion, we expect that the number of edges and the length of the curve tend to infinity. \\
We instead consider the flat norm of $X\in \Cf_1$ (see \cite[p.4]{federer}) that is defined as:
\begin{equation}\label{def:flat}
\|X\|_F:=\inf_{\Delta\in  \Cf_2} \|X-\partial_2 \Delta\|_{\Cf_1}+\|\Delta\|_{\Cf_2}.
\end{equation}Roughly speaking, the infimum in the definition of the flat norm is over all the $\Delta\in \Cf_2$ such that $\partial_2 \Delta$ is close to $X$ and that have a small area $\|\Delta\|_{\Cf_2}$. In the following, we will need the following conjecture:

\begin{conjecture}\label{conjecture-moments-discrets}
For all $T>0$, we require the following control on the flat norm of our process:
\begin{equation}
    \lim_{n\rightarrow +\infty} \E\Big(\sup_{t\leq T}  \|X^{(n)}_t\|_F\Big)<+\infty\label{conj2}
\end{equation}
\end{conjecture}

Notice that with some modification of the support of the cycle-valued chain, we can show easily that the conjecture holds. If we introduce
\[\sigma_0=\big[(0,0),(2\epsilon_n,0)\big]+\big[(2\epsilon_n,0),(\epsilon_n,\sqrt{3}\epsilon_n)\big]+\big[(\epsilon_n,\sqrt{3}\epsilon_n),(0,0)\big],\mbox{ and } \tau_0=[(0,0),(2\epsilon_n,0),(\epsilon_n,\sqrt{3}\epsilon_n)],\]
then, for the simplicial complex $\widetilde{\mathbf{C}}_n := \mathbf{C}_n \setminus \tau_0$, $H_1$ is of dimension 3 with the basis $\{\sigma_0,\sigma_1,\sigma_2\}$ and a chain starting in the homology class of $\sigma_0$ satisfies the condition \eqref{hyp:th10}. This ensures that $\Delta^n_t$ can be chosen as a combination of triangles with weights in $\{-1,0,1\}$ and also, it is possible to choose $\Delta^n_t$ such that $X^{(n)}_t-\partial_2 \Delta^n_t=X^{(n)}_0$. Therefore 
\[\|\Delta^n_t\|_{\Cf_2}\leq 2\sqrt{3},\]
which is the area of $\torus$. 
In this case, \eqref{conj2} is satisfied as soon as
\begin{equation}\E\big(\|X^{(n)}_0\|_{\Cf_1}\big)<+\infty,\end{equation}a condition that depends only on the initial condition.
However, rather than working on such a perforated torus we chose to stick with the random walk on the complete torus, since it is a more natural space, at the price of Conjecture \ref{conjecture-moments-discrets}. Another possible direction to avoid the Conjecture would be to consider chains with coefficients in $\Z/2\Z$.\\

As we mentioned earlier, we expect the length of our process to explode as $n$ goes to infinity. Still, our control on the flat norm of the process can be used to obtain an (exploding) bound on its length.
\begin{lemma}
\label{lem:flattolength}
For any $1$-chain $\sigma$ of the simplicial complex $\mathbf{C}_n$, we have 
\[
    \|\sigma\|_{\Cf_1} \leq \frac{2 \sqrt{3} \|\sigma\|_F}{\epsilon_n } .
\]
\end{lemma}

\begin{proof}
    By the definition of the flat norm, there must exists a $1$-cycle $\sigma_0 \in \Cf_1$ and $\Delta \in \Cf_2$ such that $    \sigma = \sigma_0 + \partial_{2} \Delta $
    and 
    \[
    \|\sigma\|_F = \|\sigma_0\|_{\Cf_1} + \|\Delta\|_{\Cf_2}.
    \]
    Since $\sigma$ is a $1$-chain of the simplicial complex $\mathbf{C}^n$, $\sigma_0$ is also a $1$-chain of $\mathbf{C}^n$ and $\Delta$ a $2$-chain of $\mathbf{C}^n$.
    We thus have $\Delta = \sum_{\tau \in \s_2^n} \lambda_\tau \tau$, where $\forall \tau \in \s_2^n, \lambda_\tau \in \mathbb{R}$ and \[
     \|\Delta\|_{\Cf_2} =  \sqrt{3}\epsilon_n^2 \sum_{\tau \in \s_2^n} |\lambda_\tau|.
    \]
    Therefore, 
    \[
    \|\partial_{2} \Delta\|_{\Cf_1} \leq \sum_{\tau \in \s_2^n} |\lambda_\tau| \|\partial_{2} \tau\|_{\Cf_1}  = 6 \epsilon_n \sum_{\tau \in \s_2^n} |\lambda_\tau| = \frac{2 \sqrt{3} \|\Delta\|_{\Cf_2}}{\epsilon_n }
    \]
    and
    \[
    \|\sigma\|_{\Cf_1} \leq  \|\sigma_0\|_{\Cf_1} + \|\partial_{2} \Delta\|_{\Cf_1} \leq \frac{2 \sqrt{3} (\|\sigma_0\|_{\Cf_1} + \|\Delta\|_{\Cf_2})}{\epsilon_n },
    \]
    which concludes the proof. 
\end{proof}
As a consequence of this result, Conjecture \ref{conjecture-moments-discrets} implies that  
\begin{equation}
     \lim_{n\rightarrow +\infty} \E\Big(\sup_{t\leq T}  \epsilon_n \|X^{(n)}_t\|_{\Cf_1}\Big)<+\infty.\label{conj1}.
\end{equation}
We can also control the norms $\|.\|_{\Cf_{1,4}}$ under the Conjecture \ref{conjecture-moments-discrets} and obtain non-exploding bounds.
\begin{proposition}\label{prop:flat}
Under the Conjecture \ref{conjecture-moments-discrets}, we have for any $T>0$
\begin{equation}
     \sup_{n\in \N}\E\big(\sup_{t\leq T}\|X^{(n)}_t\|_{\Cf_{1,4}}\big)<+\infty.
\end{equation}
\end{proposition}

\begin{proof}
Let $n\in \N$ and let $\Delta^n\in \ch^n_2$ be a chain of triangles. We have:
\begin{align}
     \|X^{(n)}_t\|_{\Cf_{1,4}}= & \sup_{f\in \Cf^{1,4}}\frac{\langle X^{(n)}_t,f\rangle_{\Cf_{1,4},\Cf^{1,4}}}{\|f\|_{\Cf^{1,4}}}\nonumber\\
     \leq & \sup_{f\in \Cf^{1,4}}\frac{\langle X^{(n)}_t-\partial_2 \Delta^n,f\rangle_{\Cf_{1,4},\Cf^{1,4}}}{\|f\|_{\Cf^{1,4}}}+\sup_{f\in \Cf^{1,4}}\frac{\langle \Delta^n, \partial^*_2 f\rangle_{\Cf_{2,3},\Cf^{2,3}}}{\|f\|_{\Cf^{1,4}}}\nonumber\\
     \leq & \|X^{(n)}_t-\partial_2 \Delta^n\|_{\Cf_{1,4}}+ \|\Delta^n\|_{\Cf_{2,3}}\label{flat}
\end{align}by using that 
\begin{equation}
    \|\partial_2^* f \|_{\Cf^{2,3}}\leq  \|f\|_{\Cf^{1,4}},\label{consequence-etape2}
\end{equation}by \eqref{etape2}. Taking the infinimum in the right hand side with respect to $\Delta^n$, we obtain the flat norm $\|X^{(n)}_t\|_F$ and we conclude with the Conjecture \ref{conjecture-moments-discrets}. 
\end{proof}


\subsubsection{Main result}

Let us first introduce a notation. For a sequence $\sigma\in \Cf_1$, let  $\widetilde{\sigma}$ be the measure defined for any $\psi\in \Co(\torus,\R)$ and any sequence $(\sigma_n)_{n\geq 1}$ converging to $\sigma$ such that $\limsup_{n\rightarrow +\infty}\epsilon_n \|\sigma_n\|_{\Cf_1}<+\infty$, by:
\begin{equation}\langle \widetilde{\sigma},\psi\rangle= \lim_{n\rightarrow +\infty} \epsilon_n \langle \sigma_n,\psi\rangle.\label{def:sigma-tilde}\end{equation}
This measure is a kind of uniform measure along the cycle $\sigma$ (whose length is possible infinite).\\

The main result of this section is that: 
\begin{theorem}\label{th:cv-cyclerw}
Assume that the Conjecture \ref{conjecture-moments-discrets} holds and that the initial conditions $(X^{(n)}_0)_{n\geq 1}$ converge in probability and in $\Cf_{1,4}$ to a limit $X_0$ and are such that:
\begin{equation}
    \sup_{n\in \N}\E\big(\|X^{(n)}_0\|_{\Cf_{1,4}}\big)<+\infty.
\end{equation}
The limiting values $(X_t)_{t\in \R_+}$ in $\D(\R_+,\Cf_{1,4})$ of the sequence of random walks $(X^{(n)})_{n\geq 1}$ all solve the martingale problem in $\mathcal{C}(\R_+,\Cf_{1,4})$ associated with the generator $\mathfrak{A}$, defined for functions $F_\phi\, :\, \sigma\in \Cf_{1,4}\mapsto F(\langle \sigma,\phi\rangle_{\Cf_{1,4},\Cf^{1,4}})$ by:
\begin{equation}\label{def:mathfrakA}
    \mathfrak{A}F_\phi(\sigma)=F'\Big(\int_\sigma \phi\Big) \ \int_\sigma \mathfrak{L}^{\uparrow}\phi + 3 F''\Big(\int_\sigma \phi\Big) \Big\langle \widetilde{\sigma}, \big(* \dif \phi\big)^2\Big\rangle ,
\end{equation}
More precisely, the process:
\begin{equation}\label{pbm}
M_t := F_\phi\big(X_t\big)- F_\phi\big(X_0\big)  - \int_0^t  \mathfrak{A}F_\phi(X_s)ds,
\end{equation}is a square integrable real-valued martingale with predictable quadratic variation process:
\begin{align}\label{crochet}
\langle M\rangle_t= &  \lim_{n\rightarrow +\infty}  \epsilon_n^{-2} \int_0^t \sum_{\tau\in \s^n_2}  \langle \p_2 \tau,\phi\rangle^2\,  \big(F'(\langle X^{(n)}_{u},\phi\rangle \big)^2 w(X^{(n)}_{u},\p_2\tau) \dif u \\
= &  \int_0^t F'_\phi (X_u)^2 \Big(\int_{\torus} \big(\frac{\p \phi^{2}}{\p x_{1}}-\frac{\p \phi^{1} }{\p x_{2}}\big)^2 \widetilde{X}_u(\dif x_1, \dif x_2)\Big)  \dif u ,\nonumber
\end{align}where $\widetilde{X}$ is defined in \eqref{def:sigma-tilde}.
\end{theorem}

Notice that if $F$ is linear (for instance $F(x)=x$), the second term in \eqref{def:mathfrakA} disappears and we recover $\mathfrak{L}^{\uparrow}$ in the first term.

The end of this Section \ref{sec:conv} is devoted to the proof of Theorem \ref{th:cv-cyclerw}. We first show that the generators of these random walks converge in $\Cf^{1,4}$ to the closure of $(\mathfrak{A},\Cf_{1,4})$.  Then, we prove the tightness of the sequence of distributions of the $X^{(n)}$'s on $\Cf_{1,4}$. The uniqueness of the limiting martingale problem solved by the limiting values remains open.

\subsection{Proof of Theorem \ref{th:cv-cyclerw}}\label{sec:proof_cv_rw}

The proof of Theorem \ref{th:cv-cyclerw} is divided into several steps. We first start with some preliminary computation in Section \ref{sec:prelim}. In Section \ref{sec:cv_generators}, we show that the generators $\mathcal{A}_n$ restricted to functions $F_\phi(\sigma)=\langle \sigma,\phi\rangle$ converge to $\mathfrak{L}^\up$. The difficulty here is to control the error carefully so that we can show tightness of the sequence of distributions of $(X^{(n)})$ in Section \ref{sec:tightness} and convergence of the generators $\mathcal{A}_n$ to $\mathfrak{A}$ in Section \ref{section:fin_generateur}.

\subsubsection{Preliminary estimates}\label{sec:prelim} Let us first establish some estimates that will be needed.

\begin{lemma}\label{lemme:phi-triangle}
  Let $\phi \in \Cf^{1,4}$, with $\phi(x)=\phi^1(x) dx_1 + \phi^2(x) dx_2$, and let $ \tau\in \s^n_2$. Without loss of generality, we can assume (by a change of coordinates) that $\tau=[021]$ (see Fig. \ref{fig:torus}).
  Then, 
  \begin{align}
  \langle \partial_2\tau , \phi\rangle_{\Cf_{1,4},\Cf^{1,4}} = &  \sqrt{3}\epsilon_n^2 \big(\phi^2_{1}(0,0)-\phi^1_{2}(0,0)\big)+\sqrt{3} \epsilon_n^3 \big(\phi^2_{11}(0,0)-\phi^1_{12}(0,0)\big)\nonumber\\
   & \hspace{2cm} +\epsilon_n^3 \big(\phi^2_{12}(0,0)-\phi^1_{22}(0,0)\big) +O(\epsilon_n^4)\label{eq:prelimtauphi1}\\
  = & \frac{\sqrt{3}}{2}  \epsilon_n \int_{0}^{2\epsilon_n}   \big(\phi^2_{1}(x_1,0)-\phi^1_{2}(x_1,0)\big)\dif x_1 + O(\epsilon_n^3).\label{eq:prelim_tauphi}
  \end{align}
\end{lemma}

\begin{proof}Using the Stoke's formula (e.g. \cite[Th.6.2.1]{bergergostiaux}):
  \begin{align*}
      \langle \partial_2\tau , \phi\rangle = \int_{\tau}  d\phi = \int_{\tau} \big(\phi^2_1-\phi^1_2\big) dx_1 \wedge dx_2,
  \end{align*}where the exterior derivative of $\phi$, $d\phi$, has been computed in \eqref{etape2}. We then obtain the result using Lemma \ref{lem:H1} in Appendix with $H(x_1,x_2)=\phi_1^2(x_1,x_2)-\phi_2^1(x_1,x_2)$, that is of class $\Co^2$ under our assumption. 
\end{proof}

\begin{lemma}\label{lem:bis1}
Let $H\,:\,\torus\to \R$ be $\Co^{3}$. Let us consider the edge $[0,1]$ and the two adjacent triangles $[021]$ and $[031]$ where the vertices $0$, $1$, $2$ and $3$ have coordinate $(0,0)$, $(\epsilon_n,\sqrt{3}\epsilon_n)$ and $(2\epsilon_n,0)$ (see Figure \ref{fig:torus}). Then, as $\epsilon_{n}$ goes to $0$,
  \begin{align}
    \label{eq:H1}
     \epsilon_n^{-2}\int_{\tau^+}H({x_1},{x_2})\dif {x_2}\dif {x_1} +\epsilon_n^{-2}\int_{\tau^-}H({x_1},{x_2})\dif {x_2}\dif {x_1} 
     = &   \int_0^{2\epsilon_n } 3 c(x_1,0)H_2(x_1,0) \dif x_1 + O(\epsilon_n^3)
     \end{align}
where $c(x)=\min_{x'\in V_n} \left(\frac{\|x-x'\|}{\epsilon_n}\right)^2.$
\end{lemma}

\begin{proof}We have:
\begin{multline*}
    A:= \epsilon_n^{-2}\int_{\tau^+}H({x_1},{x_2})\dif {x_2}\dif {x_1} +\epsilon_n^{-2}\int_{\tau^-}H({x_1},{x_2})\dif {x_2}\dif {x_1} \\
    \begin{aligned}
     = & \epsilon_n^{-2} \int_0^{\epsilon_n} \int_{0}^{\sqrt{3}x_1} \big(H(x_1,x_2)-H(x_1,-x_2)\big)\dif x_2\ \dif x_1 \\
     & + \epsilon_n^{-2} \int_{\epsilon_n}^{2\epsilon_n} \int_0^{\sqrt{3}(2\epsilon_n -x_1)}\big(H(x_1,x_2)-H(x_1,-x_2)\big) \dif x_2\ \dif x_1.
     \end{aligned}
\end{multline*}
      For a given point $(x_1,x_2)$, a Taylor expansion gives that, for some $\xi_{x_1,x_2}\in (0,x_2)$,
  \begin{align}\label{eq:TaylorHbis}
    H({x_1},{x_2})=H(x_1,0)+x_2 H_2(x_1,0)+ \frac{x_2^2}{2}H_{22}(x_1,0)+ \frac{x_2^3}{6} H_{222}(x_1,\xi_{x_1,x_2}),\\
    H({x_1},-x_2)=H(x_1,0)+x_2 H_2(x_1,0)+ \frac{x_2^2}{2}H_{22}(x_1,0)+ \frac{x_2^3}{6} H_{222}(x_1,\xi_{x_1,-x_2}).
  \end{align}Thus:
\begin{align*}
    A= &   \epsilon_n^{-2} \int_0^{\epsilon_n} \int_{0}^{\sqrt{3}x_1} \big( 2 x_2 H_2(x_1,0) + \frac{x_2^3}{6} (H_{222}(x_1,\xi_{x_1,x_2})+H_{222}(x_1,\xi_{x_1,-x_2}))\big)\dif x_2\ \dif x_1 \\
     & + \epsilon_n^{-2} \int_{\epsilon_n}^{2\epsilon_n} \int_0^{\sqrt{3}(2\epsilon_n -x_1)}\big(2 x_2 H_2(x_1,0) + \frac{x_2^3}{6} (H_{222}(x_1,\xi_{x_1,x_2})+H_{222}(x_1,\xi_{x_1,-x_2}))\big) \dif x_2\ \dif x_1\\
     = &    \int_0^{\epsilon_n}  3 \big(\frac{x_1}{\epsilon_n}\big)^2 H_2(x_1,0)  \dif x_1  + \int_{\epsilon_n}^{2\epsilon_n} 3\big(\frac{2\epsilon_n-x_1}{\epsilon_n}\big)^2 H_2(x_1,0) \dif x_1 + O(\epsilon_n^3).
\end{align*}
\end{proof}

\subsubsection{Convergence of the generators for linear maps} \label{sec:cv_generators}

From the above lemma, we can show that the generators $\mathcal{A}_n$ converge to $\L^\up$. The limiting generator is defined for twice differentiable $1$-forms, but for establishing convergences, we will require some more regularity and will consider $\Cf^{1,4}$.

\begin{proposition}\label{prop:conv-generateur-global}
  For any $1$-cycle $\chain$ and for a function $\phi\in \Cf^{1,4}(\torus)$, we have 
  \[
  \mathcal{A}_n (\langle \chain,\phi\rangle) = \int_\sigma \L^\up \phi + O( \epsilon_n \|\chain\|_{F} ),
  \]
  where $\|.\|_F$ is the flat norm defined in \eqref{def:flat}.
\end{proposition}

\begin{proof}

Consider a cycle $\chain$, we have 
\begin{align}
\mathcal{A}_n (\langle \chain,\phi\rangle)= &   \sum_{e \in \sigma} \epsilon_n^{-2} \Big(\langle e -\partial_2 \tau^+ ,\phi\rangle_{\Cf_{1,3},\Cf^{1,3}} - \langle e,\phi\rangle_{\Cf_{1,3},\Cf^{1,3}} \Big)+\epsilon_n^{-2}\Big(\langle e -\partial_2 \tau^- ,\phi\rangle_{\Cf_{1,3},\Cf^{1,3}} - \langle e,\phi\rangle_{\Cf_{1,3},\Cf^{1,3}} \Big)\nonumber\\
   =  & \sum_{e \in \sigma}  - \epsilon_n^{-2} \left(\langle \partial_2 \tau^+,\phi\rangle_{\Cf_{1,3},\Cf^{1,3}} + \langle \partial_2 \tau^-,\phi\rangle_{\Cf_{1,3},\Cf^{1,3}} \right) =-\sum_{e \in \sigma} \epsilon_n^{-2} \left( \int_{\tau^+}d\phi + \int_{\tau^-}d\phi \right),
     \end{align}where for each edge $e\in \chain$, we denote by $\tau^+$ and $\tau^-$ the adjacent cofaces, which have different orientations.
Then, in view of \eqref{etape2}, taking $H(x_1,x_2)=\phi^2_1(x_1,x_2)-\phi^1_2(x_1,x_2)$ (which is of class $\Co^3$ by our assumptions) in Lemma~\ref{lem:bis1}, we obtain  
\[
\mathcal{A}_n (\langle \chain,\phi\rangle) = \int_\sigma 3 c \,\L^\up \phi + O(\epsilon_n^2 \|\sigma\|_{\Cf^1} )
\]
and, by Lemma~\ref{lem:flattolength},
\[
\mathcal{A}_n (\langle \chain,\phi\rangle) = \int_\sigma 3 c \,\L^\up \phi + O(\epsilon_n \|\sigma\|_{F} ).  
\]
Now, let $\chain_0$ and 
$\Delta$ be a cycle and a $2$-chain (in the simplicial complex $\mathcal{C}_n$) such that 
\[
\chain = \chain_0 + \partial \Delta.
\]
We have 
\begin{equation}\label{etape11}
\mathcal{A}_n (\langle \chain,\phi\rangle) = \int_{\sigma_0} 3 c \, \L^\up \phi  + \int_{\partial \Delta} 3 c \, \L^\up \phi + O(\epsilon_n \|\sigma\|_{F} ) . 
\end{equation}
Consider an edge $e \in \sigma_0$. Since 
\begin{align*}
    \int_e 3 c \, de =  & 3 \Big(\int_0^{\epsilon_n} \frac{x^2}{\epsilon_n^2}\dd x + \int_{\epsilon_n}^{2\epsilon_n} \frac{(2\epsilon_n-x)^2}{\epsilon_n^2} \dd x\Big) =2 \epsilon_n= \int_e \, de,
\end{align*}we can use a Taylor expansion on $\L^\up \phi$ at the barycenter of $e$ which we denote by $G$ to obtain 
\[
\int_{e} 3 c \L^\up \phi = \int_{e} 3 c \L^\up \phi(G) + O(\epsilon_n^3) = \int_{e}  \L^\up \phi(G) + O(\epsilon_n^3) = \int_{e}  \L^\up \phi + O(\epsilon_n^3). 
\]
This yields for the first term in the right hand side of \eqref{etape11} that
\[
\int_{\sigma_0} 3 c \L^\up \phi = \int_{\sigma_0} \L^\up \phi + O(\epsilon_n^2 \|\sigma_0\|_{\Cf^1}).
\]
Now let us consider the second term in the right hand side of \eqref{etape11}. For a smooth kernel $K$ on $\R^2$, define:
\[K_n(x)=\frac{1}{h_n^2}K\big(\frac{x}{h_n}\big).\]
Then, we have 
 \begin{equation}\label{etape14}
    \int_{\partial \Delta} 3 c \L^\up \phi = \int_{\partial \Delta} 3 (K_n * c)  \L^\up \phi + \int_{\partial \Delta} 3 (c - K_n * c) \L^\up \phi.
\end{equation}
First, remark that for $h_n$ sufficiently small, $\|c-K_n * c\|_\infty \leq \epsilon_n^2$. So for the second term in the rhs of \eqref{etape14},
\[
\int_{\partial \Delta} 3 (c - K_n * c) \L^\up \phi = O(\epsilon_n^2 \|\Delta\|_{\Cf^2}).
\]
Let us consider the first term in the rhs of \eqref{etape14}. Since $ (K_n * c)  \L^\up \phi$ is a $\Co^1$ differential form (thanks to the smoothing of $c$ by $K_n$), we can use Stokes' Theorem to obtain 
\begin{align}
\int_{\partial \Delta} 3 (K_n * c)  \L^\up \phi = &  3 \int_{\Delta} \dd ( (K_n * c)  \L^\up \phi ) \nonumber\\
= &  3 \int_{\Delta}(K_n * c) \  \dd\L^\up \phi +  3 \int_{\Delta} \Big(\frac{\partial K_n * c}{\partial x_1} (\L^\up \phi)^2 - \frac{\partial K_n * c}{\partial x_2} (\L^\up \phi)^1 \Big) \dd x_1 \dd x_2 ,\label{etape16}
\end{align}by using \eqref{etape2}.\\

The computation now depends on details shown in Lemma \ref{lem:calculs_triangle} in Appendix (and leading to \eqref{calcul3} and \eqref{calcul4}). Let $\tau$ be a triangle of $\Delta$ with barycenter denoted again by $G$. By a Taylor expansion of $\dif \L^\up \phi$ at $G$, we have 
\begin{align}
\int_{\tau} 3 (K_n * c)  \dd \L^\up \phi =  &   \int_\tau 3 (K_n * c) \dd \L^\up \phi(G) + O(\epsilon_n^4) \nonumber \\
= &  \frac{5}{3} \int_\tau \dd \L^\up \phi(G) + O(\epsilon_n^4).\label{calcul3}
\end{align}

Furthermore, taking $h_n$ sufficiently small and using another Taylor expansion for $(\L^\up \phi)^1$ and $(\L^\up \phi)^2$ at $G=(G_1,G_2)$, we obtain, using Lemma \ref{lem:calculs_triangle},
\begin{multline}
3 \int_\tau \Big(\frac{\partial K_n * c}{\partial x_1} (\L^\up \phi)^2 - \frac{\partial K_n * c}{\partial x_2} (\L^\up \phi)^1 \Big) \dd x_1\, \dd x_2 \\
\begin{aligned}
    = & 3 \frac{\partial}{\partial x_1}(\L^\up \phi)^2(G) \int_\tau (x_1-G_1) \frac{\partial K_n * c}{\partial x_1}  \dd x_1\, \dd x_2 - 3 \frac{\partial}{\partial x_2}(\L^\up \phi)^1(G) \int_\tau   (x_2-G_2)\frac{\partial K_n * c}{\partial x_2} \dd x_1\, \dd x_2 +O(\epsilon_n^4)\\
 = & -\frac{2}{3} \int_\tau \dd \L^\up \phi(G) + O(\epsilon_n^4).\label{calcul4}
\end{aligned}
\end{multline}

From \eqref{etape16}, \eqref{calcul3} and \eqref{calcul4}, and using another Taylor expansion around $G$,
\[
3 \int_\tau  \dd \big((K_n * c)  \L^\up  \phi\big) = \int_\tau d \L^\up \phi (G) + O(\epsilon_n^4)= \int_\tau d \L^\up \phi + O(\epsilon_n^4).
\]
Hence, using Stokes's theorem once more,
\[
\int_{\partial \Delta} 3 c \L^\up \phi = \int_{\partial \Delta} \L^\up \phi + O(\epsilon_n^2 \|\Delta\|_{\Cf^2}).
\]
Finally, the result follows from taking the $\sigma_0$ and $\Delta$ achieving the flat norm.
\end{proof}


\subsubsection{Tightness of the sequence $(X^{(n)})$}\label{sec:tightness}

The main result of this section is:

\begin{proposition}\label{prop:tightness}Let us work under the assumptions of Theorem \ref{th:cv-cyclerw}. The distributions of the $X^{(n)}$'s form a tight family in the set of probability measures on $\mathbb{D}(\R_+,\Cf_{1,4})$.
\end{proposition}

Once Proposition \ref{prop:tightness} is proved, since $\Cf_{1,4}$ is Polish, we will obtain that the sequence of distributions of $X^{n}$'s is relatively compact for the weak convergence, by Prohorov's theorem \cite[Th.5.1 P.59]{billingsley99}. 

\begin{proof}[Proof of Proposition \ref{prop:tightness}]
First recall that as a direct consequence of the Banach-Alaoglu theorem, see e.g. \cite{rudin} or \cite[Th. III.15]{brezis}, the unit ball in $\Cf_{1,4}$ is a compact set with respect to the weak-* topology. This and  Proposition \ref{prop:flat} hence imply that the marginals of $X^{(n)}$ are tight. Thus, it only remains to check the tightness of the distributions of the real-valued c\`adl\`ag processes $F(X^{(n)}_.)$ for test functions $F$ of the form \eqref{test_function} (see \cite[Theorem 3.1]{jakubowski}). For this, we can use Aldous-Rebolledo's criterion \cite{aldous,joffemetivier}.

Starting from \eqref{eq:cycle-rw-Xn}, we have:
\begin{align}
F(\langle X^n_{\epsilon_n^{-2} t},\phi\rangle) = &  F(\langle X^n_{\epsilon_n^{-2}s},\phi\rangle)\nonumber\\
 & + \int_{\epsilon_n^{-2}s}^{\epsilon_n^{-2}t} \int_{\R_+} \int_{\s_n^2} \Big(F\big(\langle X^n_{u_-},\phi\rangle - \langle \p_2 \tau,\phi\rangle\big) - F(\langle X^n_{u_-},\phi\rangle )\Big) \ind_{\{\theta\leq w(X^n_{u_-},\p_2\tau)\}} N(du,d\theta,d\tau) \nonumber\\
= &  F(\langle X^{(n)}_{s},\phi\rangle) + A^{n}_{\epsilon_n^{-2}t} - A^{n}_{\epsilon_n^{-2}s}+ M^{n}_{\epsilon_n^{-2}t}-M^{n}_{\epsilon_n^{-2}s},\label{etape8}
\end{align}where $A^{n}$ is a predictable finite variation process
\begin{equation}
    A^{(n)}_{\epsilon_n^{-2}t}=\epsilon_n^{-2} \int_{0}^{t} \sum_{\tau \in \s_n^2} \Big(F\big(\langle X^{(n)}_{u},\phi\rangle - \langle \p_2 \tau,\phi\rangle\big) - F(\langle X^{(n)}_{u},\phi\rangle )\Big)  w(X^{(n)}_{ u},\p_2\tau) du , 
\end{equation}
and where $M^{n}$ is a square integrable martingale started from $0$ and with quadratic variation:
\begin{equation}
    \langle M^{n}\rangle_{\epsilon_n^{-2}t} = \epsilon_n^{-2} \int_0^t \sum_{\tau\in \s^n_2} \Big(F\big(\langle X^{(n)}_{ u},\phi\rangle - \langle \p_2 \tau,\phi\rangle\big) - F(\langle X^{(n)}_{ u},\phi\rangle )\Big)^2 w(X^{(n)}_{ u},\p_2\tau) du.\label{bracket:Mn}
\end{equation}
Aldous-Rebolledo's criterion tells us to show that the real-valued processes $(A^n_{\epsilon_n^{-2}t})$ and $(\langle M^{n}\rangle_{\epsilon_n^{-2}t} )$ satisfy the Aldous criterion \cite{aldous}.
Let $\delta>0$ and let $\sigma$ and $\sigma'$ be two stopping times such that $\sigma\leq \sigma'\leq \sigma+\delta$. For $\eta>0$, by Markov's inequality and Proposition \ref{prop:conv-generateur-global}:
\begin{align}
\P\big(| A^n_{\epsilon_n^{-2}\sigma'} -  A^n_{\epsilon_n^{-2}\sigma}|\geq \eta\big)\leq  &\frac{1}{\eta} \Big( \E\Big[\int_{\sigma}^{\sigma'} F'(\langle X^{(n)}_{ u}, \phi\rangle) \mathcal{A}_n(\langle .,\phi\rangle)(X^{(n)}_{u}) \ du\Big]+ \E\big[|R_n(\sigma,\sigma')|\big]\Big) \nonumber\\
\leq  &\frac{1}{\eta} \Big( \E\Big[\int_{\sigma}^{\sigma'} \|F'\|_\infty  \big( \big|\langle \L^{\up} \phi , X^{(n)}_{ u}\rangle  \big| + C_1 \epsilon_n \|X^{(n)}_{ u}\|_F \big)\ du\Big]\nonumber\\
 & \hspace{2cm}+ \E\big[|R_n(\sigma,\sigma')|\big]\Big)\label{etape3b}
 \end{align}
 where for $s$ and $t$
 \begin{align*}
   |  R_n(s,t)|= & \left| \frac{1}{2}\epsilon_n^{-2} \int_s^t \sum_{\tau\in \s^n_2} \langle \partial_2 \tau,\phi\rangle^2 F''\big(\langle X^{(n)}_{ u},\phi\rangle -\theta_\tau \langle  \partial_2 \tau,\phi\rangle\big)w(X^{(n)}_{ u},\partial_2 \tau) du\right|\\
   \leq & C_2 \|F''\|_\infty \sup_{s\leq u\leq t}\big( \epsilon_n \|X^{(n)}_{u}\|_{\Cf_1}\big) \times |t-s| \\
   \leq & C_2 \|F''\|_\infty \sup_{s\leq u\leq t}\big( \|X^{(n)}_{u}\|_{F}\big) \times |t-s|.
 \end{align*}
Now, using that for $\phi\in \Cf^{1,4}$, $\L^{\up} \phi\in \Cf^{1,2}$ and that $\Cf_{1,2} \subset \Cf_{1,4}$,
\begin{equation*}
\big|\langle \L^{\up} \phi , X^{(n)}_{ u}\rangle  \big| \leq  \|X^{(n)}_{ u}\|_{\Cf_{1,2}} \| \L^{\up} \phi\|_{\Cf^{1,2}} \leq C  \|X^{(n)}_{ u }\|_{\Cf_{1,4}}  \| \phi\|_{\Cf^{1,4}},
\end{equation*}for a constant $C$ (that will change from line to line). Under the Conjecture \ref{conjecture-moments-discrets}, and by Proposition \ref{prop:flat}, 
 \begin{align}
 \P\big(|\langle A^n_{\sigma'},\phi\rangle - \langle A^n_\sigma,\phi\rangle|\geq \eta\big)
\leq  &\frac{C \delta}{\eta}  \label{etape3}
\end{align}that can be made as small as wished by setting $\delta$ sufficiently small.\\

For the predictable quadratic variation process of $\langle M^n\rangle$, using similar computation with a Taylor expansion at the order 1, we have:
\begin{align}
    \E\Big(\big| \langle M^n\rangle_{\epsilon_n^{-2}t}- \langle M^n\rangle_{\epsilon_n^{-2}s}\big|\Big)\leq  & C\epsilon_n^{-2} \int_s^t  \E\Big[\sum_{\tau \in \s^n_2} \big( \epsilon_n^4 \|F'\|_\infty^2 \big)  w\big(X^{(n)}_{u}, \p_2\tau \big)\  du \Big]\nonumber\\
\leq & C \delta  \|F'\|^2_\infty  \sup_{s\leq u\leq t} \|X^{(n)}_u\|_{F}.
\end{align}Under Conjecture \ref{conjecture-moments-discrets}, 
\begin{align}
    \P\big(| \langle M^n\rangle_{\epsilon_n^{-2}\sigma'} -  \langle M^n\rangle_{\epsilon_n^{-2}\sigma}|\geq \eta\big)\leq  &\frac{C}{\eta} \|F'\|^2_\infty   \times \delta.\label{etape15}
\end{align}
This concludes the proof of Proposition \ref{prop:tightness} and of Theorem \ref{th:cv-cyclerw}.
\end{proof}

\subsubsection{Convergence of the generators (2)}\label{section:fin_generateur}


Let us consider a limiting value $\bar{X}$ of the tight sequence $(X^{(n)})$ in $\D(\R_+,\Cf_{1,4})$ and a subsequence (again denoted by $(X^{(n)})$ for the sake of notation) that converges in distribution to $\bar{X}$. By Skorokhod representation theorem (see e.g. \cite[Th. 6.7 p.70]{billingsley_probability_and_measure}) it is possible to assume that the convergence is almost sure.\\

Notice that the sequence $(X^{(n)})$ is in fact $C$-tight (see \cite[p.561]{jacod}) as for all $T>0$ and for all $\eta>0$,
\begin{equation}
\lim_{n\rightarrow +\infty} \P\big(\sup_{t\leq T} \| X^{(n)}_t-X^{(n)}_{t_-}\|_{\Cf_{1,4}}>\eta\big)=\lim_{n\rightarrow +\infty} \P\big( \| \partial_2 [012]\|_{\Cf_{1,4}}>\eta\big)=0,
\end{equation}where $[012]$ stands for the triangle in Fig. \ref{fig:torus} whose norm is of order $\epsilon_n$ and smaller than $\eta$ for $n$ sufficiently large. The limiting values are then in $\Co(\R_+,\Cf_{1,4})$.\\

Recall that the unit ball of $\Cf_{1,4}$ is compact, and hence we can find a denumerable dense family of functions $(\phi_k)_{k\in \N}$ of $\Cf^{1,4}$. 

\begin{proposition}\label{prop:convmartingale-fin}Let us work under the assumptions of Theorem \ref{th:cv-cyclerw}. For $F(x)=x$ and for $\phi_k\in \Cf^{1,4}$, $k\in \N$, as introduced above the Proposition, let us denote by $M^{n,k}$ the martingale defined in \eqref{etape8}. This real-valued martingale is tight and converges in distribution to the martingale $W^{(k)}\in \Co(\R_+,\R)$ started at zero and with predictable quadratic variation:
    \begin{equation}
        \langle W^{(k)}\rangle_t :=  \lim_{n\rightarrow +\infty} \epsilon_n^{-2} \int_0^t \sum_{\tau\in \s^n_2}  \langle \p_2 \tau,\phi_k\rangle^2 w(X^{(n)}_{u},\p_2\tau) du.\label{cv:bracket}
    \end{equation}Moreover, it is possible to define a martingale measure $(W_t(d\sigma))_{t\in \R_+}$ such that for all $k\in \N$,
    \[W^{(k)}_t = \int_{ \Cf_{1,4}} \phi_k(\sigma) W_t (d\sigma).\]
\end{proposition}

\begin{proof}
Adapting the computation of the previous section, we can show that not only the sequence $(X^{(n)})$ is tight, but also the joint sequence $(X^{(n)},M^{n,k}, k\in \N)$.
Let us consider a limiting value $(\bar{X},W^{(k)}, k\in \N)$ of the tight sequence $(X^{(n)},M^{n,k},k\in \N)$ and a subsequence (again denoted by $(X^{(n)},M^{n,k},k\in \N)$ for the sake of notation) that converges in distribution to $(\bar{X},W^{(k)}, k\in \N)$. We conclude by applying \cite[Theorem 2.4 P.487]{jacod}. By the Skorokhod representation theorem (see e.g. \cite[Th. 6.7 p.70]{billingsley_probability_and_measure}) it is possible to assume that the convergence is almost sure. Linearity then implies the existence of a martingale measure (see \cite{Walsh}).
\end{proof}

Recall now that the test functions characterizing the convergence of our generators have been defined in \eqref{test_function}. For $\phi\in \Cf^{1}$ and $F\in \Co(\R)$, we will denote $F_\phi(\sigma)=F(\langle \sigma,\phi\rangle)$.

\begin{corollary}
    Let us work under the assumptions of Theorem \ref{th:cv-cyclerw}. Consider the martingale $M^n$ defined for any functions $F\in \Co^2_b(\R,\R)$ and $\phi$ in \eqref{etape8}. The sequence of real-valued martingales $(M^n)$ converges in distribution to the centered martingale $W\in \Co(\R_+,\R)$ with predictable quadratic variation:
\begin{equation}
    \langle W\rangle_t := \int_0^t F^{'2}\big(\langle \bar{X}_u,\phi\rangle\big) d\langle W_u,\phi\rangle.
\end{equation}    
\end{corollary}
\begin{proof}
Recall the bracket of the martingale $M^n$ computed in \eqref{bracket:Mn}.
We have:
\begin{align*}
\lim_{n\rightarrow +\infty} \langle M^{n}\rangle_{\epsilon_n^{-2}t} = &  \lim_{n\rightarrow +\infty} \epsilon_n^{-2} \int_0^t \sum_{\tau\in \s^n_2} \Big(F\big(\langle X^{(n)}_{u},\phi\rangle - \langle \p_2 \tau,\phi\rangle\big) - F(\langle X^{(n)}_{u},\phi\rangle )\Big)^2 w( X^{(n)}_{u},\p_2\tau) du\nonumber\\
= & \lim_{n\rightarrow +\infty} \epsilon_n^{-2} \int_0^t  F^{'2}\big(\langle  X^{(n)}_{u},\phi\rangle \big) \sum_{\tau\in \s^n_2} \langle  \partial_2 \tau,\phi\rangle^2  w(X^{(n)}_{u},\p_2\tau) du+ O\big(\epsilon_n^{4} \sup_{u\leq t} \|X^{(n)}_u\|_{F}\big),
\end{align*}by using Taylor's formula as in \eqref{etape3b}. The result is then a consequence of Conjecture \ref{conjecture-moments-discrets} and of \eqref{cv:bracket}.
\end{proof}

We are now almost in position to show that the limiting solution $X$ solves the martingale problem associated with $\mathfrak{A}$.
Notice that in \eqref{def:mathfrakA}, we have several equivalent formulation of the last term. This is stated in the following corollary, whose proof is postponed to Appendix \ref{app:lemH1}. The proof is similar to the one of Proposition \ref{prop:conv-generateur-global} and based on Taylor expansions.

\begin{corollary}\label{cor:somme_taucarres}For $X_u\in \Cf_{1,4}$, recall the definition of $\widetilde{X}_u$ in \eqref{def:sigma-tilde}. We have:
\begin{align*}
\int_0^t   F''\Big(\int_{X_u} \phi\Big) \Big\langle \widetilde{X}_u, \big(* d \phi\big)^2\Big\rangle  \dif u= &    \int_0^t  F''\Big(\int_{X_u} \phi\Big) \int_{\torus} \big(\frac{\p \phi^{2}}{\p x_{1}}-\frac{\p \phi^{1} }{\p x_{2}}\big)^2 \widetilde{X}_u(\dif x_1, \dif x_2)  \dif u\\
    = &  \lim_{n\rightarrow +\infty} \epsilon_n^{-2} \int_0^t F''\Big(\int_{X^{(n)}_u} \phi\Big) \sum_{\tau\in \s^n_2}  \langle \p_2 \tau,\phi\rangle^2\,   w(X^{(n)}_{u},\p_2\tau) \dif u \\
    = & \int_0^t F''\Big(\int_{X_u} \phi\Big) \dif \langle W_u,\phi\rangle 
\end{align*}
\end{corollary}

\begin{proposition}\label{prop:conv_fini_dim}
Let us consider $\phi\in \Cf^{1,4}$ and $F\in \Co^3_b(\R)$ and recall the definition of $\mathfrak{A}$ in \eqref{def:mathfrakA}. Then, for $k\in \N$, for $0\leq t_1\leq \dots t_k\leq  s<t$,  and $\Phi_1,\dots \Phi_k$ bounded continuous functions of $\Cf_{1,4}$,
\begin{equation}
    \lim_{n\rightarrow +\infty}\E\Big[ \Big(F(\langle X^{(n)}_{t},\phi\rangle) -F(\langle X^{(n)}_{s},\phi\rangle) - \int_s^t \mathfrak{A}F_\phi \big( X^{(n)}_{u}\big)du \Big)\prod_{i=1}^k \Phi_i(X^{(n)}_{t_i}) \Big]=0.
\end{equation}
\end{proposition}

\begin{proof}
Let us consider $s<t$. 
Using \eqref{eq:cycle-rw-Xn} and Itô's formula \cite[Th.5.1 P.66]{ikedawatanabe} for jump processes, we have that for every $n\in \N$:
\begin{equation}
    \E\Big[ \Big(F(\langle X^{(n)}_{t},\phi\rangle) -F(\langle X^{(n)}_{s},\phi\rangle) - \int_s^t \mathcal{A}_n F_\phi \big( X^{(n)}_{u}\big)du \Big)\prod_{i=1}^k \Phi_i(X^{(n)}_{t_i}) \Big]=0.
\end{equation}
Using a Taylor expansion, we have for some $\theta\in (0,1)$:
\begin{multline}
F\big(\langle X^{(n)}_{ u},\phi\rangle - \langle \p_2 \tau,\phi\rangle\big) = F(\langle X^{(n)}_{u},\phi\rangle )
- \langle \p_2\tau,\phi\rangle \ F'(\langle X^{(n)}_{ u},\phi\rangle )\\+\frac{1}{2}\langle \p_2\tau,\phi\rangle^2 \ F''(\langle X^{(n)}_{u},\phi\rangle) + \frac{1}{6}\langle \p_2\tau,\phi\rangle^3 F^{(3)}\big(\langle X^{(n)}_{ u},\phi\rangle- \theta \langle \p_2\tau ,\phi\rangle\big).
\end{multline}
As a result,
\begin{multline*}
\big|\int_s^t \mathcal{A}_n F_\phi(X^{(n)}_u)du - \int_s^t \mathfrak{A} F_\phi(X^{(n)}_u) du \big| \\
\begin{aligned}
= & \Big|\epsilon_n^{-2} \int_s^t \sum_{\tau\in \s^n_2}\big(F_\phi(X^{(n)}_u-\partial_2 \tau)-F_\phi(X^{(n)}_u)\big)w(X^{(n)}_u,\partial_2 \tau) \ du-\int_s^t  F'(\langle X^{(n)}_u,\phi\rangle) \int_{X^{(n)}_u} \mathfrak{L}^{\uparrow} \phi \ du \\
 & \hspace{3cm}-   \lim_{n\rightarrow +\infty} \epsilon_n^{-2} \int_s^t F''\Big(\int_{X^{(n)}_u} \phi\Big) \sum_{\tau\in \s^n_2}  \langle \p_2 \tau,\phi\rangle^2\,   w(X^{(n)}_{u},\p_2\tau) \dif u \Big|\\
\leq & \int_s^t \Big|F'_\phi(X^{(n)}_u) \big(\mathcal{A}_n (\langle X^{(n)}_u,\phi\rangle) -  \langle X^{(n)}_u,\mathfrak{L}^{\uparrow} \phi\rangle\big) \dif u \Big| 
 +  \Big| \frac{1}{2}\epsilon_n^{-2} \int_s^t F''\big(\langle X^{(n)}_u,\phi\rangle\big) \sum_{\tau\in \s^n_2} \langle \partial_2 \tau,\phi\rangle^2 w(X^{(n)}_u,\p_2\tau) \dif u \\
 & - \lim_{n\rightarrow +\infty} \epsilon_n^{-2} \int_s^t F''\Big(\int_{X^{(n)}_u} \phi\Big) \sum_{\tau\in \s^n_2}  \langle \p_2 \tau,\phi\rangle^2\,   w(X^{(n)}_{u},\p_2\tau) \dif u\Big|+ O(\epsilon_n^3 \|\sigma\|_{\Cf_1})\\
\leq & C \|F'\|_\infty \epsilon_n \|\sigma\|_F + O(1)+ O(\epsilon_n^2 \|\sigma\|_{F}),
\end{aligned}
\end{multline*}
which converge in expectation to zero when $n\rightarrow +\infty$ under Conjecture \ref{conjecture-moments-discrets}. Then, the announced result is deduced because the functions $\Phi_i$ are bounded.
\end{proof}


Let us now conclude. In view of Theorem 8.2 in \cite[p.226]{ethierkurtz}, given the result of Proposition \ref{prop:conv_fini_dim}, every limiting process $\bar{X}$ of the sequence $(X^n)$ is a solution of the martingale problem associated with $\mathfrak{A}$. 
If there is a unique solution to this martingale problem, then the whole sequence converges to this solution. The uniqueness problem is left open. This concludes the proof of Theorem \ref{th:cv-cyclerw}.

\section{Application}
\label{sec:applications}
A straightforward application of our random walk is the enumeration of the
elements of $\ker L_{k}$  (i.e. chains which are cycles) with information only on the simplexes
(i.e. the basis of $\ch_{k}$). Start with any cycle and let the random walk run
to obtain a partial list of the chains which are in the homology class of the
initial chain. The longest the run, the longest  the list.

An avatar of this simple algorithm is to use the Markov chain in the context of
the simulated annealing algorithm. This heuristic is devised  to minimize an energy
function through a random walk in the state space trying to avoid local
minimums. We here focus  on $1$-chains but the procedure could
be applied as is to higher order simplexes.

As we have seen above, the dimension of the kernel of $L_{1}$ is the Betti
number~$\beta_{1}$. Furthermore,  any element of $\ker L_{1}$ can be seen geometrically as
a closed path (i.e. a cycle) which surrounds
one or several holes, see Figure~\ref{fig:Rips_known}.

\begin{figure}[!ht]
  \centering
  \includegraphics*[width=0.45\textwidth]{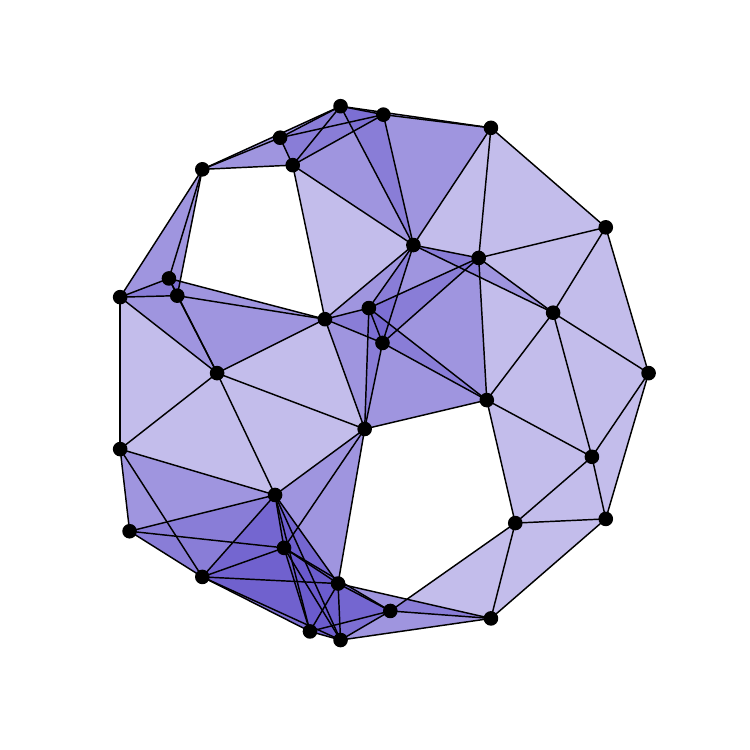}\hfil \includegraphics*[width=0.45\textwidth]{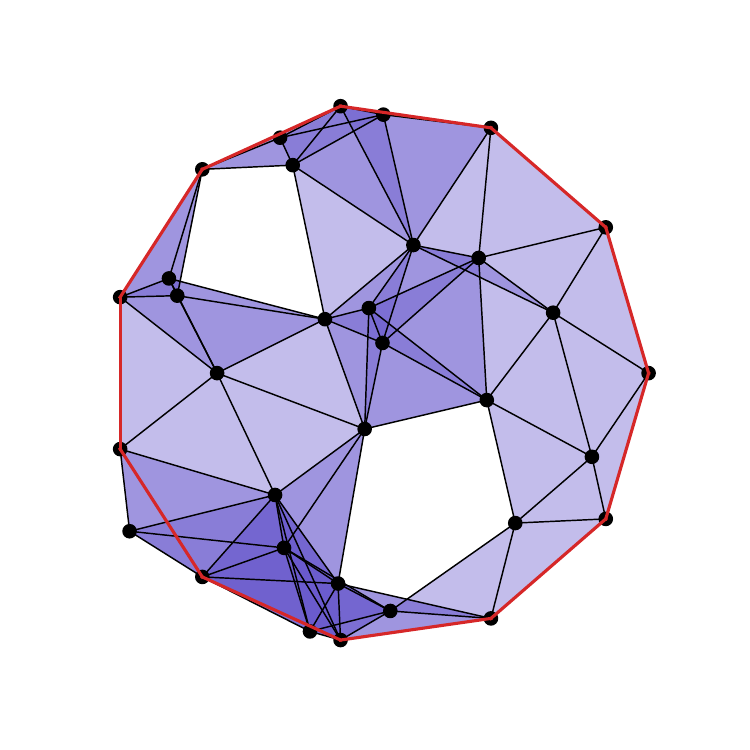}
  \caption{A simplicial complex with $\beta_1=2$ and a cycle which surrounds the two
    holes.}
  \label{fig:Rips_known}
\end{figure}

It is important for some applications to be able to precise the holes locations,
in order for instance to \textsl{refill} them (see e.g.\cite{Vergne2014a}). The first
idea which comes to mind is to find a closed path around each hole with \textsl{minimum
length}. The algebraic representation of the geometric
problem does not contain any information about the Euclidean length of a path so
this definition of length cannot be used here. We thus have to rely on the
algebraic length of a path, namely, the
sum of the weight of each edge in a chain:
\begin{equation*}
  U(\sigma)=\sum_{\tau\in \s_{1}^{+}}\left|\langle \tau^*, \,\sigma \rangle_{\ch^{1},\, \ch_{1}}\right|.
\end{equation*}
For reasons which will be explained below, it is not always possible to define
$U(\sigma)$ as the number of edges in $\sigma$. We here have
\begin{theorem}
For $k\geq 0$, let us consider the vector
  \begin{equation*}
    \omega(t)=\Bigl( \P\bigl(\langle \tau^*,X_t\rangle_{\ch^{k},\ch_{k}}=1 \,|\, X_0\bigr),\ \tau \in \s_{k} \Bigr).
  \end{equation*}
The $\ch_{k}$-valued process $(\omega(t),\, t\ge 0)$  satisfies the differential
equation
\begin{equation}\label{eq_rw:4b}
  \frac{d}{dt}\omega(t)=-L_k^{\up}\omega(t).
\end{equation}
\end{theorem}

 It is shown in \cite{muhammad_control_2006} that the solution of the
differential equation
\begin{equation}\label{eq_rw:4}
  \frac{d}{dt}\omega(t)=-L_k\omega(t).
 \end{equation}
converges as $t$ goes to infinity towards $\omega_{\infty}\in \ker L_{k}.$ An
element of $\ker L_{k}$ is a cycle and the numerical experiments of
\cite{muhammad_control_2006} tend to show  that the weights of
$\omega_{\infty}$ are heavier  around the hole.
 
 \begin{proof}
   For $\tau$ a given $k$-simplex, consider the function
   \begin{align*}
     \Theta_{\tau}\, :\, \ch_{k}&\longrightarrow \R\\
     \sigma &\longmapsto \langle \tau^*,\, \sigma\rangle_{\ch^{k},\ch_{k}}^{+}.
   \end{align*}
   If $\sigma$ is a simple chain, $ \Theta_{\tau}(\sigma)=1$ whenever $\tau$
   (taking care of the orientation) belongs to the support of $\sigma$ and $0$
   otherwise, i.e.
   \begin{equation*}
     \Theta_{\tau}(\sigma)=\car_{\supp \sigma}(\tau).
   \end{equation*}
   The function $\Theta_{\tau}$ is not a linear function but we can proceed as
   in the proof of Theorem~\ref{thm:AEgalLup} to establish
   that \begin{equation*} A\Theta_{\tau}(\sigma)=-L^{\up}_k \Theta_{\tau}(\sigma)
   \end{equation*}
   by noticing that for $\sigma'\sim \sigma$, $\sigma'$ differs from $\sigma$ only by
   the boundary of a $(k+1)$-simplex and hence:
   \begin{equation*}
     \Theta_{\tau}(\sigma)- \Theta_{\tau}(\sigma')= \Theta_{\tau}(\sigma-\sigma').
   \end{equation*}
   We conclude by the Chapman-Kolmogorov equations for continuous time Markov
   chains.
 \end{proof}
 Then, our first hope was that by letting the random walk run, the most visited edges
 would be located around the holes. Our numerical experiments showed
 disappointing  results probably because of a too slow convergence.

 The determination of a shortest path around some hole
 has strong reminiscence with the so-called minimum cycle bases in graphs (see
 \cite{Mehlhorn} and references therein) but the algorithms developed there
 cannot be used here since they, by essence, cannot take into account the
 topological features like holes. For instance, the Horton's collection algorithm
 starts by computing all the shortest cycles $C(v,e)$ which start at a vertex $v$ and
 contains an edge $e$ containing~$v$. Then, it computes some linear combinations of
 these cycles to find a minimum weight bases. We immediately see that there is no
 guarantee that these linear combinations are in the convenient homology class
 since this information is not contained in $\{C(v,e),v\in \s_{0},e\in
 \s_{1}\}$.
Moreover, the number of paths is usually so huge that even a polynomial
algorithm in the number of paths would be unrealistic. We here resort to  a simulated
annealing (SA for short) algorithm  where the energy function to be minimized is simply $U$ as
defined above. There is no guarantee that the SA algorithm is not stuck in a local
minimum of $U$ but in practice, it works pretty well.

We initiate the SA algorithm with an element $\tau_{0}$ of $\ker L_{1}$ and a temperature
$T_{0}$. We fix some $\alpha\in (0,1)$.  Assume that the SA algorithm has
already made $m$ steps with $m=0,\cdots$. The new temperature is $T_{m+1}=T_{0}\alpha^{m+1}$. We apply a transition of
our Markov chain  from the chain $\sigma_{m}$ and get a chain $\sigma'$. If
$U(\sigma')<U(\sigma_{m})$ then $\sigma_{m+1}=\sigma'$. If $U(\sigma')>U(\sigma_{m})$, the new state is
$\sigma'$ with probability
\begin{equation*}
  \exp\left( -\frac{U(\sigma')-U(\sigma_{m})}{T_{m+1}} \right)
\end{equation*}
and $\sigma_{m+1}=\sigma_{m}$ with the complementary probability. We stop the
algorithm when the temperature is below a predefined threshold and we expect
that at this time, the state of the algorithm yields a short path around one or
several  holes.

In the situation of Figure~\ref{fig:Rips_known}, we know easily a path (i.e. a
chain with unit weights) around the two holes. After a few iterations of the SA
algorithm, we obtain a very neat answer to our problem illustrated in Figure~\ref{fig:HoleDetectionI}.

\begin{figure}[!ht]
  \centering
  \includegraphics*[width=0.45\textwidth]{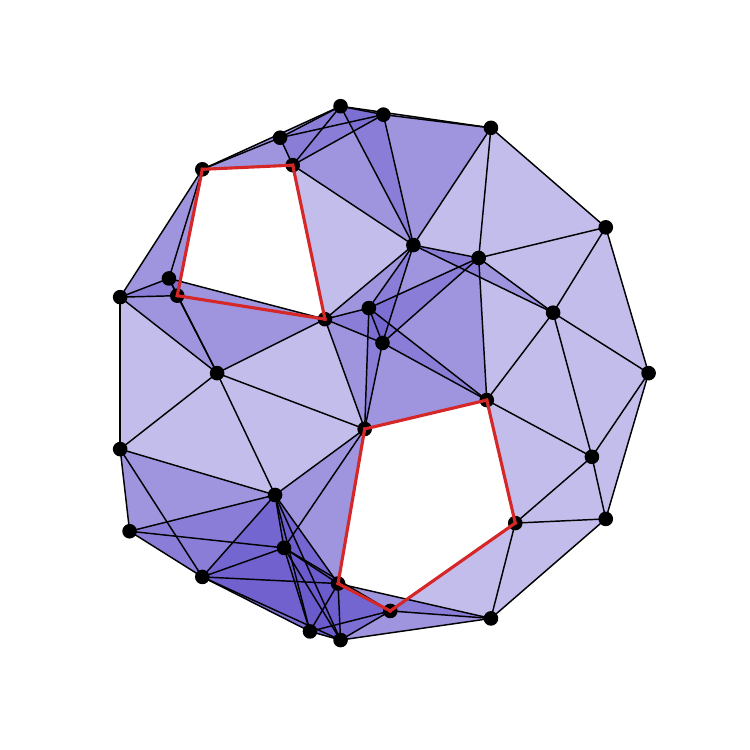}
  \caption{The result of the SA algorithm with a manually defined initial
    chain.}
  \label{fig:HoleDetectionI}
\end{figure}

It is not always feasible to describe simply a path around the holes. Then, we
can use as $\sigma_{0}$ an element of $\ker L_{1}$ obtained from the Smith normal
form of $L_{1}$ (we must be careful not to work with floating numbers
since the computations error quickly propagate and the final result is no longer a
cycle). The price to pay for such a generality is that the
weights of the components of $\sigma_{0}$ are usually very large integers and
almost no edge has a zero weight, see the left picture of Figure~\ref{fig:Weights}.
After the SA algorithm has been run, the edges around the hole have a much
greater weight meanwhile the other edges see their weight almost unchanged or
even decreased. It
remains to remove the uninteresting edges by a cut-off procedure to have a good
location of the hole.

\begin{figure}[!ht]
  \centering
  \includegraphics*[width=0.45\textwidth]{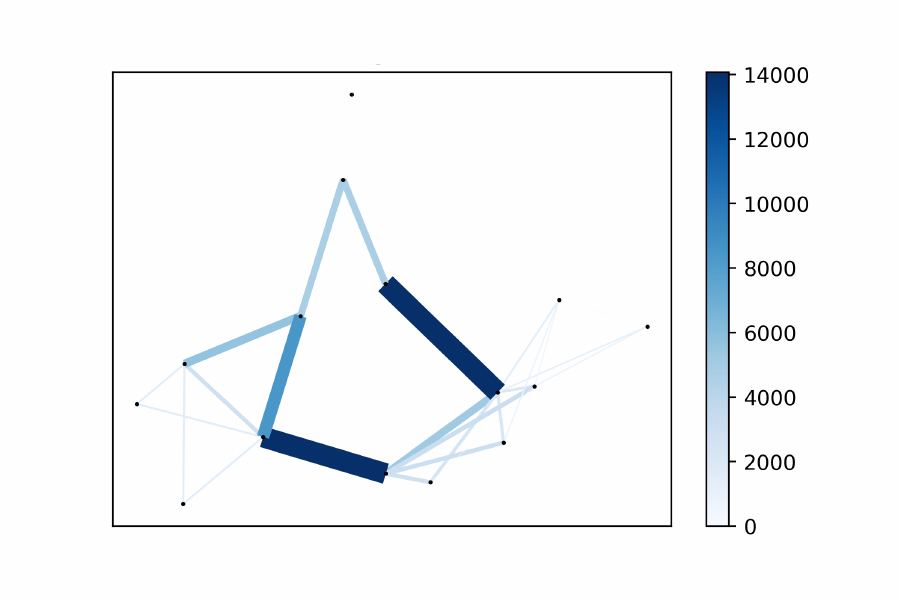}\hfil \includegraphics*[width=0.45\textwidth]{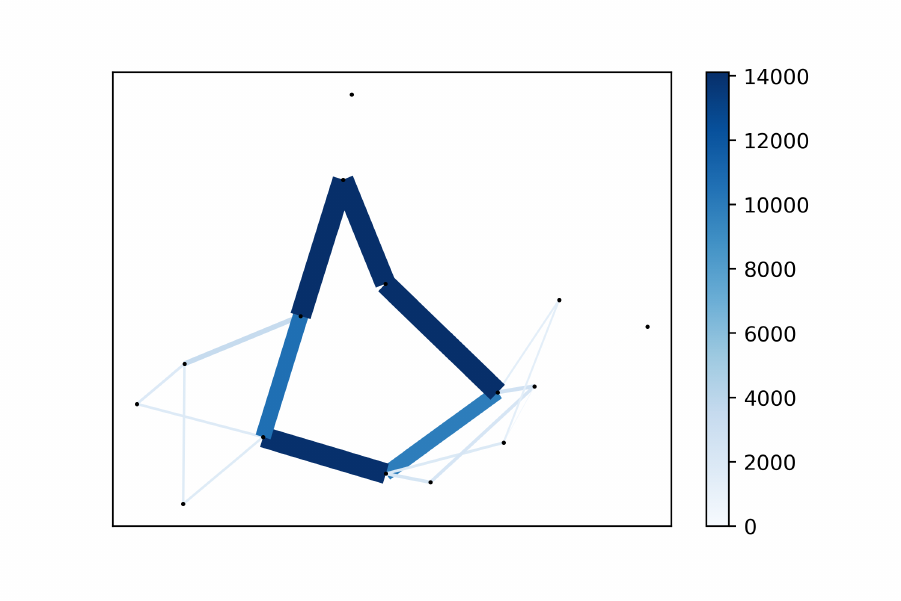}
  \caption{Another simplicial complex with $\beta_{1}=1$. The width of the edge and the darkness of the blue color are proportional to the weight of the edge given by the Smith normal form of $L_1$. On the left, some
    edges not adjacent to the \textsl{hole} have a relative great weight. After the SA
    has been run, these weights are no longer so important and the edges around
  the \textsl{hole} are much heavier.}
  \label{fig:Weights}
\end{figure}

According to the results of \cite{Catoni,trouve}, one can show (see
\cite{phd_zhang}) that there exists an optimal cooling scheme, i.e. an optimal
$\alpha$, such that
\begin{equation*}
  \sup_{\sigma_{0} \in \ker L_{1}} \P_{\sigma_{0}}(U(\sigma_{m})>\min U)\le \frac{c}{m^{1/L}}
\end{equation*}
where $L=\max\{U(\sigma), \, \sigma \text{ in the homology class of }\sigma_{0}\}.$ On
the examples we test our SA algorithm with, the localisation was obtained much faster.



\appendix 

\section{A technical Lemma for integrating a function $H(x_1,x_2)$ on small triangles of the torus}

\begin{lemma}
  \label{lem:H1}
  Let $H\,:\,\torus\to \R$ be $\Co^{3}$. Let us consider a triangle of $\mathcal{T}_n$. Without loss of generality, we can assume (by a change of coordinates) that it is the triangle $[021]$ whose vertices have coordinate $(0,0)$, $(\epsilon_n,\sqrt{3}\epsilon_n)$ and $(2\epsilon_n,0)$. Then, as $\epsilon_{n}$ goes to $0$,
  \begin{align}
     \int_{[021]}H({x_1},{x_2})\dif {x_2}\dif {x_1} 
     = &    -\sqrt{3} H(0,0) \epsilon_n^2 - \big(\sqrt{3}H_1(0,0)+H_2(0,0)\big) \epsilon_n^3 \nonumber \\
     - & \Big(\frac{7\sqrt{3}}{6}H_{11}(0,0)+\frac{\sqrt{3}}{2}H_{22}(0,0)+2H_{12}(0,0)\Big) \epsilon_n^4+O(\epsilon_n^5)\nonumber\\
     = & -\frac{\sqrt{3}}{2} \epsilon_n \int_0^{2\epsilon_n} H(x_1,0)\dif x_1+O(\epsilon_n^3).    \label{eq:H}
     \end{align}
\end{lemma}
\begin{proof}
  Using Fubini's theorem and taking into account the orientations,
  \begin{align*}
    M_{1}&:=  \int_{[021]}H({x_1},{x_2})\dif {x_2}\dif {x_1}\\&=-\int_{0}^{\epsilon_{n}}\int_{0}^{{x_1}\sqrt{3}}
    H({x_1},{x_2})\dif {x_2}\dif {x_1}-\int_{\epsilon_{n}}^{2\epsilon_{n}}\int_{0}^{\sqrt{3}(2\epsilon_{n}-{x_1})} H({x_1},{x_2})\dif {x_2}\dif {x_1},\\
    & =:- I-II.
  \end{align*}
  A Taylor expansion gives
  \begin{multline}\label{eq:TaylorH}
    H({x_1},{x_2})=H(0,0)+x_1 H_1(0,0)+x_2H_2(0,0)
    +x_1^{2}\,H_{11}(0,0)+ x_2^2 H_{22}(0,0)+ 2 x_1 x_2 H_{12}(0,0)\\
    + x_1^3 r_{111}(x_1,x_2) + x_1^2 x_2 r_{112}(x_1,x_2)+ x_1 x_2^2 r_{122}(x_1,x_2)+x_2^3 r_{222}(x_1,x_2)
  \end{multline}
  where the remainder terms $r_{111}$, $r_{112}$, $r_{122}$ and $r_{222}$ satisfy: 
  \[\sup_{{x_1},{x_2}\in [0,1]^{2}}\big(|r_{111}({x_1},{x_2})|,|r_{122}({x_1},{x_2})|,|r_{122}({x_1},{x_2})|, |r_{222}(x_1,x_2)|\big)<+\infty,\]
  with a bound that we can chose to depend only on $H$. Then, injecting \eqref{eq:TaylorH} in $I$,
  \begin{align}
      I = & H(0,0) \int_0^{\epsilon_n} x_2\sqrt{3}\dif{x_1}+H_1(0,0) \int_0^{\epsilon_n} x_1^2 \sqrt{3} \dif{x_1}+ H_2(0,0) \int_0^{\epsilon_n} \frac{3x_1^2}{2} \dif{x_1}\nonumber\\
       & + H_{11}(0,0)\int_0^{\epsilon_n} x_1^3\sqrt{3} \dif{x_1}+H_{22}(0,0)\int_0^{\epsilon_n}x_1^3\sqrt{3} \dif{x_1}+2H_{12}(0,0)\int_0^{\epsilon_n} x_1 \int_0^{x_1\sqrt{3}}x_2 \dif{x_2}\ \dif{x_1} +R^I_n,\nonumber\\
      = & \frac{\sqrt{3}}{2} H(0,0) \epsilon_n^2 + \big(\frac{\sqrt{3}}{3}H_1(0,0)+\frac{1}{2}H_2(0,0)\big)\epsilon_n^3  + \Big(\frac{\sqrt{3}}{4}H_{11}(0,0)+ \frac{\sqrt{3}}{4} H_{22}(0,0)+\frac{3}{4} H_{12}(0,0)\Big) \epsilon_n^4 + R^I_n,\label{etape4}
  \end{align}where the remainder term $R^I_n$ satisfies $R^I_n=O(\epsilon_n^5)$. A similar computation shows that:
  \begin{multline}
      II=  \frac{\sqrt{3}}{2}H(0,0) \epsilon_n^2 + \big(\frac{2\sqrt{3}}{3} H_1(0,0)+\frac{1}{2}H_2(0,0)\big)\epsilon_n^3 \\
      + \Big(\frac{11 \sqrt{3}}{12} H_{11}(0,0)+\frac{\sqrt{3}}{4} H_{22}(0,0)+ \frac{5}{4}H_{12}(0,0)\Big) \epsilon_n^4 + R^{II}_n,\label{etape5}
  \end{multline}with $R_n^{II}=O(\epsilon_n^5)$. The results stems from the addition of \eqref{etape4} and \eqref{etape5}.
\end{proof}

\section{A technical lemma for integrating the function $c$ and its derivatives on small triangles of the torus}
\begin{lemma}\label{lem:calculs_triangle}    Consider the triangle $[021]$ (with the notation of Fig. \ref{fig:torus}) whose barycenter is $G=(\epsilon_n,\epsilon_n/\sqrt{3})$. We have:
    \begin{align}
       & \int_{[021]} c(x_1,x_2)\dd x_1 \, \dd x_2= \frac{5}{9} \, \epsilon_n^2\sqrt{3},\nonumber\\
       & \int_{[021]} (x_1-\epsilon_n) c(x_1,x_2)\dd x_1 \, \dd x_2=0,\nonumber\\
       & \int_{[021]} \frac{\partial c}{\partial x_1}(x_1,x_2)\dd x_1 \, \dd x_2=  0,\nonumber\\
       & \int_{[021]} (x_1-\epsilon_n) \frac{\partial c}{\partial x_1}(x_1,x_2) \dd x_1 \, \dd x_2=  -\frac{2}{9}\epsilon_n^2 \sqrt{3}, \label{morceau_final}\\
       & \int_{[021]} \big(x_2-\frac{\sqrt{3}}{3}\epsilon_n\big) \frac{\partial c}{\partial x_1}(x_1,x_2) \dd x_1 \, \dd x_2=  0.
    \end{align}
\end{lemma}

\begin{proof}
Let us denote by
$C_1$, $C_2$ and $C_3$ the Voronoi cells associated with the three vertices (see Fig. \ref{fig:triangle012}).

\begin{figure}[!ht]
  \centering
 \begin{tikzpicture}[scale=2,font=\fontsize{6}{6}\selectfont]
           \fill[color=black] (0,0) circle(1pt) node[below] {$0$};
           \fill[color=black] (2,0) circle(1pt) node[below] {$1$};
           \fill[color=black] (1,1.73) circle(1pt) node[above] {$2$};
            \fill[color=black] (1,0.58) circle(1pt) node[above] {$G$};
            \fill[color=black] (1,0) circle(1pt) node[below] {$H$};
            \draw (0,0) -- (2,0);
            \draw (0,0) -- (1,1.73);
            \draw (2,0) -- (1,1.73);
            \draw (0.5,0.87) -- (1,0.58);
            \draw (1.5,0.87) -- (1,0.58);
            \draw (1,0) -- (1,0.58);
            \draw (0.5,0.5) node[right] {$C_1$};
            \draw (1.5,0.5) node[left] {$C_2$};
            \draw (1.1,1) node[above] {$C_3$};
            \draw [dashed] (0,0) -- (1,0.58);
            \draw [dashed] (1,1.73) -- (1,0.58);
            \draw [dashed] (2,0) -- (1,0.58);
          \end{tikzpicture}
  \caption{Triangle $[012]$. $0:=(0,0),\
    1:=(2\epsilon_{n},0),\ 2:=(\epsilon_{n},\sqrt{3}\epsilon_{n}),\,    3:=(\epsilon_{n},-\sqrt{3}\epsilon_{n})$.}
  \label{fig:triangle012}
\end{figure}
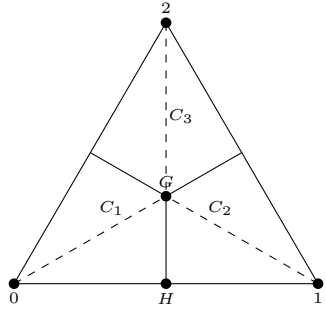

 We have:
\begin{align*}
c(x_1,x_2)=\frac{1}{\epsilon_n^2} \times \begin{cases}
    x_1^2+x_2^2 & \mbox{ on }C_1\\
    (x_1-2\epsilon_n)^2+x_2^2 & \mbox{ on }C_2\\
    (x_1-\epsilon_n)^2 + (x_2-\epsilon_n \sqrt{3})^2 &  \mbox{ on }C_3.
\end{cases}
\end{align*}
and,
\begin{align*}
\frac{\partial c}{\partial x_1}(x_1,x_2)=\frac{1}{\epsilon_n^2} \times \begin{cases}
    2 x_1 & \mbox{ on }C_1\\
    2(x_1-2\epsilon_n) & \mbox{ on }C_2\\
    2(x_1-\epsilon_n) &  \mbox{ on }C_3.
\end{cases}
\end{align*}
By symmetry, 
\begin{align}
\int_{[021]} c(x_1,x_2)  \dd x_2\dd x_1 = &  6\times \int_0^{\epsilon_n} \int_{0}^{x_1/\sqrt{3}}\frac{x_1^2+x_2^2}{\epsilon_n^2}  \dd x_2\dd x_1 \nonumber\\
= & \frac{6}{\epsilon_n^2} \int_0^{\epsilon_n} \Big( \frac{x_1^3}{\sqrt{3}} + \frac{ x_1^3}{9\sqrt{3}} \Big) \dd x_1 \nonumber\\
= & \frac{6}{\epsilon_n^2} \frac{10}{9\sqrt{3}}\frac{\epsilon_n^4}{4}= \frac{5}{9} \epsilon_n^2\sqrt{3}.\label{morceau1}
\end{align}

Let us now compute the fourth equality. We have:
\begin{equation}\int_{[021]} (x_1-\epsilon_n)\frac{\partial c}{\partial x_1}(x_1,x_2)\dd x_2\ \dd x_1 = \sum_{i=1}^3 \int_{C_i} (x_1-\epsilon_n)\frac{\partial c}{\partial x_1}(x_1,x_2)\dd x_2\ \dd x_1.\label{etape12}
 \end{equation}
By symmetry, the integrals on the cells $C_1$ and $C_2$ are the same and:
\begin{align}
\int_{C_1}(x_1-\epsilon_n)\frac{\partial c}{\partial x_1}(x_1,x_2) \dd x_2\, \dd x_1 = & \int_0^{\epsilon_n/2} \int_0^{x_1\sqrt{3}}
 \frac{2x_1}{\epsilon_n^2}(x_1-\epsilon_n) \dd x_2 \, \dd x_1 \nonumber\\
 & + \int_{\epsilon_n/2}^{\epsilon_n} \int_0^{\epsilon_n \sqrt{3}/2-(x_1-\epsilon_n/2)/\sqrt{3}}
 \frac{2x_1}{\epsilon_n^2}(x_1-\epsilon_n) \dd x_2 \, \dd x_1\nonumber\\
 = & - \frac{1}{8} \epsilon_n^2 \sqrt{3}.
 \end{align}For the cell $C_3$:
 \begin{multline}
 \int_{C_3}(x_1-\epsilon_n)\frac{\partial c}{\partial x_1}(x_1,x_2) \dd x_2\, \dd x_1 =  \frac{2}{\epsilon_n^2} \int_{\epsilon_n /\sqrt{3}}^{\epsilon_n \sqrt{3}/2} \int_{\epsilon_n-\sqrt{3}(x_2-\epsilon_n/\sqrt{3})}^{\epsilon_n + \sqrt{3} (x_2-\epsilon_n/\sqrt{3})}(x_1-\epsilon_n)^2 \dd x_1 \,\dd x_2\\
\begin{aligned}
& \hspace{1.8cm}+ \frac{2}{\epsilon_n^2} \int_{\epsilon_n \sqrt{3}/2}^{\epsilon_n \sqrt{3}} \int_{\epsilon_n-\frac{1}{\sqrt{3}}(\epsilon_n \sqrt{3}-x_2)}^{\epsilon_n + \frac{1}{\sqrt{3}} (\epsilon_n \sqrt{3}-x_2)}(x_1-\epsilon_n)^2 \dd x_1 \,\dd x_2\\
 = & \frac{2}{\epsilon_n^2} \int_{\epsilon_n /\sqrt{3}}^{\epsilon_n \sqrt{3}/2} \frac{2}{3} \times 3\sqrt{3}\big(x_2-\frac{\epsilon_n \sqrt{3}}{3}\big)^3 \dd x_2+ \frac{2}{\epsilon_n^2} \int_{\epsilon_n \sqrt{3}/2}^{\epsilon_n \sqrt{3}} \frac{2}{3}  \Big(\frac{1}{\sqrt{3}}\big(\epsilon_n \sqrt{3} -x_2\big)\Big)^3\\
 = & \frac{\sqrt{3}}{\epsilon_n^2} \big(\frac{\epsilon_n \sqrt{3}}{2}-\frac{\epsilon_n \sqrt{3}}{3}\big)^4 + \frac{1}{9  \sqrt{3} \epsilon_n^2} \frac{9\epsilon_n^4 }{2^4}\\
 = & \Big( \frac{1}{ 9 \times 2^4} + \frac{1}{3 \times 2^4} \Big) \epsilon_n^2 \sqrt{3}= \frac{1}{36} \epsilon_n^2 \sqrt{3}.\label{morceau2}
 \end{aligned}
 \end{multline}
Gathering \eqref{etape12}, \eqref{morceau1} and \eqref{morceau2},
\begin{equation}
 \int_{[021]} (x_1-\epsilon_n)\frac{\partial c}{\partial x_1}(x_1,x_2)\dd x_2\ \dd x_1 = - \frac{2}{9}\epsilon_n^2 \sqrt{3}. \label{etape13}
\end{equation}The second, third and fifth equalities of \eqref{morceau_final} stem from the fact that $G$ is the barycenter of the triangle.
\end{proof}

\section{Proof of Corollary \ref{cor:somme_taucarres}}\label{app:lemH1}

We can now prove Corollary \ref{cor:somme_taucarres} that is useful for the computation of the generator in Section \ref{section:fin_generateur}:

\begin{proof}[Proof of Corollary \ref{cor:somme_taucarres}]The first equality is just the definition of $* \dif \phi$, see \eqref{def:difphi}. \\

Let us consider the second equality. For a cycle $\sigma \in \Cf_{1,4}$, let us consider the term
    \begin{align*}
    \epsilon_n^{-2} \sum_{\tau\in \s^n_2} \langle \partial_2 \tau,\phi\rangle^2 w(\sigma,\partial_2\tau) = & 
     \epsilon_n^{-2}   \sum_{e\in \sigma} \big(\langle \partial_2 \tau^+,\phi\rangle^2 + \langle \partial_2\tau^-,\phi\rangle^2 \big)\\
     = & 2\epsilon_n^2 \sum_{e\in \sigma} \langle \partial_2\tau^+,\phi\rangle^2
     \end{align*}where $\tau^+$ and $\tau^-$ are the triangles adjacent to the edge $e\in \sigma$. Without loss of generality, consider $e=[01]$ in the notation of Fig. \ref{fig:torus}. We deduce from \eqref{eq:prelim_tauphi} that:
     \begin{align*}
         \langle \partial_2 \tau^+,\phi\rangle^2 = \frac{3}{2} \epsilon_n^3 \int_0^{2\epsilon_n} \big(\phi^2_1(x_1,0)-\phi_2^1(x_1,0)\big)^2 \dif x_1 + O(\epsilon^5_n).
     \end{align*}As a consequence:
     \begin{align*}
       \epsilon_n^{-2} \sum_{\tau\in \s^n_2} \langle \partial_2 \tau,\phi\rangle^2 w(\sigma,\partial_2\tau)  = &      2 \epsilon_n^{-2}   \sum_{e\in \sigma} \Big[\frac{3}{2} \epsilon_n^3 \int_e \big(\phi^2_1-\phi_2^1\big)^2 \dif e + O(\epsilon_n^5)\Big]\nonumber\\
     = &   3 \epsilon_n \sum_{e\in \sigma}   \int_e \big(\phi^2_1-\phi_2^1\big)^2 \dif e + O(\epsilon_n^2 \|\sigma\|_{\Cf_1}).
    \end{align*}The first term can be interpreted as the integral of $e\mapsto \int_e \big(\phi^2_1-\phi_2^1\big) \dif e $ with respect to the measure 
    $\epsilon_n \sum_{e\in \sigma} \delta_{e}$ that puts a weight $\epsilon_n$ on each edge of the cycle $\sigma$. Under Conjecture \ref{conjecture-moments-discrets}, we can conclude when we consider the cycle $X^{(n)}_u$ at a time $u$ instead of $\sigma$. This proves the second equality.\\

The last equality is a consequence of Proposition \ref{prop:convmartingale-fin}.
\end{proof}


{\footnotesize
\providecommand{\noopsort}[1]{}\providecommand{\noopsort}[1]{}\providecommand{\noopsort}[1]{}\providecommand{\noopsort}[1]{}\providecommand{\noopsort}[1]{}\providecommand{\noopsort}[1]{}\providecommand{\noopsort}[1]{}\providecommand{\noopsort}[1]{}
}
\end{document}